\documentclass[11pt]{amsart}

\usepackage{mathrsfs}
\usepackage{amsfonts}
\usepackage{amssymb}
\usepackage{amsxtra}
\usepackage{dsfont}
\usepackage[dvipsnames]{xcolor}
\usepackage[english,polish]{babel}
\usepackage[shortlabels]{enumitem}
\usepackage{todonotes}

\usepackage{tikz}
\usetikzlibrary{arrows}

\usepackage[compress, sort]{cite}

\usepackage[T1]{fontenc}

\usepackage{newtxmath}
\usepackage[type1]{newtxtext}

\usepackage[margin=2.5cm, centering]{geometry}
\usepackage[colorlinks,citecolor=blue,urlcolor=blue,bookmarks=true]{hyperref}
\hypersetup{
pdfpagemode=UseNone,
pdfstartview=FitH,
pdfdisplaydoctitle=true,
pdfborder={0 0 0}, % No link borders
pdftitle={Asymptotic zeros' distribution of orthogonal polynomials with unbounded recurrence coefficients},
pdfauthor={Grzegorz Świderski and Bartosz Trojan},
pdflang=en-US
}

\newcommand{\CC}{\mathbb{C}}
\newcommand{\ZZ}{\mathbb{Z}}

\newcommand{\NN}{\mathbb{N}}
\newcommand{\RR}{\mathbb{R}}

\newcommand{\calR}{\mathcal{R}}

\newcommand{\calC}{\mathcal{C}}

\newcommand{\calO}{\mathcal{O}}

\newcommand{\calA}{\mathcal{A}}

\newcommand{\calY}{\mathcal{Y}}
\newcommand{\calD}{\mathcal{D}}

\newcommand{\frakB}{\mathfrak{B}}
\newcommand{\frakX}{\mathfrak{X}}
\newcommand{\frakA}{\mathfrak{A}}

\newcommand{\fraks}{\mathfrak{s}}
\newcommand{\frakr}{\mathfrak{r}}

\newcommand{\fraku}{\mathfrak{u}}

\newcommand{\frakt}{\mathfrak{t}}
\newcommand{\scrD}{\mathscr{D}}

\newcommand{\vphi}{\varphi}

\newcommand{\Id}{\operatorname{Id}}

\newcommand{\Arg}{\operatorname{Arg}}

\newcommand{\sign}[1]{\operatorname{sign}({#1})}

\newcommand{\pl}[1]{\foreignlanguage{polish}{#1}}

\newcommand{\abs}[1]{\lvert {#1} \rvert}

\newcommand{\tr}{\operatorname{tr}}

\DeclareMathOperator{\intr}{int}
\DeclareMathOperator{\Wrk}{Wr}
\DeclareMathOperator{\lin}{span}
\DeclareMathOperator{\Log}{Log}

\newcommand{\cl}[1]{\operatorname{cl}({#1})}

\newcommand{\GL}{\operatorname{GL}}
\newcommand{\SL}{\operatorname{SL}}
\newcommand{\discr}{\operatorname{discr}}
\newcommand{\diag}{\operatorname{diag}}

\newcommand{\Mat}{\operatorname{Mat}}

\newcommand{\sigmaEss}{\sigma_{\mathrm{ess}}}
\newcommand{\sigmaP}{\sigma_{\mathrm{p}}}

\newcommand{\sigmaAC}{\sigma_{\mathrm{ac}}}
\newcommand{\sigmaS}{\sigma_{\mathrm{sing}}}

\newcommand{\ind}[1]{{\mathds{1}_{{#1}}}}

\newcommand{\ud}{{\: \rm d}}
\newcommand{\ue}{\textrm{e}}
\newcommand{\supp}{\operatornamewithlimits{supp}}

\newcommand{\restr}[2]{#1 \!\! \restriction_{#2}}

\newtheorem{theorem}{Theorem}[section]

\newtheorem{proposition}[theorem]{Proposition}
\newtheorem{lemma}[theorem]{Lemma}
\newtheorem{corollary}[theorem]{Corollary}
\newtheorem{claim}[theorem]{Claim}

\theoremstyle{plain}
\newcounter{thm}

\newtheorem{main_theorem}[thm]{Theorem}

\numberwithin{equation}{section}

\theoremstyle{definition}
\newtheorem{example}[theorem]{Example}
\newtheorem{remark}[theorem]{Remark}

\title[Asymptotic zeros' distribution of orthogonal polynomials]
{Asymptotic zeros' distribution of orthogonal polynomials with unbounded recurrence coefficients}

\author{Grzegorz \'{S}widerski}
\address{
	Grzegorz \'{S}widerski\\
	Institute of Mathematics\\
	Polish Academy of Sciences\\
	ul. Śniadeckich 8\\
	00-696 Warsaw, Poland
}
\email{gswiderski@impan.pl}

\author{Bartosz Trojan}
\address{
	\pl{
		Bartosz Trojan\\
		Wydzia\l{} Matematyki,
		Politechnika Wroc\l{}awska\\
		Wyb. Wyspia\'{n}skiego 27\\
		50-370 Wroc\l{}aw\\
		Poland}
}
\email{bartosz.trojan@gmail.com}


\begin{document}
\selectlanguage{english}

\begin{abstract}
We study spectrum of finite truncations of unbounded Jacobi matrices with periodically modulated entries. 
In particular, we show that under some hypotheses a sequence of properly normalized eigenvalue counting measures converge 
vaguely to an explicit infinite Radon measure. To do so we link the asymptotic behavior of the 
Christoffel--Darboux kernel on the diagonal with the limiting measure. Finally, we derive strong asymptotics of 
the associated orthogonal polynomials in the complex plane, which allows us to prove that Cauchy transforms of 
the normalized eigenvalue counting measures converge pointwise and which leads to a stronger notion of convergence.
\end{abstract}

\maketitle

\section{Introduction} \label{sec:1}

Consider two sequences $a=(a_n : n \in \NN_0)$ and $b = (b_n : n \in \NN_0)$ of positive and real numbers, respectively. 
The infinite tridiagonal symmetric matrix
\[
	\calA = 
    \begin{pmatrix}
        b_0	& a_0	&		&       \\
        a_0 & b_1	& a_1	&         \\
            & a_1	& b_2 	& \ddots   \\
            &		&		\ddots &  \ddots 
    \end{pmatrix}
\]
is called \emph{Jacobi matrix} and the sequences $a,b$ are called its \emph{Jacobi parameters}. By considering complex valued sequences as column vectors, one can define an operator acting on sequences having finite support by formal multiplication by the matrix $\calA$. Let the \emph{Jacobi operator} $A$ be its closure in $\ell^{2}(\NN_0)$, the Hilbert space of square summable complex valued sequences with the scalar product
\[
	\langle x,y \rangle_{\ell^2(\NN_0)} =
	\sum_{n=0}^\infty x_n \overline{y_n}.
\]
Notice that $A$ is a (possibly unbounded) symmetric operator $A$. It is bounded if and only if its Jacobi parameters are
bounded, i.e.
\begin{equation} 
	\label{eq:int:2}
	\sup_{n \in \NN_0} a_n < \infty \quad \text{and} \quad
	\sup_{n \in \NN_0} |b_n| < \infty.
\end{equation}
A sufficient condition for its self-adjointness is the so-called \emph{Carleman's condition}
\begin{equation} 
	\label{eq:1}
	\sum_{n=0}^\infty \frac{1}{a_n} = \infty.
\end{equation}

The interest in Jacobi operators comes from their intimate connection to the classical moment problem and the theory of 
orthogonal polynomials on the real line, see e.g. \cite{Schmudgen2017}. As every self-adjoint operator having a cyclic vector
is unitary equivalent to a (possibly finite dimensional) Jacobi operator, the Jacobi operators are basic building blocks of 
all self-adjoint operators. Some types of Jacobi operators are related to random walks and birth-death processes, see e.g. 
\cite{Karlin1957, Karlin1959}.

This article is concerned with analysis of finite truncations of $\calA$. More precisely, for any $n \geq 1$ we consider
\[
	\calA_n = 
	 \begin{pmatrix}
     b_0 & a_0 &    &       &\ \\
     a_0 & b_1 & a_1 &        &  \\
        & a_1 & \ddots & \ddots     &  \\
        &    & \ddots & \ddots &  a_{n-2}\\
      &  &   & a_{n-2} & b_{n-1}
  \end{pmatrix}.
\]
The characteristic polynomial $P_n(z) = \det(z \Id - \calA_n)$ and its normalized variant
\[
	\begin{aligned}
		p_0(z) & = P_0(z) \equiv 1\\
		p_n(z) & = \frac{1}{a_0 a_1 \ldots a_{n-1}} P_n(z), \quad n \geq 1,
	\end{aligned}
\]
are crucial in the spectral analysis of $A$. In fact, the symmetric operator $A$ has always self-adjoint extensions acting on 
$\ell^2(\NN_0)$. Let $\breve{A}$ be a self-adjoint extension of $A$, and let $E_{\breve{A}}$ be its projection valued spectral 
measure. We define the Borel measure
\begin{equation} \label{eq:int:24}
	\breve{\mu}(\cdot) = \langle E_{\breve{A}}(\cdot) e_0, e_0 \rangle_{\ell^2(\NN_0)}
\end{equation}
where for any $n \in \NN_0$ the vector $e_n$ is the sequence having $1$ on $n$th position and $0$ elsewhere. Then the operator
$U : \ell^2(\NN_0) \to L^2(\breve{\mu})$ defined on the basis vectors by
\[
	U e_n = p_n,
\]
is a unitary isomorphism, see e.g. \cite[Section 6.1]{Schmudgen2017}. In particular, $(p_n : n \in \NN_0)$ is an orthonormal
basis of $L^2(\breve{\mu})$. Moreover,
\[
	(U A U^{-1} f)(x) = xf(x)
\]
for all $f \in L^2(\breve{\mu})$ such that $xf \in L^2(\breve{\mu})$. If $A$ is self-adjoint, then the only self-adjoint
extension of $A$ is $A$ itself. In such a case $\mu := \breve{\mu}$ is the unique measure such that $(p_n : n \in \NN_0)$
are orthonormal in $L^2(\mu)$, see e.g. \cite[Theorem 6.10]{Schmudgen2017}.

In order to describe asymptotic of $\sigma(\calA_n) = p_n^{-1}(0)$ one defines the \emph{zero counting measure}
\begin{equation} 
	\label{eq:int:3}
	\nu_n = \sum_{y \in \sigma(\calA_{n+1})} \delta_y
\end{equation}
where $\delta_y$ is the Dirac measure at $y$. Then one is interested in the convergence of the sequence of probability measures 
$\big( \tfrac{1}{n+1} \nu_n : n \in \NN_0 \big)$ in some sense. Let us recall that a sequence of positive Borel measures $(\omega_n : n \in \NN_0)$
supported on some open set $U \subset \RR$ \emph{converges vaguely} (on $U$) to $\omega$, $\omega_n \xrightarrow{v} \omega$, 
if we have convergence of integrals 
\begin{equation} \label{eq:int:1}
	\lim_{n \to \infty} \int_\RR f \ud \omega_n =
	\int_\RR f \ud \omega, \quad f \in \calC_c(U).
\end{equation}
Here, $\calC_c(U)$ denotes the space of compactly supported continuous functions on $U$. If the convergence holds for any 
$f \in \calC_b(U)$, a continuous bounded function on $U$, then $(\omega_n : n \in \NN)$ \emph{converges weakly} (on $U$) to 
$\omega$, $\omega_n \xrightarrow{w} \omega$. If the set $U$ is not clear from the context, then we shall understand that 
$U = \RR$.

The case when the operator $A$ is bounded, which is equivalent to the measure $\mu$ being compactly supported,
is thoroughly studied. Here $\big( \tfrac{1}{n+1} \nu_n : n \in \NN_0 \big)$ is a sequence of probability measures with support contained in the convex 
hull of $\sigma(A) = \supp \mu$ being a compact subset of $\RR$. Thus, by the weak sequential compactness, we get the existence
of a strictly increasing subsequence $(n_k : k \in \NN_0)$ and a probability measure $\nu$ such that
$\frac{1}{n_k + 1} \nu_{n_k} \xrightarrow{w} \nu$. Logarithmic potential theory provides a condition for the convergence of the whole
sequence as well as the description of the limiting measure $\nu$ (see e.g. the survey \cite{Simon2007} for details). 
Let us mention some results relevant to our paper. Nevai (see \cite[Theorem 3, p.50]{Nevai1979}) considered the case when
\[
	\lim_{n \to \infty} a_n = \frac{1}{2} \quad \text{and} \quad
	\lim_{n \to \infty} b_n = 0.
\]
Then $\sigmaEss(A) = [-1,1]$ and
\begin{equation} \label{eq:int:11}
	\frac{1}{n+1} \nu_n \xrightarrow{w} 
	\frac{1}{\pi} \frac{1}{\sqrt{1 - x^2}} \mathds{1}_{(-1,1)}(x) \ud x,
\end{equation}
where by $\mathds{1}_X$ we denote the indicator function of the set $X \subset \RR$.
Later, Geronimo and Van Assche in \cite[Theorem 11]{VanAssche1988a}, studied a more general case of Jacobi parameters 
such that for certain $N$-periodic sequences $(\alpha_n : n \in \NN_0), (\beta_n : n \in \NN_0)$ of positive and real numbers, 
respectively, it holds
\[
	\lim_{n \to \infty} |a_n -\alpha_n| = 0 \quad \text{and} \quad
	\lim_{n \to \infty} |b_n -\beta_n| = 0.
\]
For each $n \geq 1$, we set
\begin{equation} 
	\label{eq:int:12}
	\frakX_n(x) = 
	\frakB_{n+N-1}(x) \ldots \frakB_{n+1}(x) \frakB_{n}(x) \quad \text{where} \quad
	\frakB_j(x) = 
	\begin{pmatrix}
		0 & 1 \\
		-\frac{\alpha_{j-1}}{\alpha_j} & \frac{x-\beta_j}{\alpha_j}
	\end{pmatrix}, \quad j \geq 1.
\end{equation}
Then(\footnote{For any set $X \subset \RR$ we denote by $\cl{X}$ its closure.}) $\sigmaEss(A) = \cl{\Lambda}$ where 
$\Lambda := (\tr \frakX_N)^{-1}[(-2,2)]$. It turns out that $\Lambda = \bigcup_{j=1}^N I_j$ where $I_j$ are non-empty open 
disjoint intervals whose closures might touch each other. Moreover(\footnote{For any $X \in \Mat(2,\CC)$, a $2 \times 2$ 
matrix with complex coefficients, its \emph{discriminant} is defined by $\discr X = (\tr X)^2 - 4 \det X$.}), 
\begin{equation} \label{eq:int:14}
	\frac{1}{n+1} \nu_n \xrightarrow{w} 
	\frac{1}{\pi N} \frac{|\tr \frakX_N'(x)|}{\sqrt{-\discr \frakX_N(x)}} \mathds{1}_{\Lambda}(x) \ud x,
\end{equation}
Notice that the formula \eqref{eq:int:14} for $N=1$ and $\alpha_n \equiv \frac{1}{2}, \beta_n \equiv 0$ 
reduces to \eqref{eq:int:11}.

The situation when $A$ is unbounded, or equivalently when \eqref{eq:int:2} is violated, is different. A well-known phenomenon 
is the fact that part of the points from $\sigma(\calA_n)$ is repelled to infinity. So if we consider as previously the sequence
$\big( \tfrac{1}{n+1} \nu_n : n \in \NN_0 \big)$, then the subsequential limits might be only sub-probability measures. 
A popular way to deal with this problem is to consider
\begin{equation} \label{eq:int:5}
	\tilde{\nu}_n = 
	\sum_{y \in \sigma(\calA_{n+1})} \delta_{\frac{y}{c_{n}}}
\end{equation}
for certain positive sequence $(c_n : n \in \NN_0)$ tending to infinity, see e.g. \cite[Section 4]{VanAssche1990}. The r\^ole of 
this sequence is to assure that the supports of $\tilde{\nu}_n$ lie in a compact subset of $\RR$, so every subsequential limit of 
$\big( \tfrac{1}{n+1} \tilde{\nu}_n : n \in \NN_0 \big)$ is a probability measure. In this setup, under some conditions on the measure $\mu$, 
logarithmic potential theory with external fields gives weak convergence of $\big( \tfrac{1}{n+1} \tilde{\nu}_n : n \in \NN_0 \big)$, 
and describes its limit, see e.g. \cite[Theorem 4]{Rakhmanov1984} or \cite[Chapter VII.1]{Saff1997} for details. 

Another approach was proposed by Van Assche in \cite{VanAssche1985} where for any purely absolutely continuous measure $\mu$ 
with density $\mu'$ satisfying $\log \mu' \in L^1_{\mathrm{loc}}(\RR)$ and such that for some $\lambda > 1$,
\[
	\lim_{|x| \to \infty} \frac{-\log \mu'(x)}{|x|^\lambda} = 1,
\]
the following sequence of measures 
\[
	\tilde{\nu}_n =
	\sum_{y \in \sigma(\calA_n)} \frac{\delta_y}{1+y^2}
\]
was studied. It was shown there that $\tfrac{1}{\rho_n} \tilde{\nu}_n \xrightarrow{w} \frac{\ud x}{1+x^2}$, where
$\rho_n = c_\lambda n^{1-\frac{1}{\lambda}}$ for some explicit constant $c_\lambda>0$. Notice that $\rho_n/n \to 0$.
Consequently, for any $f \in \calC_b(\RR)$ satisfying $\sup_{x \in \RR} (1+x^2) |f(x)| < \infty$ we have
\[
	\lim_{n \to \infty} \frac{1}{\rho_n} \int_\RR f(x) \ud \nu_n(x) =
	\int_\RR f(x) \ud x.
\]
In particular, $\tfrac{1}{\rho_n} \nu_n \xrightarrow{v} \ud x$. Let us emphasize that in this case the limiting measure 
is locally finite only. In the proof asymptotics of $\log |p_n(z)|$ for any non-real $z$ obtained in \cite[Theorem 1]{Rakhmanov1984} 
by Rakhmanov (see also \cite{Buyarov1991}) was instrumental.

In this article we shall start a systematic study of the sequences \eqref{eq:int:3} for \emph{unbounded} Jacobi operators, i.e. 
for which the condition \eqref{eq:int:2} is violated. To the best of our knowledge it is the first general study of spectra 
of truncations of unbounded $A$ in terms of its Jacobi parameters which is not of the form \eqref{eq:int:5}. 
Actually, we shall consider a large class of Jacobi matrices with periodically modulated Jacobi parameters. This class has been 
introduced in \cite{JanasNaboko2002}, and it is studied extensively ever since (see e.g. \cite{jordan2} and the references
therein). To be more precise: Jacobi parameters $(a_n : n \in \NN_0), (b_n : n \in \NN_0)$ are \emph{$N$-periodically modulated} 
if for some $N$-periodic sequences $(\alpha_n : n \in \NN_0), (\beta_n : n \in \NN_0)$ of positive and real numbers, 
respectively, it holds
\begin{equation} \label{eq:int:13}
	\lim_{n \to \infty} \bigg| \frac{a_{n-1}}{a_n} - \frac{\alpha_{n-1}}{\alpha_n} \bigg| = 0, \quad
	\lim_{n \to \infty} \bigg| \frac{b_{n}}{a_n} - \frac{\beta_{n}}{\alpha_n} \bigg| = 0, \quad
	\lim_{n \to \infty} a_n = \infty.
\end{equation}
It turns out that spectral properties of Jacobi operators with periodically modulated Jacobi parameters crucially depend on 
the properties of the matrix $\frakX_N(0)$ where $\frakX_N(x)$ is defined in \eqref{eq:int:12}, see Section~\ref{sec:2.5} for 
details. In our results below we shall use the notion of the Stolz class $\calD_1^N$. Let us recall that a sequence
$(x_n : n \in \NN)$ belongs to $\calD_1^N$ if
\[
	\sum_{n=1}^\infty |x_{n+N} - x_n| < \infty.
\]
\begin{main_theorem} \label{thm:A:1}
	Suppose that a Jacobi matrix $\calA$ has $N$-periodically modulated Jacobi parameters satisfying \eqref{eq:int:13}. Suppose 
	that $\tr \frakX_N(0) \in (-2,2)$ where $\frakX_N(x)$ is defined in \eqref{eq:int:12} and 
	\begin{equation} \label{eq:int:18}
		\bigg( \frac{a_{n-1}}{a_n} \bigg), 
		\bigg( \frac{b_n}{a_n} \bigg), 
		\bigg( \frac{1}{a_n} \bigg) \in \calD_1^N.
	\end{equation}
	Set
	\[
		\rho_n = \sum_{k=0}^{n} \frac{\alpha_k}{a_k}.
	\]
	If the Carleman's condition \eqref{eq:1} is satisfied, then for the sequence $(\nu_n : n \in \NN_0)$ defined by 
	\eqref{eq:int:3} and for any $f \in \calC_b(\RR)$ satisfying $\sup_{x \in \RR} (1+x^2) |f(x)| < \infty$ we have
	\[
		\lim_{n \to \infty} \frac{1}{\rho_n} \int_{\RR} f \ud \nu_n =
		\int_{\RR} f \ud \nu
	\]
	where the measure $\nu$ is purely absolutely continuous with the density
	\begin{equation} \label{eq:int:15}
		\frac{\ud \nu}{\ud x} \equiv \frac{1}{\pi N} \frac{|\tr \frakX_N'(0)|}{\sqrt{-\discr \frakX_N(0)}}.
	\end{equation}
\end{main_theorem}
Notice that \eqref{eq:int:15} says that $\nu$ is just the Lebesgue measure on $\RR$ multiplied by the value at $x=0$ 
of the density of the measure appearing on the right-hand side of \eqref{eq:int:14}. If in Theorem~\ref{thm:A:1} we assume 
that the Carleman's condition is violated, then $A$ is not self-adjoint. Thus, by Remark~\ref{rem:5}, we have
$\tfrac{1}{\rho_n} \nu_n \xrightarrow{v} 0$ for every sequence $\rho_n \to \infty$. A similar phenomenon occurs in the case 
when $\tr \frakX_N(0) \in \RR \setminus [-2,2]$.
\begin{main_theorem} \label{thm:A:2}
	Suppose that a Jacobi matrix $\calA$ has $N$-periodically modulated Jacobi parameters satisfying \eqref{eq:int:13}. 
	Suppose that $\tr \frakX_N(0) \in \RR \setminus [-2,2]$ where $\frakX_N(x)$ is defined in \eqref{eq:int:12} and
	\begin{equation} \label{eq:int:19}
		\bigg( \frac{a_{n-1}}{a_n} \bigg), 
		\bigg( \frac{b_n}{a_n} \bigg), 
		\bigg( \frac{1}{a_n} \bigg) \in \calD_1^N.
	\end{equation}
	Take any positive sequence $(\rho_n)$ tending to infinity and define $\nu_n$ by \eqref{eq:int:3}. Then
	$\tfrac{1}{\rho_n} \nu_n \xrightarrow{v} 0$.
\end{main_theorem}
In fact Theorem \ref{thm:A:2} is a consequence of Proposition~\ref{prop:12} which states that for every interval $(a,b)$ 
that is disjoint from the essential spectrum of some $\breve{A}$ we have(\footnote{For any Borel measure $\omega$ and any open 
subset $U \subset \RR$ we define a measure $\restr{\omega}{U}$ on Borel subsets of $U$ by $\restr{\omega}{U}(\cdot) = 
\omega(\cdot \cap U)$.}) $\tfrac{1}{\rho_n} \restr{\nu_n}{(a,b)} \xrightarrow{v} 0$ for every sequence $\rho_n \to \infty$. 
Since \eqref{eq:int:19} implies that $A$ is self-adjoint and $\sigmaEss(A) = \emptyset$ the result easily follows, see 
Example~\ref{ex:III} for more details.

The case when $\tr \frakX_N(0) \in \{-2,2\}$ turns out to be quite subtle. Notice that $\discr \frakX_N(0) = 0$.
Our answer depends strongly on the diagonalizability of the matrix $\frakX_N(0)$. To describe our results we need the concept 
of transfer matrices. More precisely, the sequence of $N$-step transfer matrices is defined by the requirement
\begin{equation} \label{eq:int:25}
	\begin{pmatrix}
		p_{n+N-1}(x) \\
		p_{n+N}(x)
	\end{pmatrix}
	=
	X_n(x)
	\begin{pmatrix}
		p_{n-1}(x) \\
		p_n(x)
	\end{pmatrix}, \quad n \geq 1, x \in \CC.
\end{equation}
It can be expressed by the formula
\begin{equation} \label{eq:int:26}
	X_n(x) = B_{n+N-1}(x) \ldots B_{n+1}(x) B_n(x) \quad \text{where} \quad 
	B_j(x) = 
	\begin{pmatrix}
		0 & 1 \\
		-\frac{a_{j-1}}{a_j} & \frac{x-b_j}{a_j}
	\end{pmatrix}, \quad j \geq 1.
\end{equation}

\begin{main_theorem} \label{thm:A:3a}
	Suppose that a Jacobi matrix $\calA$ has $N$-periodically modulated Jacobi parameters satisfying \eqref{eq:int:13}. 
	Suppose that $\frakX_N(0) = \varepsilon \Id$ for some $\varepsilon \in \{-1,1\}$ where $\frakX_N(x)$ is defined in 
	\eqref{eq:int:12}, and
	\[
		\bigg( a_n \Big( \frac{a_{n-1}}{a_n} - \frac{\alpha_{n-1}}{\alpha_n} \Big) \bigg), 
		\bigg( a_n \Big( \frac{b_n}{a_n} - \frac{\beta_n}{\alpha_n} \Big) \bigg), 
		\bigg( \frac{1}{a_n} \bigg) \in \calD_1^N.
	\]
	Then the limit 
	\[
		h(z) = \lim_{j \to \infty} a_{jN+N-1}^2 \discr X_{jN}(z), \quad z \in \CC,
	\]
	exists and defines a polynomial of degree $2$ with negative leading coefficient. Set $\Lambda = \{ x \in \RR : h(x) < 0 \}$
	and 
	\[
		\rho_n = \sum_{k=0}^{n} \frac{\alpha_k}{a_k}.
	\]
	Then for the sequence $(\nu_n : n \in \NN_0)$ defined by \eqref{eq:int:3} and for any $f \in \calC_b(\RR)$ satisfying 
	$\sup_{x \in \RR} (1+x^2) |f(x)| < \infty$ we have
	\[
		\lim_{n \to \infty} \frac{1}{\rho_n} \int_{\RR} f \ud \nu_n =
		\int_{\RR} f \ud \nu
	\]
	where the measure $\nu$ is purely absolutely continuous with the density
	\[
		\frac{\ud \nu}{\ud x} = \frac{1}{4 \pi N \alpha_{N-1}} \frac{|h'(x)|}{\sqrt{-h(x)}} \mathds{1}_\Lambda(x).
	\]
\end{main_theorem}

Finally, the following result covers the case of non-diagonalizable $\frakX_N(0)$. Let us introduce an auxiliary positive
sequence $\gamma = (\gamma_n : n \in \NN_0)$ tending to infinity. We say that $N$-periodically modulated Jacobi parameters
$(a_n : n \in \NN_0)$ and $(b_n : n \in \NN_0)$ are $\gamma$-tempered if the sequences
\[
	\bigg(\sqrt{\gamma_n} \Big(\frac{\alpha_{n-1}}{\alpha_n} - \frac{a_{n-1}}{a_n} \Big) : n \in \NN \bigg),
	\bigg(\sqrt{\gamma_n} \Big(\frac{\beta_n}{\alpha_n} - \frac{b_n}{a_n}\Big) : n \in \NN \bigg),
	\bigg(\frac{\gamma_n}{a_n} : n \in \NN\bigg)
\]
belongs to $\calD_N^1$. About the sequence $\gamma$ we assume that
\begin{equation}
	\label{eq:83a}
	\bigg(\sqrt{\gamma_n} \Big(\sqrt{\frac{\alpha_{n-1}}{\alpha_n}} - \sqrt{\frac{\gamma_{n-1}}{\gamma_n}}\Big) : n \in \NN\bigg),
	\bigg(\frac{1}{\sqrt{\gamma_n}} : n \in \NN\bigg) \in \calD^1_N,
\end{equation}
and
\begin{equation}
	\label{eq:83b}
	\lim_{n \to \infty} \big(\sqrt{\gamma_{n+N}} - \sqrt{\gamma_n} \big) = 0.
\end{equation}
We also impose that
\begin{equation}
	\label{eq:86}
	\bigg(
	\gamma_n\big(1- \varepsilon [\frakX_n(0)]_{11}\big)\Big(\frac{\alpha_{n-1}}{\alpha_n} - \frac{a_{n-1}}{a_n}\Big)
	-
	\gamma_n \varepsilon \Big(\frac{\beta_n}{\alpha_n} - \frac{b_n}{a_n}\Big)
	:  n \in \NN \bigg) \in \calD_1^N
\end{equation}
for certain $\varepsilon \in \{-1, 1\}$.
\begin{main_theorem} \label{thm:A:3b}
	Let $N$ be a positive integer. Let $\gamma = (\gamma_n : n \in \NN)$ be a sequence of positive numbers tending to infinity 
	and satisfying \eqref{eq:83a} and \eqref{eq:83b}. Let $(a_n : n \in \NN_0)$ and $(b_n : n \in \NN_0)$ be $\gamma$-tempered
	$N$-periodically modulated Jacobi parameters such that $\frakX_0(0)$ is a non-trivial parabolic element. 
	Suppose that \eqref{eq:86} holds true with $\varepsilon = \sign{\tr\frakX_0(0)}$. Then the limit 
	\[
		h(z) = \lim_{j \to \infty} \gamma_{jN+N-1} \discr X_n(z), \quad z \in \CC,
	\]
	exists and defines a polynomial of degree at most one. Let $\Lambda = \{ x \in \RR : h(x) < 0 \}$ and 
	\begin{equation} \label{eq:202} 
		\rho_n = \sum_{k=0}^{n} \frac{\sqrt{\alpha_k \gamma_k}}{a_k}.
	\end{equation}
	If $\rho_n \to \infty$ and $\Lambda \neq \emptyset$,
	then for the sequence $(\nu_n : n \in \NN_0)$ defined by \eqref{eq:int:3} and for any $f \in \calC_b(\RR)$ satisfying
	$\sup_{x \in \RR} (1+x^2) |f(x)| < \infty$ we have
	\[
		\lim_{n \to \infty} \frac{1}{\rho_n} \int_{\RR} f \ud \nu_n =
		\int_{\RR} f \ud \nu
	\]
	where the measure $\nu$ is purely absolutely continuous with the density
	\[
		\frac{\ud \nu}{\ud x} =
		\frac{\sqrt{\alpha_{N-1}}}{\pi N} \frac{|\tr \frakX_0'(0)|}{\sqrt{-h(x)}} \ind{\Lambda}(x).
	\]
\end{main_theorem}
If in Theorem~\ref{thm:A:3b} one has $\Lambda \neq \emptyset$ but the sequence \eqref{eq:202} is bounded, then $A$ is not self-adjoint 
(see \cite[Theorem B]{jordan2}). So then $\tfrac{1}{\rho_n} \nu_n \xrightarrow{v} 0$ for any sequence $\rho_n \to \infty$. 
Further, if $\Lambda = \emptyset$ and $h$ is not a zero polynomial, then either $A$ is self-adjoint and $\sigmaEss(A) = \emptyset$ 
or $A$ is not self-adjoint (see \cite[Theorem A]{jordan2}). Thus, for any sequence $\rho_n \to \infty$ we have $\tfrac{1}{\rho_n} \nu_n \xrightarrow{v} 0$.

The study of the sequence $\big( \tfrac{1}{n+1} \nu_n : n \in \NN_0 \big)$ has natural motivation. In physics the limiting measure $\nu$ is called 
\emph{density of states}, see e.g. \cite[Chapter III]{Pastur1992}. It is usually more accessible for computations than 
the measure $\mu$ but still it is helpful in description of $\sigmaEss(A)$. In the theory of orthogonal polynomials a 
central object of study, which is closely related to $\nu$, is the \emph{Christoffel--Darboux kernel}, i.e.
\[
	K_n(x,y) = \sum_{k=0}^{n} p_k(x) \overline{p_k(y)}, \quad x,y \in \CC.
\]
The Christoffel--Darboux kernel plays an important r\^ole in the theory of polynomial least squares approximation 
(see e.g. the survey \cite{Nevai1986}), spectral theory (see e.g. the survey \cite{Simon2008}), approximation of the 
measure $\mu$ (see e.g. \cite{VanAssche1993, Totik2000}), universality in random matrix theory (see e.g. the survey 
\cite{Lubinsky2016}), and data analysis (see e.g. the monograph \cite{Lasserre2022}). To understand the weak behavior of 
$K_n$ on the diagonal one usually defines the sequence of positive measures
\begin{equation} \label{eq:int:4}
	\ud \eta_n(x) = K_n(x,x) \ud \breve{\mu}(x)
\end{equation}
where $\breve{\mu}$ is defined in \eqref{eq:int:24} for certain self-adjoint extension of $A$.
It turns out that the sequences \eqref{eq:int:3} and \eqref{eq:int:4} satisfy
\begin{equation} \label{eq:int:6}
	\lim_{n \to \infty}
	\bigg|
		\frac{1}{n+1}
		\int_\RR
		f \ud \nu_n -
		\frac{1}{n+1}
		\int_\RR
		f \ud \eta_n
	\bigg| = 0, \quad f \in \calC_0(\RR),
\end{equation}
see \cite[Theorem 2.1]{Geronimo1988}. Here $\calC_0(\RR)$ is the space of continuous functions tending to $0$ at both $-\infty$ and $+\infty$. Thus knowledge of weak limits of $\big( \tfrac{1}{n+1} \nu_n : n \in \NN_0 \big)$ allows us to understand 
the behavior of the Christoffel--Darboux kernel on the diagonal and vice-versa. Since $\big( \tfrac{1}{n+1} \nu_n : n \in \NN_0 \big)$ does not 
depend on the choice of $\breve{\mu}$, the weak behavior of $\big( \tfrac{1}{n+1} \eta_n : n \in \NN_0 \big)$ also does not depend on the choice 
of $\breve{\mu}$.

The following result, which is motivated by \eqref{eq:int:6}, allows to obtain local description of the limit of the 
sequence $\big( \tfrac{1}{\rho_n} \nu_n : n \in \NN_0 \big)$ under the assumption that we have a locally uniform convergence of 
the Christoffel--Darboux kernel.
\begin{main_theorem}[=Lemma~\ref{lem:1}] \label{thm:A}
	Suppose that $(\rho_n : n \in \NN)$ is a positive sequence tending to infinity. If there are an open $U \subset \RR$ 
	and a non-zero function $h : U \to [0, \infty)$ such that
	\[
		\lim_{n \to \infty} \frac{1}{\rho_n} K_n(x,x) = h(x)
	\]
	locally uniformly with respect to $x \in U$, then the operator $A$ is self-adjoint and
	$\tfrac{1}{\rho_n} \restr{\nu_n}{U} \xrightarrow{v} \nu$ where
	\[
		\ud \nu(x) = h(x) \ud \mu(x), \quad x \in U.
	\]
\end{main_theorem}
In Section~\ref{sec:3.2}, using \cite{PeriodicII, ChristoffelI, ChristoffelII, jordan}, we 
show how the result can be applied to a large class of Jacobi operators with periodically modulated Jacobi parameters. 
In particular, we can relax hypotheses \eqref{eq:int:18} to conclude $\tfrac{1}{\rho_n} \nu_n \xrightarrow{v} \nu$. Moreover, by exploiting
our knowledge of $\sigmaEss(A)$ and applying Proposition~\ref{prop:12}, under additional hypothesis $\lim_{n \to \infty} (a_{n+N}
- a_n) = 0$ in Theorem \ref{thm:A:3a} we obtain(\footnote{For any set $X \subset \RR$ by $\partial X$ we denote its boundary.})
$\tfrac{1}{\rho_n} \restr{\nu_n}{\RR \setminus \partial \Lambda} \xrightarrow{v} \nu$. Finally, under much more restrictive assumptions 
than \eqref{eq:83a}, \eqref{eq:83b} and \eqref{eq:86} we get 
$\tfrac{1}{\rho_n} \restr{\nu_n}{\RR \setminus \partial \Lambda} \xrightarrow{v} \nu$.

Since in Theorems~\ref{thm:A:1}, \ref{thm:A:3a} and \ref{thm:A:3b} we study a stronger notion of convergence, we need further tools.
Let us recall that for any finite positive measure $\omega$ its Cauchy transform $\calC[\omega]$, is defined by
\[
	\calC[\omega](z) = \int_\RR \frac{1}{x-z} \ud \omega(x), \quad z \in \CC \setminus \RR.
\]
It is a \emph{Herglotz function}, namely it is a holomorphic function which maps the upper half plane 
$\CC_+ = \{ z \in \CC : \Im z > 0\}$ into $\CC_+ \cup \RR$. The function $\calC[\omega]$ uniquely determines the measure, and under some 
hypotheses, for a sequence $(\omega_n : n \in \NN)$ of positive finite measures satisfying 
$\sup_{n \in \NN} \omega_n(\RR) < \infty$, the pointwise convergence of $\calC[\omega_n]$ implies 
$\omega_n \xrightarrow{w} \omega$ for some finite measure $\omega$, see \cite{Geronimo2003} for details. In Lemma~\ref{lem:3} 
we present a version of this result in the case when the condition $\sup_{n \in \NN} \omega_n(\RR) < \infty$ might be violated. 
The reasoning is based on the idea from the proof of \cite[Theorem 3.1]{VanAssche1985}. The lemma states that the pointwise 
convergence of $\calC[\omega_n]$ on a set with a limit point in $\CC_+$ implies existence of a finite measure $\tilde{\omega}$
such that
\begin{equation} \label{eq:int:28}
	\lim_{n \to \infty} \int_\RR f(x) \frac{\ud \omega_n(x)}{1+x^2} =
	\int_\RR f(x) \ud \tilde{\omega}(x), \quad f \in \calC_0(\RR).
\end{equation}
In particular, $\omega_n \xrightarrow{v} (1+x^2) \ud \tilde{\omega}(x)$. Moreover, one can describe the Cauchy transform of 
$\tilde{\omega}$ explicitly.

Our next result together with Lemma~\ref{lem:3} allows to extend \eqref{eq:int:6} from the case $\rho_n = n+1$ to the general
sequence $\rho_n \to \infty$. Its proof is inspired by \cite[Theorem 2.1]{Geronimo1988} and \cite[Theorem 1.5]{Simon2009}.
\begin{main_theorem}[=Theorem~\ref{thm:4}] \label{thm:C}
	Let $(\rho_n : n \in \NN_0)$ be a positive sequence tending to infinity. Let $(\nu_n : n \in \NN_0)$ and $(\eta_n : n \in \NN_0)$ be defined in 
	\eqref{eq:int:3} and \eqref{eq:int:4}, respectively. Then
	\begin{equation} \label{eq:int:32}
		\lim_{n \to \infty} 
		\big| 
		\calC \big[ \tfrac{1}{\rho_n} \nu_n \big](z) - 
		\calC \big[ \tfrac{1}{\rho_n} \eta_n \big](z) 
		\big| = 0
	\end{equation}
	locally uniformly with respect to $z \in \CC \setminus \RR$ and for any $L \geq 0$,
	\begin{equation} \label{eq:int:33}
		\lim_{n \to \infty} 
		\big| 
		\calC \big[ \tfrac{1}{\rho_{n}} \eta_{n+L} \big](z) - 
		\calC \big[ \tfrac{1}{\rho_n} \eta_n \big](z) 
		\big| = 0
	\end{equation}
	locally uniformly with respect to $z \in \CC \setminus \RR$.
\end{main_theorem}

Again, as in \eqref{eq:int:6}, the formula~\eqref{eq:int:32} implies that the weak behavior of the sequence 
$\big( \tfrac{1}{\rho_n} \eta_n : n \in \NN_0 \big)$ does not depend on the choice of $\breve{\mu}$. 
Hence, in order to understand the vague limit of the sequence $\big( \tfrac{1}{\rho_n} \eta_n : n \in \NN_0 \big)$ one can study Cauchy transforms 
of the sequence $\big( \tfrac{1}{\rho_n} \nu_n : n \in \NN_0 \big)$ and vice-versa. The choice to work with $\big( \tfrac{1}{\rho_n} \nu_n : n \in \NN_0 \big)$ 
might be advantageous if one understands the asymptotic behavior of orthogonal polynomials $(p_n : n \in \NN_0)$ because
\begin{equation} \label{eq:int:27}
	\calC \big[ \tfrac{1}{\rho_n} \nu_n \big](z) = -\frac{1}{\rho_n} \frac{p_{n+1}'(z)}{p_{n+1}(z)}, \quad z \in \CC \setminus \RR,
\end{equation}
see Remark~\ref{rem:2}. 

Motivated by \eqref{eq:int:27} we show that, in the setup of Theorems \ref{thm:A:1}, \ref{thm:A:3a} and \ref{thm:A:3b}, for any 
$i \in \{0, 1, \ldots, N-1\}$ and most compact $K \subset \CC \setminus \RR$ there are: $M \geq 1$, a continuous function $\varphi_i$, which is holomorphic in the interior of $K$, and a sequence of explicit continuous functions $\lambda_j : K \to \CC \setminus \{0\}$, which are holomorphic in the interior of $K$ such that
\begin{equation} \label{eq:int:29}
	\lim_{k \to \infty} \sup_{z \in K} 
	\bigg|
	\frac{p_{kN+i}(z)}{\prod_{j=M}^{k-1} \lambda_j(z)} - \varphi_i(z) 
	\bigg| = 0.
\end{equation}
Moreover, for any $z_0 \in K$ we have $|\varphi_{i}(z_0)| + |\varphi_{i+1}(z_0)| \neq 0$.
Since the asymptotics are uniform, given $z_0 \in \intr(K)$ such that $\varphi_i(z_0) \neq 0$ we can differentiate them in the neighborhood of $z_0$, which reduces the problem to the understanding of the limit
\[
	\lim_{k \to \infty} 
	\frac{1}{\rho_{kN+i}} 
	\sum_{j=M}^{k-1}	\frac{\lambda_j'(z)}{\lambda_j(z)}.
\]
Thanks to explicit form of $\lambda_j$ we can treat such limits via Stolz--Ces\`aro theorem. Having computed the limit of \eqref{eq:int:27} we can use Lemma~\ref{lem:3} to obtain the Cauchy transform of the measure $\tilde{\omega}$ such that \eqref{eq:int:28} holds for $\omega_n = \tfrac{1}{\rho_{nN+i}} \nu_{nN+i}$. Then by using Stieltjes inversion formula we identify the measure $\tilde{\omega}$. In our setup we can prove that $\rho_{nN}/\rho_{nN+i} \to 1$ for any $i \in \{0,1,\ldots,N-1\}$, so in view of \eqref{eq:int:32} and \eqref{eq:int:33} we have \eqref{eq:int:28} for the whole sequence $\omega_{n} = \tfrac{1}{\rho_n} \nu_n$.

By \eqref{eq:int:25} and \eqref{eq:int:26} orthogonal polynomials $(p_n(z) : n \in \NN_0)$ are solutions of the equation
\begin{equation} \label{eq:int:30}
	\begin{pmatrix}
		u_{n+N-1}(z) \\
		u_{n+N}(z)
	\end{pmatrix} 
	=
	X_n(z) 
	\begin{pmatrix}
		u_{n-1}(z) \\
		u_n(z)
	\end{pmatrix}.
\end{equation}
Our discrete Levinson's type theorem (see Theorem~\ref{thm:10}) allows us to understand asymptotic behavior of a basis of 
solutions of \eqref{eq:int:30} provided that we can properly diagonalize $X_{kN+i}(z)$ in a holomorphic way. There is $M \geq 1$ 
such that for any $z \in K$ and any $k \geq M$, this matrix has two eigenvalues: $\lambda_k^+(z), \lambda_k^-(z)$ such that
$|\lambda_k^+(z)| \geq |\lambda_k^-(z)|$. We show that the eigenvalues $(\lambda_k^+(z) : k \geq M)$ dominate 
$(\lambda_k^-(z) : k \geq M)$ in the sense that
\[
	\inf_{z \in K}
	\prod_{j=M}^\infty  
	\bigg| \frac{\lambda_j^+(z)}{\lambda_j^-(z)} \bigg| = \infty.
\]
Thus, for any $z \in K$ we obtain existence of a basis $\{ (u_n^+(z) : n \in \NN_0), (u_n^-(z) : n \in \NN_0) \}$ of the space 
of solutions of \eqref{eq:int:30} such that for any $i \in \{0, 1, \ldots, N-1\}$ the limits
\begin{equation} \label{eq:int:31}
	\lim_{k \to \infty} 
	\frac{1}{\prod_{j=M}^{k-1} \lambda_j^+(z)}
	\begin{pmatrix}
		u_{kN+i}^+(z) \\
		u_{kN+i+1}^+(z)
	\end{pmatrix}, \quad
	\lim_{k \to \infty}
	\frac{1}{\prod_{j=M}^{k-1} \lambda_j^-(z)}
	\begin{pmatrix}
		u_{kN+i}^-(z) \\
		u_{kN+i+1}^-(z)
	\end{pmatrix}
\end{equation}
exist uniformly with respect to $z \in K$ and are non-zero. Moreover, for any $n \in \NN_0$ the functions $u^+_n, u^-_n$ are 
continuous on $K$ and holomorphic in $\intr(K)$. Since $p_n(z) = f(z) u_n^+(z) + g(z) u_n^-(z)$ for some functions 
$f,g : K \to \CC$ it leads to the desired formula~\eqref{eq:int:29}, see Section~\ref{sec:7.3} for more details. Let us only 
mention that the functions $f,g$ are holomorphic in $\intr(K)$, which is a consequence of the formula~\eqref{eq:120} relating 
them to some Wronskians.

The article is organized as follows. In Section~\ref{sec:2} we collect some basic results and fix our notation. In particular, 
in Section~\ref{sec:2.1} we define the general Stolz class $\calD_r^N$ for any $r \geq 1$ and $N \geq 1$. 
In Section~\ref{sec:2.2} we collect the notation on spectrum of $A$. Section~\ref{sec:2.3} is devoted to the definition of 
transfer matrices and generalized eigenvectors of $A$. Sections~\ref{sec:2.4} and \ref{sec:2.5} discuss properties of periodic 
and periodically modulated Jacobi parameters, respectively. In Section~\ref{sec:3} we start our investigations on connections 
between Christoffel--Darboux kernel and essential spectrum of $A$ to the description of $\nu$. In particular, in 
Section~\ref{sec:3.1} we prove Theorem~\ref{thm:A}. In Section~\ref{sec:3.2} we show how the tools from Section~\ref{sec:3.1} 
can be applied to Jacobi matrices discussed in Theorems~\ref{thm:A:1}--\ref{thm:A:3b}. 
In Section~\ref{sec:5} we start our investigations based on the Cauchy transform. In particular, we prove there 
Lemma~\ref{lem:3} and Theorem~\ref{thm:C}. Section~\ref{sec:6} is the central part of this article. In Section~\ref{sec:6.1}, 
by usage of Joukowsky map, we construct and study solutions of the eigenvalue equation for the transfer matrices. 
In Section~\ref{sec:6.2} we explicitly diagonalize transfer matrices depending on their form. In Section~\ref{sec:7.1} we 
formulate our discrete Levinson's type theorem. We use it in Section~\ref{sec:7.2} to construct a basis of generalized 
eigenvectors of the form \eqref{eq:int:31} and in Section~\ref{sec:7.3} we give applications to the asymptotics of the form 
\eqref{eq:int:29}. Finally, in Section~\ref{sec:8} we prove Theorems~\ref{thm:A:1}, \ref{thm:A:3a} and \ref{thm:A:3b}.

\subsection*{Acknowledgments}
The first author was partially supported by long term structural funding -- Methusalem grant of the Flemish Government. 
Part of this work was done while he was a postdoctoral fellow at KU Leuven. He would like to thank Walter Van Assche for turning 
his attention to the articles \cite{VanAssche1985} and \cite{Geronimo1988}.

%\subsection*{Notation}
%$\CC_+ = \{ z \in \CC : \Im z > 0 \}$, 
%$B(z_0, r) = \{ z \in \CC : |z - z_0| < r \}$,
%$\overline{B}(z_0, r) = \{ z \in \CC : |z - z_0| \leq r \}$.

\section{Preliminaries} \label{sec:2}

\subsection{Stolz class} \label{sec:2.1}
In this section we define a proper class of slowly oscillating sequences which is motivated by \cite{Stolz1994}, see also \cite[Section 2]{SwiderskiTrojan2019}. 

Let $V$ be a normed space. We say that a sequence $(x_n : n \in \NN)$ of vectors from $V$ belongs to $\calD_r(V)$ for certain $r \in \NN$, if it is \emph{bounded} and for each $j \in \{1,\ldots,r\}$,
\[
	\sum_{n=1}^\infty \| \Delta^j x_n \|^{\tfrac{r}{j}} < \infty
\]
where
\begin{align*}
	\Delta^0 x_n &= x_n, \\
	\Delta^j x_n &= \Delta^{j-1} x_{n+1} - \Delta^{j-1} x_n, \quad j \geq 1.
\end{align*}
Notice that $\calD_1(V)$ consists of sequences of bounded variation. Moreover, let us recall that for any $r \geq 1$ we have $\calD_r(V) \subsetneq \calD_{r+1}(V)$.

If $V$ is the real line with Euclidean norm we abbreviate $\calD_{r} = \calD_{r}(V)$. Given a compact set
$K \subset \CC$ and a normed vector space $R$, we denote by $\calD_{r}(K, R)$ the case when $V$ is the space of all
continuous mappings from $K$ to $R$ equipped with the supremum norm. 
Let us recall that $\calD_r(V)$ is an algebra provided $V$ is a normed algebra.

Let $N$ be a positive integer. We say that a sequence $(x_n : n \in \NN)$ belongs to $\calD_r^N (V)$, if
for any $i \in \{0, 1, \ldots, N-1 \}$,
\[
	(x_{nN+i} : n \in \NN) \in \calD_r(V).
\]
Again, $\calD_r^N(V)$ is an algebra provided $V$ is a normed algebra. 

\subsection{Jacobi parameters} \label{sec:2.2}
Given two sequences $a = (a_n : n \in \NN_0)$ and $b = (b_n : n \in \NN_0)$ of positive and real numbers, respectively, by $A$ we
define the closure in $\ell^2$ of the operator acting on sequences having finite support by the matrix
\[
        \begin{pmatrix}
                b_0 & a_0 & 0   & 0      &\ldots \\
                a_0 & b_1 & a_1 & 0       & \ldots \\
                0   & a_1 & b_2 & a_2     & \ldots \\
                0   & 0   & a_2 & b_3   &  \\
                \vdots & \vdots & \vdots  &  & \ddots
        \end{pmatrix}.
\]
The operator $A$ is called \emph{Jacobi matrix}. If the Carleman's condition \eqref{eq:1}
is satisfied then the operator $A$ is self-adjoint (see e.g. \cite[Corollary 6.19]{Schmudgen2017}).
Let us denote by $E_A$ its spectral resolution of the identity. Then for any Borel subset $B \subset \RR$, we set
\[
        \mu(B) = \langle E_A(B) e_0, e_0 \rangle_{\ell^2}
\]
where $e_0$ is the sequence having $1$ on the $0$th position and $0$ elsewhere. The polynomials $(p_n : n \in \NN_0)$
form an orthonormal basis of $L^2(\RR, \mu)$. By $\sigma(A), \sigmaP(A), \sigmaS(A), \sigmaAC(A)$ and $\sigmaEss(A)$ we denote
the spectrum, the point spectrum, the singular spectrum, the absolutely continuous spectrum and the essential spectrum of $A$,
respectively.

\subsection{Generalized eigenvectors} \label{sec:2.3}
A sequence $(u_n : n \in \NN_0)$ is a \emph{generalized eigenvector} associated to
$x \in \CC$ and corresponding to $\eta \in \RR^2 \setminus \{0\}$, if the sequence of vectors
\begin{align*}
        \vec{u}_0 &= \eta, \\
        \vec{u}_n &=
        \begin{pmatrix}
                u_{n-1} \\
                u_n
        \end{pmatrix}, \quad n \geq 1,
\end{align*}
satisfies
\[
        \vec{u}_{n+1} = B_n(x) \vec{u}_n, \quad n \geq 0,
\]
where $B_n$ is the \emph{transfer matrix} defined as
\begin{equation}
        \label{eq:3}
        \begin{aligned}
        B_0(x) &=
        \begin{pmatrix}
                0 & 1 \\
                -\frac{1}{a_0} & \frac{x-b_0}{a_0}
        \end{pmatrix} \\
        B_n(x) &=
        \begin{pmatrix}
                0 & 1 \\
                -\frac{a_{n-1}}{a_n} & \frac{x - b_n}{a_n}
        \end{pmatrix}
        ,
        \quad n \geq 1.
        \end{aligned}
\end{equation}
Sometimes we write $(u_n(\eta, x) : n \in \NN_0)$ to indicate the dependence on the parameters. In particular, the
sequence of orthogonal polynomials $(p_n(x) : n \in \NN_0)$ is the generalized eigenvector associated to $\eta = e_2$
and $x \in \CC$.

Let $u,v$ be solutions of
\[
	a_{n-1} u_{n-1} + b_n u_n + a_n u_{n+1} = zu_n, \quad n \geq 0.
\]
Their \emph{Wronskian} is defined by
\[
	\Wrk_n(u,v) = 
	a_n
	\det
	\begin{pmatrix}
		u_{n-1} & v_{n-1} \\
		u_n & v_n
	\end{pmatrix}.
\]
Since it is independent of $n$ we usually omit the subscript. Notice that $\Wrk_n(u, v) \neq 0$ if and only if 
$u$ and $v$ are linearly independent.

\subsection{Periodic Jacobi parameters} \label{sec:2.4}
By $(\alpha_n : n \in \ZZ)$ and $(\beta_n : n \in \ZZ)$ we denote
$N$-periodic sequences of real and positive numbers, respectively. For each $k \geq 0$, let us define polynomials
$(\mathfrak{p}^{[k]}_n : n \in \NN_0)$ by relations
\[
        \begin{gathered}
                \mathfrak{p}_0^{[k]}(x) = 1, \qquad \mathfrak{p}_1^{[k]}(x) = \frac{x-\beta_k}{\alpha_k}, \\
                \alpha_{n+k-1} \mathfrak{p}^{[k]}_{n-1}(x) + \beta_{n+k} \mathfrak{p}^{[k]}_n(x)
                + \alpha_{n+k} \mathfrak{p}^{[k]}_{n+1}(x)
                = x \mathfrak{p}^{[k]}_n(x), \qquad n \geq 1.
        \end{gathered}
\]
Let
\[
        \frakB_n(x) =
        \begin{pmatrix}
                0 & 1 \\
                -\frac{\alpha_{n-1}}{\alpha_n} & \frac{x - \beta_n}{\alpha_n}
        \end{pmatrix},
        \qquad\text{and}\qquad
        \frakX_n(x) = \prod_{j = n}^{N+n-1} \mathfrak{B}_j(x), \qquad n \in \ZZ.
\]
By $\frakA$ we denote the Jacobi matrix corresponding to
\begin{equation*}
        \begin{pmatrix}
                \beta_0 & \alpha_0 & 0   & 0      &\ldots \\
                \alpha_0 & \beta_1 & \alpha_1 & 0       & \ldots \\
                0   & \alpha_1 & \beta_2 & \alpha_2     & \ldots \\
                0   & 0   & \alpha_2 & \beta_3   &  \\
                \vdots & \vdots & \vdots  &  & \ddots
        \end{pmatrix}.
\end{equation*}

\subsection{Periodic modulations} \label{sec:2.5}
Given $N \in \NN$, we say that $(a_n : n \in \NN_0)$ and $(b_n : n \in \NN_0)$ are $N$-periodically modulated if there
are two $N$-periodic sequences $(\alpha_n : n \in \ZZ)$ and $(\beta_n  : n \in \ZZ)$ of positive and real numbers, 
respectively, such that
\begin{enumerate}[label=\rm (\alph*), start=1, ref=\alph*]
	\item
	$\begin{aligned}[b]
	\lim_{n \to \infty} a_n = \infty
	\end{aligned},$
	\item
	$\begin{aligned}[b]
	\lim_{n \to \infty} \bigg| \frac{a_{n-1}}{a_n} - \frac{\alpha_{n-1}}{\alpha_n} \bigg| = 0
	\end{aligned},$
	\item
	$\begin{aligned}[b]
	\lim_{n \to \infty} \bigg| \frac{b_n}{a_n} - \frac{\beta_n}{\alpha_n} \bigg| = 0
	\end{aligned}.$
\end{enumerate}
We define the $N$-step transfer matrix by
\begin{equation}
	\label{eq:122}
	X_n = B_{n+N-1} B_{n+N-2} \cdots B_{n+1} B_n,
\end{equation}
where $B_n$ is defined in \eqref{eq:3}. Let us observe that for each $i \in \{0, 1, \ldots, N-1\}$,
\[
	\lim_{j \to \infty} B_{jN+i}(x) = \frakB_i(0)
\]
and
\[
	\lim_{j \to \infty} X_{jN+i}(x) = \frakX_i(0)
\]
locally uniformly with respect to $x \in \CC$. In view of \cite[Proposition 3]{PeriodicIII}, we have
\begin{equation}
	\label{eq:4}
	X_n =
	\begin{pmatrix}
		-\frac{a_{n-1}}{a_n} p^{[n+1]}_{N-2} & p_{N-1}^{[n]} \\
		-\frac{a_{n-1}}{a_n} p^{[n+1]}_{N-1} & p_N^{[n]}
	\end{pmatrix}
\end{equation}
where for each $k \in \NN_0$, $(p^{[k]}_n : n \in \NN_0)$ are $k$th associated \emph{orthonormal} polynomials defined as
\[
        \begin{gathered}
                p^{[k]}_0(x) = 1, \qquad p^{[k]}_1(x) = \frac{x - b_k}{a_k}, \\
                a_{n+k-1} p^{[k]}_{n-1}(x) + b_{n+k} p^{[k]}_n(x) + a_{n+k} p^{[k]}_{n+1}(x) =
                        x p^{[k]}_n(x), \qquad n \geq 1.
        \end{gathered}
\]
We usually omit the superscript if $k = 0$.

It turns out that the properties of the measure $\mu$ corresponding to $N$-periodically modulated
Jacobi parameters depend on the matrix $\frakX_0(0)$.
More specifically, one can distinguish four cases:
\begin{enumerate}[label=\rm (\Roman*), start=1, ref=\Roman*]
\item \label{eq:PI}
if $|\tr \frakX_0(0)|<2$, then, under some regularity assumptions on Jacobi parameters, the measure $\mu$ is purely absolutely continuous on $\RR$ with positive continuous density;

\item 
if $|\tr \frakX_0(0)|=2$, then we have two subcases: 
\begin{enumerate}[label=\rm (\alph*), start=1, ref=II(\alph*)]
\item
\label{eq:PIIa}
if $\frakX_0(0)$ is diagonalizable, then, under some regularity assumptions on Jacobi parameters,  there is a compact interval $I \subset \RR$ such that the measure $\mu$ is purely absolutely continuous on $\RR \setminus I$ with positive continuous density and it is purely discrete in the interior of $I$;

\item
\label{eq:PIIb}
if $\frakX_0(0)$ is \emph{not} diagonalizable, then usually the measure $\mu$ is purely absolutely continuous on a real half-line (or the whole real line) and discrete on its complement;
\end{enumerate}

\item \label{eq:PIII}
if $|\tr \frakX_0(0)|>2$, then, under some regularity assumptions on Jacobi parameters, the measure $\mu$ is purely
discrete with the support having no finite accumulation points. 
\end{enumerate}
One can describe these four cases geometrically. Specifically, we have 
\[
	(\tr \frakX_0)^{-1} \big( (-2, 2) \big) = \bigcup_{j=1}^N I_j
\]
where $I_j$ are disjoint open non-empty bounded intervals whose closures might touch each other. Let us denote
\begin{equation} \label{eq:5}
	I_j = (x_{2j-1}, x_{2j}) \qquad (j=1,2, \ldots, N),
\end{equation}
where the sequence $(x_k : k = 1, 2, \ldots, 2N)$ is increasing. Then we are in the case~\ref{eq:PI} if $0$ belongs to some interval \eqref{eq:5}, in the case~\ref{eq:PIIa} if $0$ lies on the boundary of exactly two intervals, in the case~\ref{eq:PIIb} if $0$ lies on the boundary of exactly one interval and in~\ref{eq:PIII} in the remaining cases. An example for $N=4$ is presented in Figure~\ref{img:1}.
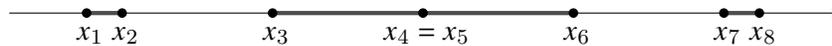
\begin{figure}[h!]
	\centering
	\begin{tikzpicture}
		\draw[->, thin, black] (-5.5,0) -- (5.5, 0);  
		\draw[-,ultra thick, black!70] (-4.472135954,0) -- (-4,0);
		\draw[-,ultra thick, black!70] (-2,0) -- (2,0);
		\draw[-,ultra thick, black!70] (4,0) -- (4.472135954,0);
		\filldraw[black] (0,0) circle (1.6pt);
		\draw (0.05,-0.3) node {$x_4=x_5$};
		\filldraw[black] (-4.472135954,0) circle (1.6pt);
		\draw (-4.422135954,-0.3) node {$x_1$};
		\filldraw[black] (-4,0) circle (1.6pt);
		\draw (-3.95,-0.3) node {$x_2$};		
		\filldraw[black] (-2,0) circle (1.6pt);
		\draw (-1.95,-0.3) node {$x_3$};
		\filldraw[black] (2,0) circle (1.6pt);
		\draw (2.05,-0.3) node {$x_6$};
		\filldraw[black] (4,0) circle (1.6pt);
		\draw (4.05,-0.3) node {$x_7$};
		\filldraw[black] (4.472135954,0) circle (1.6pt);
		\draw (4.522135954,-0.3) node {$x_8$};
	\end{tikzpicture}
	\caption{An example for $N=4$. If $0 = x_4$, then we are in the case~\ref{eq:PIIa}, while $0 \in \{x_1, x_2, x_3, x_6, x_7, x_8 \}$ corresponds to the case~\ref{eq:PIIb}.}
	\label{img:1}
\end{figure}

\section{Vague convergence via the Christoffel--Darboux kernel} \label{sec:3}
\subsection{Density of zeros} \label{sec:3.1}
Let $\calA$ be a Jacobi matrix and let us define a discrete measure
\[
	\nu_{n} = \sum_{y \in \sigma(\calA_{n+1})} \delta_{y}
\]
where $\delta_y$ denotes the Dirac measure at $y$.

\begin{lemma} 
	\label{lem:1}
	Suppose that there are an open subset $U$ of the real line, a positive sequence $(\rho_n : n \in \NN_0)$ tending to infinity, 
	a strictly increasing subsequence $(n_k : k \in \NN)$ of the natural numbers and a non-zero function $h: U \rightarrow \RR$,
	such that
	\begin{equation}
		\label{eq:7}
		\lim_{n \to \infty} \frac{1}{\rho_{n_k}} K_{n_k}(x,x) = h(x)
	\end{equation}
	locally uniformly with respect to $x \in U$. Then for each $f \in \calC_c(U)$,
	\begin{equation}
		\label{eq:8}
		\lim_{n \to \infty} 
		\frac{1}{\rho_{n_k}}
		\int_\RR f(x) \ud \nu_{n_k}(x) = \int_\RR f(x) h(x) \ud \mu(x).
	\end{equation}
\end{lemma}
\begin{proof}
	Since $h(x_0) \neq 0$ for certain $x_0 \in U$, the sequence $(p_n(x_0) : n \in \NN_0)$ cannot belong to $\ell^2$, 
	and consequently, the moment problem for $\mu$ is determinate (see e.g. \cite[Theorem 6.16(v)]{Schmudgen2017}).

	Let us define the sequence of probability measures
	\begin{equation} 
		\label{eq:9}
		\mu_{n} = 
		\sum_{y \in \sigma(\calA_{n+1})} \frac{\delta_{y}}{K_n(y,y)}
	\end{equation}
	By Gauss quadrature formula (see, e.g. \cite[Theorem 9.6]{Schmudgen2017})
	\[
		\int_{\RR} x^\ell \ud \mu_{n}(x) = 
		\int_{\RR} x^{\ell} \ud \mu(x), \qquad \ell = 0, 1, \ldots, 2n+1.
	\]
	In particular, since the moment problem for $\mu$ is determinate, we have $\mu_{n} \xrightarrow{w} \mu$
	(see, e.g. \cite[Theorem 30.2]{Billingsley1995}). Observe that \eqref{eq:9} can be written in the form 
	\[
		\ud \mu_{n}(x) = 
		\frac{1}{K_{n}(x,x)} \ud \nu_{n}(x),
	\]
	thus
	\[
		\frac{1}{\rho_{n}} K_{n}(x,x) \ud \mu_{n}(x) = \frac{1}{\rho_n} \ud \nu_{n}(x).
	\]
	Hence, for $f \in \calC_c(U)$ and $K = \supp(f)$, we have
	\begin{align*}
		\bigg| 
		\frac{1}{\rho_n}
		\int_{\RR} f(x) \ud \nu_{n}(x) - 
		\int_{\RR} f(x) h(x) \ud \mu_{n}(x) 
		\bigg| 
		&=
		\bigg| 
		\int_{\RR} f(x) \Big( \frac{1}{\rho_n} K_n(x,x) - h(x) \Big) \ud \mu_{n}(x) 
		\bigg| \\
		&\leq 
		\int_{K} |f(x)| \Big| \frac{1}{\rho_n} K_n(x,x) - h(x) \Big| \ud \mu_{n}(x) \\
		&\leq
		\mu_{n}(K) \cdot \sup_{x \in K} |f(x)| \cdot
		\sup_{x \in K} \Big| \frac{1}{\rho_n} K_n(x,x) - h(x) \Big|
	\end{align*}
	which by \eqref{eq:7} implies that
	\begin{equation} 
		\label{eq:10}
		\lim_{k \to \infty} \frac{1}{\rho_{n_k}} \int_{\RR} f(x) \ud \nu_{n_k}(x) =
		\lim_{k \to \infty} \int_{\RR} f(x) h(x) \ud \mu_{n_k}(x).
	\end{equation}
	Since $f \cdot h \in \calC_c(U) \subset \calC_c(\RR)$, by the weak convergence of $\mu_{n}$ we conclude that
	\[
		\lim_{k \to \infty} 
		\int_{K} f(x) h(x) \ud \mu_{n_k}(x) = \int_K f(x) h(x) \ud \mu(x),
	\]
	which together with \eqref{eq:10} implies \eqref{eq:8}, and the lemma follows.
\end{proof}

\begin{corollary} 
	\label{cor:1}
	Suppose that the measure $\mu$ is absolutely continuous on $U$ an open subset of the real line with positive density 
	$\mu'$. If there is a positive sequence $(\rho_n : n \in \NN_0)$ tending to infinity and a non-negative function $\upsilon: U \rightarrow \RR$, such that
	\[
		\lim_{n \to \infty}
		\frac{1}{\rho_n} K_n(x,x) = \frac{\upsilon(x)}{\mu'(x)}
	\]
	locally uniformly with respect to $x \in U$, then
	$\restr{\tfrac{1}{\rho_n} \nu_n}{U} \xrightarrow{v} \nu$ where
	\[
		\ud \nu(x) = \upsilon(x) \ud x, \quad x \in U.
	\]
\end{corollary}

The following result tells us that the supports of the limit points of $(\nu_n : n \in \NN_0)$ are contained in $\sigmaEss(A)$.

\begin{proposition} \label{prop:12}
Let $(\rho_n : n \in \NN_0)$ be a positive sequence tending to infinity. Let $\breve{A}$ be a self-adjoint extension of $A$ acting on 
$\ell^2(\NN_0)$. If for some interval $(a,b)$ we have $\sigmaEss(\breve{A}) \cap (a,b) = \emptyset$, then 
$\restr{\tfrac{1}{\rho_n} \nu_n}{(a,b)} \xrightarrow{v} 0$.
\end{proposition}
\begin{proof}
Let $f \in \calC_c \big( (a,b) \big)$. Set $K = \supp(f)$. Then(\footnote{For a finite set $X$ by $\sharp X$ we denote its cardinality.})
\[
	\bigg|
	\int_\RR f \ud \nu_n
	\bigg| 
	\leq 
	\sup_{x \in K} |f(x)| \cdot \sharp (\sigma(\calA_{n+1}) \cap K)
	\leq 
	\sup_{x \in K} |f(x)| \cdot (1 + \sharp (\sigma(\breve{A}) \cap K))
\]
where the last inequality follows from \cite[Theorem II.4.1]{Chihara1978}. Since the set $\sigma(\breve{A}) \cap K$ is finite, we get
\[
	\lim_{n \to \infty} \frac{1}{\rho_n} \int_\RR f \ud \nu_n = 0,
\]
what we needed to prove.
\end{proof}

\begin{remark} \label{rem:5}
If $A$ is not self-adjoint, then all self-adjoint extensions of $A$ acting on $\ell^2(\NN_0)$ satisfy $\sigmaEss(\breve{A}) = \emptyset$, see e.g. \cite[Theorem 7.7]{Schmudgen2017}. So in Proposition~\ref{prop:12} we then have $\tfrac{1}{\rho_n} \nu_n \xrightarrow{v} 0$.
\end{remark}

\subsection{Examples}  \label{sec:3.2}
\begin{example}[Case \ref{eq:PI}]
	\label{ex:I}
	Let $N \geq 1$. Suppose that $(a_n)$ and $(b_n)$ are $N$-periodically modulated Jacobi parameters such that 
	$|\tr \frakX_0(0)| < 2$ and for some $r \geq 1$
	\[
		\bigg( \frac{a_{n-1}}{a_n} \bigg),
		\bigg( \frac{b_n}{a_n} \bigg),
		\bigg( \frac{1}{a_n} \bigg) \in \calD_r^N.
	\]
	Then according to \cite[Theorem A]{SwiderskiTrojan2019} the moment problem for $\mu$ is determinate if and only if
	$\sum_{n=0}^\infty 1/a_n = \infty$. If it is the case then by \cite[Theorem 1.2]{ChristoffelI} the hypotheses of
	Corollary~\ref{cor:1} for $\Lambda_- = \RR$ are satisfied with
	\[
		\rho_n = \sum_{j=0}^n \frac{\alpha_j}{a_j} \quad \text{and} \quad
		\upsilon(x) \equiv \frac{|\tr \frakX_0'(0)|}{\pi N \sqrt{-\discr \frakX_0(0)}}.
	\]
	Thus, $\tfrac{1}{\rho_n} \nu_n \xrightarrow{v} \upsilon(x) \ud x$.
\end{example}
Let us remark that hypotheses of Example~\ref{ex:I} are satisfied for Hermite, Freud and Meixner--Pollaczek polynomials 
(see \cite[Section 5]{PeriodicII} for details).

\begin{example}[Case \ref{eq:PIIa}]
	\label{ex:IIa}
	Let $N \geq 1$. Suppose that $(a_n)$ and $(b_n)$ are $N$-periodically modulated Jacobi parameters such that
	$|\tr \frakX_0(0)| = 2$, and $\frakX_0(0)$ is diagonalizable. Then $\frakX_0(0) = \varepsilon \Id$ for some
	$\varepsilon \in \{-1,1\}$. Suppose that
	\[
		\bigg( a_n \Big(\frac{a_{n-1}}{a_n} - \frac{\alpha_{n-1}}{\alpha_n} \Big) \bigg),
		\bigg( a_n \Big( \frac{b_n}{a_n} - \frac{\beta_n}{\alpha_n} \Big) \bigg),
		\bigg( \frac{1}{a_n} \bigg) \in \calD_1^N.
	\]
	Then according to \cite[Proposition 4]{ChristoffelII} the moment problem for $\mu$ is determinate. Moreover,
	\cite[Proposition 9]{PeriodicIII} implies that
	\[
		h(x) = \lim_{j \to \infty} a_{jN+N-1}^2 \discr X_{jN}(x)
	\]
	exists. In view of \cite[Corollary 1]{PeriodicIII}, $h$ is a polynomial of degree $2$ with negative leading coefficient. 
	Let $\Lambda_- = h^{-1} \big( (-\infty,0) \big)$ and $\Lambda_+ = h^{-1} \big( (0, +\infty) \big)$. 
	If 
	\begin{equation}
		\label{eq:11}
		\lim_{n \to \infty} (a_{n+N} - a_n) = 0,
	\end{equation}
	then by \cite[Theorem D]{ChristoffelII} the hypotheses of Corollary~\ref{cor:1} for $U = \Lambda_-$ are satisfied with
	\[
		\rho_n = \sum_{j=0}^n \frac{\alpha_j}{a_j} \quad \text{and} \quad
		\upsilon(x) = \frac{1}{4 \pi N \alpha_{N-1}} \frac{|h'(x)|}{\sqrt{-h(x)}}.
	\]
	Thus, $\restr{\tfrac{1}{\rho_n} \nu_n}{\Lambda_-} \xrightarrow{v} \upsilon(x) \ud x$. 
	By \cite[Theorem C]{Discrete} $\sigmaEss(A) = \cl{\Lambda_-}$. Thus, Proposition~\ref{prop:12} 
	implies that $\restr{\tfrac{1}{\rho_n} \nu_n}{\Lambda_+} \xrightarrow{v} 0$. Consequently, 
	$\restr{\tfrac{1}{\rho_n} \nu_n}{\RR \setminus \partial \Lambda_-} \xrightarrow{v} \upsilon(x) \mathds{1}_{\Lambda_-}(x) \ud x$. 
	Notice that $\partial \Lambda_- = \{x : h(x) = 0\}$ is a set containing at most two points.
\end{example}
Let us remark that hypotheses of Example~\ref{ex:IIa} are satisfied for generalized Hermite polynomials 
(see \cite[Section 5]{PeriodicII} for details).

\begin{example}[Case \ref{eq:PIIb}]
	\label{ex:IIb}
	Let $N \geq 1$. Suppose that $(a_n)$ and $(b_n)$ are $N$-periodically modulated Jacobi parameters such that
	$|\tr \frakX_0(0)| = 2$ and $\frakX_0(0)$ is non-diagonalizable. Suppose that
	\[
		\bigg( a_n \Big(\frac{a_{n-1}}{a_n} - \frac{\alpha_{n-1}}{\alpha_n} \Big) \bigg),
		\bigg( a_n \Big( \frac{b_n}{a_n} - \frac{\beta_n}{\alpha_n} \Big) \bigg),
		\bigg( \frac{1}{\sqrt{a_n}} \bigg) \in \calD_1^N.
	\]
	Then according to \cite[Proposition 4]{ChristoffelII} the moment problem for $\mu$ is determinate. Moreover,
	\cite[Corollary 3.4]{jordan} implies that
	\[
		h(x) = \lim_{j \to \infty} a_{jN+N-1} \discr X_{jN}(x)
	\]
	exists; $h$ is a polynomial of degree $1$. Let $\Lambda_- = h^{-1} \big( (-\infty,0) \big)$ and 
	$\Lambda_+ = h^{-1} \big( (0, +\infty) \big)$. If 
	\begin{equation}
		\label{eq:12}
		\lim_{n \to \infty} (a_{n+N} - a_n) = 0,
	\end{equation}
	then by \cite[Theorem D]{jordan} the hypotheses of Corollary~\ref{cor:1} for $U = \Lambda_-$ are satisfied with
	\[
		\rho_n = \sum_{j=0}^n \sqrt{\frac{\alpha_j}{a_j}} \quad \text{and} \quad
		\upsilon(x) = \frac{\sqrt{\alpha_{N-1}}}{\pi N} \frac{|\tr \frakX_0'(0)|}{\sqrt{-h(x)}}.
	\]
	Thus, $\restr{\tfrac{1}{\rho_n} \nu_n}{\Lambda_-} \xrightarrow \upsilon(x) \ud x$.
	By \cite[Theorem A]{jordan} $\sigmaEss(A) = \cl{\Lambda_-}$.
	Thus, Proposition~\ref{prop:12} implies that $\restr{\tfrac{1}{\rho_n} \nu_n}{\Lambda_+} \xrightarrow{v} 0$. Consequently, 
	$\restr{\tfrac{1}{\rho_n} \nu_n}{\RR \setminus \partial \Lambda_-} \xrightarrow{v} \upsilon(x) \mathds{1}_{\Lambda_-}(x) \ud x$. 
	Notice that $\partial \Lambda_- = \{x : h(x) = 0\}$ is a set containing one point.	
\end{example}
Let us remark that hypotheses of Example~\ref{ex:IIb} are satisfied for Laguerre-type polynomials 
(see \cite[Section 10.1.1]{jordan} for details).

For the completeness we present an example for which hypotheses of Corollary~\ref{cor:1} cannot be satisfied.
\begin{example}[Case \ref{eq:PIII}] 
	\label{ex:III}
	Let $N \geq 1$. Suppose that $(a_n)$ and $(b_n)$ are $N$-periodically modulated Jacobi parameters such that
	$|\tr \frakX_0(0)| > 2$, and for some $r \geq 1$,
	\[
		\bigg( \frac{a_{n-1}}{a_n} \bigg),
		\bigg( \frac{b_n}{a_n} \bigg),
		\bigg( \frac{1}{a_n} \bigg) \in \calD_r^N.
	\]
	Then according to \cite[Theorem A]{Discrete} the moment problem for $\mu$ is determinate and the measure $\mu$
	is purely discrete on $\RR$ and $\supp(\mu)' = \emptyset$. Hence, $\mu'$ is zero almost everywhere on $\RR$. 
	However, since $\sigmaEss(A) = \emptyset$, Proposition~\ref{prop:12} implies that $\tfrac{1}{\rho_n} \nu_n \xrightarrow{v} 0$ 
	for any sequence positive sequence $(\rho_n : n \in \NN_0)$ tending to infinity.
\end{example}
Let us remark that Meixner polynomials satisfy the hypotheses of Example~\ref{ex:III}.

\section{Weighted weak convergence via the Cauchy transform} \label{sec:5}

\subsection{The convergence of Cauchy transforms}

Let $\calA$ be a Jacobi matrix and let measure $\breve{\mu}$ be defined in \eqref{eq:int:24} for some self-adjoint extension $\breve{A}$ acting on $\ell^2(\NN_0)$ of the Jacobi operator $A$. Set
\begin{equation} 
	\label{eq:25}
	\ud \eta_n(x) = K_n(x,x) \ud \breve{\mu}(x) 
	\quad \text{and} \quad
	\nu_n = \sum_{y \in \sigma(\calA_{n+1})} \delta_y.
\end{equation}

Suppose that $\omega$ is a signed Borel measure on the real line bounded in the total variation norm: $\|\omega\|_{\mathrm{TV}} < \infty$. 
Then its \emph{Cauchy transform} is defined by the formula
\[
	\calC[\omega](z) = 
	\int_\RR \frac{1}{x-z} \ud \omega(x), \quad z \in \CC \setminus \RR.
\]
The above integral is well-defined because
\[
	\int_\RR \frac{1}{|x-z|} \ud |\omega|(x)
	\leq 
	\frac{1}{|\Im z|} \|\omega\|_{\mathrm{TV}}.
\]
It is well-known that $\calC[\omega]$ is a holomorphic function satisfying
$\overline{\calC[\omega](\overline{z})} = \calC[\omega](z)$. Moreover, if $\omega$ is positive then it maps $\CC_+$ into itself.
The function $\calC[\omega]$ uniquely determines the measure $\omega$, and under some hypotheses, for a sequence 
$(\omega_n : n \in \NN)$ of positive finite measures satisfying $\sup_{n \in \NN} \omega_n(\RR) < \infty$ the pointwise
convergence of $\calC[\omega_n]$ implies the weak convergence of $(\omega_n : n \in \NN)$, see \cite{Geronimo2003} for
more details. 

The following lemma discusses the case when the condition $\sup_{n \in \NN} \omega_n(\RR) < \infty$ might be violated. 
The idea of the proof is based on the proof of \cite[Theorem 3.1]{VanAssche1985}.
\begin{lemma} 
	\label{lem:3}
	Let $(\omega_n : n \in \NN)$ be a sequence of positive finite Borel measures on the real line. Suppose that there is 
	a set $X \subset \CC_+$ with an accumulation point in $\CC_+$ such that for each $z \in X$,
	\begin{equation}
		\label{eq:27}
		\text{the limit } \lim_{n \to \infty} \calC[\omega_n](z)
		\text{ exists.}
	\end{equation}
	Then there is a holomorphic function $g: \CC_+ \to \CC_+ \cup \RR$ such that
	\[
		\lim_{n \to \infty} \calC[\omega_n](z) = g(z)
	\]
	locally uniformly with respect to $z \in \CC_+$. Moreover, for any
	$z_0 = x_0 + i y_0$, where $x_0 \in \RR$ and $y_0 > 0$, there is a Borel measure $\tilde{\omega}_{z_0}$ on the real line 
	such that
	\begin{equation}
		\label{eq:28}
		\lim_{n \to \infty} 
		\int_{\RR} f(x) \frac{\ud \omega_n(x)}{|x-z_0|^2} =
		\int_{\RR} f(x) \ud \tilde{\omega}_{z_0}(x), 
		\quad f \in \calC_0(\RR).
	\end{equation}
	Furthermore, the measure $\tilde{\omega}_{z_0}$ satisfies $\tilde{\omega}_{z_0}(\RR) \leq \frac{1}{y_0} \Im \big( g(z_0) \big)$, and its Cauchy
	transform is equal to
	\begin{equation}
		\label{eq:29}
		\calC[\tilde{\omega}_{z_0}](z)
		=
		\frac{1}{(z-x_0)^2 + y_0^2}
		\Big( g(z) -\Re \big( g(z_0) \big) - \frac{(z-x_0)}{y_0} \Im \big( g(z_0) \big)  \Big), \quad
		z \in \CC_+ \setminus \{ x_0 + i \sqrt{y_0} \}.
	\end{equation}
	Finally, if $\tilde{\omega}_{z_0}(\RR) = \frac{1}{y_0} \Im \big( g(z_0) \big)$, then the convergence in \eqref{eq:28} holds for any
	$f \in \calC_b(\RR)$.
\end{lemma}
\begin{proof}
	For $n \in \NN$, we set
	\[
		g_n(z) = \calC[\omega_n](z), \quad z \in \CC_+.
	\]
	Since
	\[
		\Im \big( g_n(z) \big) =
		\Im(z) \int_{\RR} \frac{1}{|x-z|^2} \ud \omega_n(x),
	\]
	the function $g_n$ is a Herglotz function. Let us recall that the set of Herglotz functions forms a normal family,
	see e.g. \cite[pp. 237]{SimonComplexA}. By \eqref{eq:27}, the sequence $(g_n)$ is not divergent thus by its normality 
	every subsequence $(g_{n_k} : k \in \NN)$ contains a further subsequence $(g_{n_{k_j}} : j \in \NN)$ convergent locally
	uniformly on $\CC_+$ to certain holomorphic function $g$. Let $\tilde{g}$ be another subsequential limit. In view of 
	\eqref{eq:27}, $g(z) = \tilde{g}(z)$ for all $z \in X$, and since the set $X$ has an accumulation point we conclude that 
	$g = \tilde{g}$ on $\CC_+$. Consequently, the sequence $(g_n : n \in \NN)$ converges to $g$ locally uniformly
	on $\CC_+$.

	Next, let us fix $z_0 = x_0 + i y_0 \in \CC_+$ and define
	\[
		\ud \tilde{\omega}_n(x) = \frac{\ud \omega_n(x)}{|x-z_0|^2}.
	\]
	Since
	\[
		g(z_0) = \lim_{n \to \infty} \int_\RR \frac{1}{x-z_0} \ud \omega_n(x)
	\]
	we obtain
	\begin{align} 
		\label{eq:30}
		\Re \big( g(z_0) \big) 
		&= 
		\lim_{n \to \infty} \int_{\RR} \frac{x-x_0}{|x-z_0|^2} \ud \omega_n(x), \\
		\intertext{and}
		\label{eq:31}
		\Im \big( g(z_0) \big)
		&=
		\lim_{n \to \infty} \int_{\RR} \frac{y_0}{|x-z_0|^2} \ud \omega_n(x).
	\end{align}
	Observe that
	\[
		\frac{1}{x-z} \frac{1}{|x-z_0|^2} =
		\frac{1}{(z-x_0)^2 + y_0^2}
		\bigg( 
		\frac{1}{x-z}
		-\frac{x-x_0}{|x-z_0|^2}
		-\frac{z-x_0}{|x-z_0|^2}
		\bigg).
	\]
	Hence for $z \in \CC_+ \setminus \{ x_0 + i\sqrt{y_0} \}$,
	\begin{align*}
		\int_\RR \frac{1}{x-z} \ud \tilde{\omega}_n(x) =
		\frac{1}{(z-x_0)^2 + y_0^2} 
		\bigg(  
		\int_\RR \frac{1}{x-z} \ud \omega_n(x)
		&-
		\int_\RR \frac{x-x_0}{|x-z_0|^2} \ud \omega_n(x) \\
		&-
		\frac{z-x_0}{y_0} \int_\RR \frac{y_0}{|x-z_0|^2} \ud \omega_n(x)
		\bigg).
	\end{align*}
	Now, taking $n$ approaching infinity and using \eqref{eq:27}, \eqref{eq:30} and \eqref{eq:31} we show that
	\begin{equation}
		\label{eq:32}
		\lim_{n \to \infty}
		\calC[\tilde{\omega}_n](z) =
		\frac{1}{(z-x_0)^2 + y_0^2}
		\Big( g(z) - \Re \big( g(z_0) \big) - \frac{z-x_0}{y_0} \Im \big( g(z_0) \big) \Big)
	\end{equation}
	locally uniformly with respect to $z \in \CC_+ \setminus \{ x_0 + i\sqrt{y_0} \}$. Since
	\[
		\|\tilde{\omega}_n\|_{\mathrm{TV}}
		\leq
		\int_{\RR} \frac{1}{|x-z_0|^2} \ud \omega_n(x),
	\]
	by \eqref{eq:31} the sequence $(\tilde{\omega}_n)$ is uniformly bounded in total variation norm. Hence, by Banach--Alaoglu
	and Riesz--Markov theorems, every subsequence $(\tilde{\omega}_{n_k} : k \in \NN)$ contains a further subsequence 
	$(\tilde{\omega}_{n_{k_j}} : j \in \NN)$ such that for a finite Borel measure on the real line $\tilde{\omega}_{z_0}$ 
	\[
		\lim_{j \to \infty} 
		\int_{\RR} f \ud \tilde{\omega}_{n_{k_j}} =
		\int_{\RR} f \ud \tilde{\omega}_{z_0}, \text{ for all } f \in \calC_0(\RR).
	\]
	In particular, for all $z \in \CC_+$,
	\[
		\lim_{j \to \infty} 
		\calC[\tilde{\omega}_{n_{k_j}}](z) =
		\calC[\tilde{\omega}_{z_0}](z).
	\]
	In view of \eqref{eq:32}, the measure $\tilde{\omega}_{z_0}$ satisfies \eqref{eq:29}. Since the Cauchy transform uniquely
	determines the measure, we conclude that all subsequential limits are equal, that is the sequence 
	$(\tilde{\omega}_n : n \in \NN)$ converges to the unique positive Borel measure satisfying \eqref{eq:29}, proving
	\eqref{eq:28}.

	Next, the inequality $\tilde{\omega}_{z_0}(\RR) \leq \frac{1}{y_0} \Im \big( g(z_0) \big)$ is a consequence of \eqref{eq:28}, \eqref{eq:31} 
	and \cite[Lemma 13.15]{Klenke2020}. 
	Finally, let us suppose that $\tilde{\omega}_{z_0}(\RR) = \frac{1}{y_0} \Im \big( g(z_0) \big)$. Hence, by \eqref{eq:31}
	\[
		\lim_{n \to \infty} \tilde{\omega}_n(\RR) = \tilde{\omega}_{z_0}(\RR).
	\]
	Since the measures $(\tilde{\omega}_n : n \in \NN)$ are uniformly bounded, we can consider the sequence
	\[
		\eta_n = \frac{1}{\sup_{j \in \NN} \tilde{\omega}_j(\RR)} \tilde{\omega}_n.
	\]
	Notice that for subprobability measures $(\eta_n : n \in \NN)$ the convergence \eqref{eq:28} for $f \in \calC_0(\RR) \cup \{1 \}$
	implies the convergence for $f \in \calC_b(\RR)$, (see e.g. \cite[Theorem 13.16]{Klenke2020}), which completes the proof.
\end{proof}

Our interest in Lemma~\ref{lem:3} lies in the following theorem, which is motivated by \cite[Theorem 1.5]{Simon2009}.
\begin{theorem}
	\label{thm:4}
	Let $\calA$ be a Jacobi matrix, let the sequences $(\eta_n : n \in \NN_0), (\nu_n : n \in \NN_0)$ be defined in \eqref{eq:25}
	and $(\rho_n : n \in \NN_0)$ be a positive sequence tending to infinity.
	Then
	\begin{equation} 
		\label{eq:33}
		\lim_{n \to \infty} 
		\big| 
		\calC \big[ \tfrac{1}{\rho_n} \eta_n \big](z) - 
		\calC \big[ \tfrac{1}{\rho_n} \nu_n \big](z) 
		\big| = 0 
	\end{equation}
	locally uniformly with respect to $z \in \CC \setminus \RR$. Moreover, for any $L \geq 0$
	\begin{equation} \label{eq:34}
		\lim_{n \to \infty} 
		\big| 
		\calC \big[\tfrac{1}{\rho_{n}} \eta_{n+L} \big](z) - 
		\calC \big[ \tfrac{1}{\rho_{n}} \eta_n \big](z) 
		\big| = 0
	\end{equation}
	locally uniformly with respect to $z \in \CC \setminus \RR$.
\end{theorem}
\begin{proof}
	Let $z \in \CC \setminus \RR$. Let $\Pi_n$ be the orthogonal projection
	in $\ell^2(\NN_0)$ onto the subspace $\lin \{e_0, e_1, \ldots, e_{n-1} \}$.
	Set
	\[
		A_n := \Pi_n \breve{A} \Pi_n, \quad\text{and}\quad
		\Id_n := \Pi_n \Id \Pi_n.
	\]
	Notice that $A_n, \Id_n \in \Mat(n, \RR)$. By the spectral theorem for $A_{n+1}$ we obtain
	\[
		\int_{\RR} \frac{1}{x-z} \ud \nu_n(x) =
		\tr \big( (A_{n+1} - z \Id_{n+1})^{-1} \big).
	\]
	Similarly, by the spectral theorem for $\breve{A}$,
	\begin{align*}
		\int_{\RR} 
		\frac{1}{x-z} \ud \eta_n(x)
		&=
		\sum_{k=0}^n 
		\int_{\RR} 
		\frac{p_k^2(x)}{x - z} \ud \breve{\mu}(x) \\
		&=
		\sum_{k=0}^n 
		\big\langle (\breve{A} - z \Id)^{-1} e_k, e_k \big\rangle_{\ell^2(\NN_0)} \\
		&=
		\tr \big( \Pi_{n+1} (\breve{A} - z \Id)^{-1} \Pi_{n+1} \big).
	\end{align*}
	Hence, we get
	\begin{equation}
		\label{eq:35}
		\big| 
		\calC \big[\tfrac{1}{\rho_n} \nu_n \big](z) - 
		\calC \big[\tfrac{1}{\rho_n} \eta_n \big](z)
		\big|
		=
		\frac{1}{\rho_n} 
		\big| 
		\tr \big( (A_{n+1} - z \Id_{n+1})^{-1} \big) - 
		\tr \big( \Pi_{n+1} (\breve{A} - z \Id)^{-1} \Pi_{n+1} \big) 
		\big|.
	\end{equation}
	Let us recall that \cite[formula (2.3)]{Geronimo1988} reads
	\[
		\big| 
		\tr \big( (A_{n+1} - z \Id_{n+1})^{-1} \big) - 
		\tr \big( \Pi_{n+1} (\breve{A} - z \Id)^{-1} \Pi_{n+1} \big) 
		\big| 
		\leq
		\frac{8}{|\Im z|}.
	\]
	Since $\rho_n \to \infty$, by the formula \eqref{eq:35} we easily get \eqref{eq:33}.

	For the proof of \eqref{eq:34} it is enough to write
	\begin{align*}
		\| \eta_{n+L} - \eta_n \|_{\mathrm{TV}}
		&=
		\int_\RR |K_{n+L}(x,x) - K_{n}(x,x)| \ud \breve{\mu}(x) \\
		&=
		\sum_{j = n+1}^{n+L} \int_\RR \abs{p_j(x)}^2 \ud \breve{\mu}(x) \\
		&=
		L.
	\end{align*}
	Therefore,
	\begin{align*}
		\big| 
		\calC \big[ \tfrac{1}{\rho_{n}} \eta_{n+L} \big](z) - 
		\calC \big[ \tfrac{1}{\rho_{n}} \eta_n \big](z) 
		\big|
		&=
		\big|
		\calC \big[ \tfrac{1}{\rho_{n}} \eta_{n+L} - \tfrac{1}{\rho_{n}} \eta_n \big] (z) 
		\big| \\
		&\leq
		\int_\RR \frac{1}{\abs{x - z}} \ud \big|\tfrac{1}{\rho_{n}} \eta_{n+L} - \tfrac{1}{\rho_{n}} \eta_n \big| (x) \\
		&\leq
		\frac{1}{\abs{\Im z}} \frac{1}{\rho_{n}} \| \eta_{n+L} - \eta_n \|_{\mathrm{TV}} \\
		&=
		\frac{1}{\abs{\Im z}} \frac{L}{\rho_{n}},
	\end{align*}
	and the formula \eqref{eq:34} follows.
\end{proof}

Theorem~\ref{thm:4} together with Lemma~\ref{lem:3} allow us to prove the following partial converse to Corollary~\ref{cor:1}.
\begin{corollary} 
	\label{cor:3}
	Suppose that $A$ is self-adjoint, let the sequences $(\eta_n : n \in \NN_0), (\nu_n : n \in \NN_0)$ be defined in \eqref{eq:25}
	and $(\rho_n : n \in \NN_0)$ be a positive sequence tending to infinity.	
	Suppose that there is
	$g: \CC_+ \rightarrow \CC_+ \cup \RR$, such that for each $z \in \CC_+$, 
	\begin{equation} 
		\label{eq:36}
		\lim_{n \to \infty} \calC \big[ \tfrac{1}{\rho_n} \nu_n \big](z) = g(z).
	\end{equation}
	Let $z_0 = x_0 + i y_0$ for some $x_0 \in \RR$ and $y_0 > 0$ be given	
	and let $\tilde{\omega}_{z_0}$ be a positive Borel measure on the real line corresponding to the Cauchy transform
	\[
		\calC[\tilde{\omega}_{z_0}](z)
		=
		\frac{1}{(z-x_0)^2 + y_0^2}
		\Big( g(z) -\Re \big( g(z_0) \big) - \frac{(z-x_0)}{y_0} \Im \big( g(z_0) \big)  \Big), \quad
		z \in \CC_+ \setminus \{ x_0 + i \sqrt{y_0} \}.
	\]
	Set $\ud \omega = |x-z_0|^2 \ud \tilde{\omega}_{z_0}(x)$. Then for each function $f$ such that $(1+x^2)f \in \calC_0(\RR)$,
	\begin{equation} 
		\label{eq:152}
		\lim_{n \to \infty} 
		\frac{1}{\rho_n}
		\int_\RR f \ud \nu_n =
		\int_\RR f \ud \omega, 
		\quad \text{and} \quad
		\lim_{n \to \infty} 
		\frac{1}{\rho_n}
		\int_\RR f \ud \eta_n =
		\int_\RR f \ud \omega.
	\end{equation}
	In particular, $\tfrac{1}{\rho_n} \eta_n \xrightarrow{v} \omega$ and $\frac{1}{\rho_n} \nu_n \xrightarrow{v} \omega$. If 
	$\tilde{\omega}_{z_0}(\RR) = \Im \big( g(z_0) \big)$, then the convergence in \eqref{eq:152} holds for all $f$ such that  
	$(1+x^2)f \in \calC_b(\RR)$.

	Moreover, for every open subset $U \subset \RR$ such that the measure $\omega$ is absolutely continuous on $U$ with 
	a~continuous positive density $\upsilon$, if there is a function $h: U \rightarrow (0, \infty)$ so that
	\begin{equation}
		\label{eq:37}
		\lim_{n \to \infty} \frac{1}{\rho_n} K_n(x,x) = h(x)
	\end{equation}
	locally uniformly with respect to $x \in U$, then the measure $\mu$ is absolutely continuous on $U$ with density
	\[
		\mu'(x) = \frac{\upsilon(x)}{h(x)}, \quad x \in U.
	\]
\end{corollary}
\begin{proof}
	Let $f \in \calC_0(\RR)$. Define a new function $\tilde{f}(x) = |x-z_0|^2 f(x)$. Since
	\[
		\lim_{|x| \to \infty} \frac{|x-z_0|^2}{x^2+1} = 
		\lim_{|x| \to \infty} \frac{(x-x_0)^2 + y_0^2}{x^2+1} = 1.
	\]
	We have $\tilde{f} \in \calC_0(\RR)$ iff $(1+x^2)f \in \calC_0(\RR)$ and $\tilde{f} \in \calC_b(\RR)$ iff $(1+x^2)f \in \calC_b(\RR)$.
	Thus, by application of Lemma~\ref{lem:3} to $\big( \tfrac{1}{\rho_n} \nu_n : n \in \NN_0 \big)$
	and $\tilde{f}$, we easily get the first formula in \eqref{eq:152}. 
	In view of Theorem \ref{thm:4}, we can also apply Lemma \ref{lem:3} to $\big( \tfrac{1}{\rho_n} \eta_n : n \in \NN_0 \big)$ proving the second
	formula in \eqref{eq:152}.

	Next, let $f \in \calC_c(U)$. By \eqref{eq:37} we have
	\begin{align*}
		\lim_{n \to \infty}
		\frac{1}{\rho_n}
		\int_\RR f(x) \ud \eta_n(x)
		&=
		\lim_{n \to \infty} 
		\int_\RR f(x) \frac{1}{\rho_n} K_n(x,x) \ud \mu(x) \\
		&=
		\int_\RR f(x) h(x) \ud \mu(x).
	\end{align*}
	Since $\tfrac{1}{\rho_n} \eta_n \xrightarrow{v} \omega$, we necessarily have
	\[
		h(x) \ud \mu(x) = \ud \omega(x) = \upsilon(x) \ud x, \quad x \in U,
	\]
	(see e.g. \cite[Theorem 13.11(ii)]{Klenke2020}). Since both $\upsilon$ and $h$ are positive, the measure $\mu$
	is absolute continuous with respect to Lebesgue measure with density $\frac{\upsilon}{h}$. The proof is complete.
\end{proof}

Let us comment on Corollary \ref{cor:3}. There are examples when we can only show the existence of the limit \eqref{eq:37}, 
see e.g. Example~\ref{ex:IIa} when \eqref{eq:11} is violated, Example~\ref{ex:IIb} without the condition \eqref{eq:12}, or 
more generally, \cite[Theorem B]{jordan2}. Corollary \ref{cor:3} allows to compute the density of the measure $\mu$ under
the hypotheses \eqref{eq:36}. Moreover, it gives the convergence \eqref{eq:152} for a wider class of test functions. 

In the rest of the paper we shall be interested in the pointwise convergence of $\calC \big[ \tfrac{1}{\rho_n} \nu_n \big]$. An important observation is 
the following fact.
\begin{remark} 
	\label{rem:2}
	Let us observe that 
	\[
		\calC [\nu_{n}](z) = 
		-\frac{p_{n+1}'(z)}{p_{n+1}(z)}, \quad z \in \CC \setminus \RR.
	\]
	To see this, recall that the polynomial $p_{n+1}$ has $n+1$ simple real zeros. Hence,
	\[
		p_{n+1}(z) = c_{n+1} (z-x_{1;n+1}) (z-x_{2;n}) \ldots (z-x_{n+1;n+1}),
	\]
	for certain $c_{n+1} > 0$ and $x_{1;n+1}, \ldots, x_{n+1,n+1} \in \RR$. Therefore,
	\[
		\frac{p_{n+1}'(z)}{p_{n+1}(z)} = 
		\sum_{j=1}^{n+1} \frac{1}{z-x_{j;n+1}} =
		\int_\RR \frac{1}{z-x} \ud \nu_{n}(x)
	\]
	from which the conclusion follows.
\end{remark}

\section{Diagonalization of transfer matrices} \label{sec:6}
\subsection{Analytic tools} \label{sec:6.1}

\begin{proposition}
	\label{prop:1}
	Given $z \in \CC \setminus \{-1, 1\}$, if $\xi \in \CC$ satisfies
	\begin{equation}
		\label{eq:51}
		\xi^2 - 2 z \xi + 1 = 0,
	\end{equation}
	then
	\[
		|\xi|^2 - (|z-1| + |z+1|) |\xi| + 1 =0.
	\]
\end{proposition}
\begin{proof}
	First, let us consider
	\[
		J(w) = \frac{1}{2}\bigg(w+\frac{1}{w}\bigg), \qquad w \in \CC \setminus \{0\}.
	\]
	If $w = r \ue^{i \theta}$, then we have
	\[
		J(w) = 
		\frac{1}{2} \Big( r + \frac{1}{r} \Big) \cos \theta + 
		i \frac{1}{2} \Big( r - \frac{1}{r} \Big) \sin \theta.
	\]
	Hence circle centered at zero and radius $r \neq 1$ is transformed into an ellipse with principle axes
	\[
		a = \frac{1}{2} \Big( r + \frac{1}{r} \Big), \qquad
		b = \frac{1}{2} \Big| r - \frac{1}{r} \Big|.
	\]
	Since for every point $P$ on the ellipse we have $2a = |P-c| + |P+c|$ where $c = \pm 1$ are its foci, we obtain
	\begin{equation}
		\label{eq:54}
		|J(w) - 1| + |J(w) + 1| = r + \frac{1}{r}, \qquad |w| \neq 1.
	\end{equation}
	In fact, a direct computation shows that \eqref{eq:54} also holds true for $|w| = 1$.

	Let $z \in \CC \setminus \{-1, 1\}$. Let $\xi$ denote a solution of \eqref{eq:51}. Observe that $z = J(\xi_+)$. 
	Hence, \eqref{eq:54} leads to
	\[
		|z - 1| + |z + 1| = |\xi| + \frac{1}{|\xi|},
	\]
	that is
	\[
		|\xi|^2 - (|z-1|+|z+1|) |\xi|  + 1 = 0,
	\]
	which completes the proof.
\end{proof}

\begin{remark}
	\label{rem:3}
	Let us observe that the function $f(z) = \frac{1}{2} \Log(z - 1) + \frac{1}{2} \Log(z+1)$ is holomorphic in 
	$\CC \setminus (-\infty, 1]$ where $\Log$ denotes the principle branch of logarithm. Using the fundamental theorem of 
	calculus one can check that the function $\exp f(z)$ extends continuously to $(-\infty, -1)$. Thus by Morera's theorem 
	$\exp f(z)$ is holomorphic in $\CC \setminus [-1, 1]$. Let us recall that
	\[
		\exp f(z) = \sqrt{z-1}\sqrt{z+1}, \quad 
		z \in \CC \setminus (-\infty,1]
	\]
	where $\sqrt{\cdot}$ is the principal branch of the square root. 
	Observe that
	\begin{equation} \label{eq:56}
		\exp f [\CC_\varepsilon] \subset \CC_\varepsilon, \quad 
		\varepsilon \in \{-1,1\}.
	\end{equation}
	Since for $z \in \CC \setminus \RR$
	\[
		\Log(-z) = \Log(z) +
		\begin{cases}
			-i \pi & z \in \CC_+ \\
			i \pi & z \in \CC_-
		\end{cases}
	\]
	we get
	\begin{equation} \label{eq:57}
		\exp f(-z) = - \exp f(z), \quad
		z \in \CC \setminus \RR.
	\end{equation}

	Given $z \in \CC$, we consider the equation
	\begin{equation} \label{eq:58}
		\xi^2 - 2 z \xi + 1 = 0.
	\end{equation}
	If $z \notin [-1, 1]$, it has two distinct solutions
	\begin{equation} \label{eq:59}
		\xi_+(z) = z + \exp f(z), \quad\text{and}\quad
		\xi_-(z) = z - \exp f(z).
	\end{equation}
	The functions $\xi_+$ and $\xi_-$ are holomorphic in $\CC \setminus [-1, 1]$. Next, the implicit function theorem applied 
	to \eqref{eq:58} implies
	\begin{equation} \label{eq:60}
		\frac{\ud \xi_+}{\ud z} = \frac{\xi_+}{\xi_+ - z}, \quad z \in \CC \setminus \RR.
	\end{equation}
	Finally, by \eqref{eq:56} and \eqref{eq:59} we get
	\begin{equation} \label{eq:61}
		\xi_+ [\CC_\varepsilon] \subset \CC_\varepsilon, \quad
		\varepsilon \in \{-1,1\}.
	\end{equation}
	
	We extend the functions $\xi_+$ and $\xi_-$ to $\CC$ as follows:
	Let us consider a sequence $(z_n : n \in \NN) \subset \CC_+$ convergent to $x \in (-1, 1)$. Since 
	$|\xi_+(z_n)| > 1$, there is $C > 0$ such that
	\[
		\bigg|\frac{{\rm d} \xi_+}{{\rm d} z}(z_n) \bigg| = \bigg|\frac{\xi_+(z_n)}{\xi_+(z_n) - z_n} \bigg|\leq C.
	\]
	Hence, by the mean value theorem
	\[
		|\xi_+(z_n) - \xi_+(z_m)| \leq C |z_n - z_m|,
	\]
	thus $(\xi_+(z_n) : n \in \NN)$ is Cauchy sequence. If $(w_n : n \in \NN)$ is a sequence approaching $x$, then
	\begin{align*}
		|\xi_+(z_n) - \xi_+(w_m) | 
		&\leq C |z_n - w_m| \\
		&\leq C (|z_n - x| + |x - w_m|),
	\end{align*}
	hence $(\xi_+(w_n) : n \in \NN)$ converges to the same limit. We set
	\[
		\xi_+(x) = \lim_{n \to \infty} \xi_+(z_n),
	\]
	and $\xi_-(x) = 2 x - \xi_+(x)$. We claim that
	\[
		|\xi_+(z)| > |\xi_-(z)|, \qquad\text{for all } z \in \CC \setminus [-1, 1].
	\]
	Indeed, the inequality holds true for $z \in (1, \infty)$. Let us suppose, contrary to our claim, that
	$|\xi_+(z)| \leq |\xi_-(z)|$ for certain $z \in \CC \setminus [-1, 1]$. By continuity we deduce that there is 
	$z \in \CC \setminus [-1, 1]$ such that $|\xi_+(z)| = |\xi_-(z)|$. Hence, by Proposition \ref{prop:1}, $|z-1|+|z+1| = 2$ 
	which implies that $z \in [-1, 1]$, a contradiction.

	Lastly, we show that $\xi_+$ and $\xi_-$ are continuous at $-1$ and $1$. Let us consider a sequence
	$(z_n : n \in \NN) \subset \CC$ convergent to $x \in \{-1, 1\}$. In view of Proposition \ref{prop:1}, 
	$(\xi_+(z_n) : n \in \NN)$ and $(\xi_-(z_n) : n \in \NN)$ are bounded. Let $(z_{n_j} : j \in \NN)$ be subsequence such
	that $(\xi_+(z_{n_j}) : j \in \NN)$ and $(\xi_-(z_{n_j}) : j \in \NN)$ both converge. Their limits satisfy
	\[
		g^2 - 2 x g + 1 = 0.
	\]
	Consequently, $(\xi_+(z_n) : n \in \NN)$ and $(\xi_-(z_n) : n \in \NN)$ both converge to $x$.

	Finally, let us observe that for $z \in \CC_-$ we have
	\[
		\xi_+(z) = \overline{\xi_+(\overline{z})}.
	\]
	Therefore, if for a sequence $(z_n : n \in \NN_0) \subset \CC_-$ convergent to $x \in (-1, 1)$, we have
	\begin{align*}
		\lim_{n \to \infty} \xi_+(z_n) 
		&= \lim_{n \to \infty} \overline{\xi_+(\overline{z_n})} \\
		&= \overline{\xi_+(x)}.
	\end{align*}
	Since $x \in \RR$, we have $\overline{\xi_+(x)}=\xi_-(x)$, thus
	\[
		\lim_{n \to \infty} \xi_+(z_n) = 2x - \xi_+(x).
	\]
\end{remark}

The following proposition will be crucial in proving uniform Levinson's condition (see \eqref{eq:46}).
\begin{proposition} \label{prop:15}
For any $w \in \CC$ we have
\begin{equation} \label{eq:195}
	\abs{\xi_+(w)} = 
	\frac{1}{2}
	\Big( 
		\abs{w-1} + \abs{w+1} + \sqrt{(\abs{w-1} + \abs{w+1})^2-4}
	\Big).
\end{equation}
Moreover, for any $w \in \CC \setminus \RR$ we have
\begin{equation} \label{eq:196}
	\abs{\xi_+(w)} - 1 \geq
	\frac{\abs{\Im w}}{2\sqrt{\abs{w-1} \abs{w+1}}}.
\end{equation}
\end{proposition}
\begin{proof}
Since $\abs{\xi_+(w)} \geq \abs{\xi_-(w)}$ Proposition~\ref{prop:1} implies \eqref{eq:195}. Notice that by \eqref{eq:195} we have
\begin{equation} \label{eq:198}
	\frac{1}{2} \sqrt{(\abs{w-1} + \abs{w+1})^2-4} 
	\leq 
	\abs{\xi_+(w)} - 1.
\end{equation}
So to prove \eqref{eq:196} we need to estimate the left-hand side of this inequality. Without loss of generality suppose that $w \in \CC_+$. We shall consider the triangle $\Delta$ on the complex plane spanned by the vertices: $-1,w,1$. By law of cosines at 
the vertex $w$ we have
\[
	2^2 = \abs{w-1}^2 + \abs{w+1}^2 - 2 \abs{w-1} \abs{w+1} \cos \theta,
\]
where $\theta = \angle(-1,w,1)$. So 
\begin{align*}
	(\abs{w-1}+\abs{w+1})^2-4 
	&=
	2 \abs{w-1} \abs{w+1} (1+\cos \theta) \\
	&=
	4 \abs{w-1} \abs{w+1} \cos^2 \frac{\theta}{2}.
\end{align*}
Therefore,
\begin{equation} \label{eq:197}
	\sqrt{(\abs{w-1}+\abs{w+1})^2-4} =
	2 \Big| \cos \frac{\theta}{2} \Big|
	\sqrt{\abs{w+1} \abs{w-1}}.
\end{equation}
By computing the area of the triangle $\Delta$ in two ways we get
\[
	\frac{1}{2} \cdot 2 \cdot \Im w = 
	\frac{1}{2} \abs{w - 1} \abs{w + 1} \sin \theta.
\]
Since $\sin \theta = 2 \sin \frac{\theta}{2} \cos \frac{\theta}{2}$ it implies
\[
	\cos \frac{\theta}{2} = 
	\frac{\Im w}{\abs{w - 1} \abs{w + 1} \sin \frac{\theta}{2}}.
\]
By combining it with \eqref{eq:197} we get 
\[
	\sqrt{(\abs{w-1}+\abs{w+1})^2-4} =
	\frac{\Im w}{\sqrt{\abs{w-1} \abs{w+1}} \sin \frac{\theta}{2}},
\]
which in view of \eqref{eq:198} implies \eqref{eq:196}.
\end{proof}

\subsection{Diagonalization of matrices of special form} \label{sec:6.2}
Let us consider a sequence $(Y_n : n \in \NN_0)$ of mappings defined on an open set $U \subset \CC$ with values in
$\GL(2, \CC)$ which is convergent to $\calY$ locally uniformly on $U$, with $\det \calY > 0$. Let $K$ be a compact subset of $U$.
We assume that there is $M \geq 1$ such that for all $z \in K$ and $n \geq M$,
\[
	\discr Y_n(z) \neq 0, \qquad \text{and } \qquad
	\det Y_n(z) \notin (-\infty, 0].
\]
Given $z \in K$ and $n \geq M$, the characteristic equation of the matrix $Y_n(z)$ is
\begin{equation}
	\label{eq:62}
	\lambda^2 - (\tr Y_n(z)) \lambda + \det Y_n(z) = 0.
\end{equation}
If we denote by $\sqrt{z}$ the principle branch of the square root, the equation \eqref{eq:62} has two solutions
\begin{equation}
	\label{eq:63}
	\lambda^+_n(z) = \sqrt{\det Y_n(z)} \cdot \xi_+\bigg(\frac{\tr Y_n(z)}{2 \sqrt{\det Y_n(z)}} \bigg),
	\qquad\text{and}\qquad
	\lambda^-_n(z) = \sqrt{\det Y_n(z)} \cdot \xi_-\bigg(\frac{\tr Y_n(z)}{2 \sqrt{\det Y_n(z)}} \bigg)
\end{equation}
where the functions $\xi_+$ and $\xi_-$ are constructed in Remark \ref{rem:3}.

\begin{proposition} 
	\label{prop:2}
	Let $(Y_n : n \in \NN_0)$ be a sequence of mappings defined on an open set $U \subset \CC$ with values in $\GL(2, \CC)$,
	locally uniformly convergent to $\calY$, with $\det \calY > 0$. Let $K$ be a compact subset of $U$. We assume that there is 
	$M \geq 1$ such that for all $n \geq M$ and $z \in K$,
	\[
		\discr Y_n(z) \neq 0, \qquad \text{and } \qquad \det Y_n(z) \notin (-\infty, 0].
	\]
	If 
	\begin{equation}
		\label{eq:64}
		\sum_{n \geq M} 
		\inf_{z \in K} 
		\bigg( 
		\bigg|\xi_+ \bigg(\frac{\tr Y_n(z)}{2 \sqrt{\det Y_n(z)}}\bigg)\bigg| -1 
		\bigg)= \infty,
	\end{equation}
	then 
	\[
		\inf_{z \in K}
		\prod_{n \geq M}
		\bigg| 
		\frac{\lambda_n^+(z)}{\lambda_n^-(z)}
		\bigg|
		= \infty.
	\]
\end{proposition}
\begin{proof}
	Let us observe that for all $n \geq M$, we have
	\[
		\frac{\lambda_n^+(z)}{\lambda_n^-(z)} = 
		\frac{\big( \lambda^+_n(z) \big)^2}{\det Y_n(z)} =
		\bigg( \xi_+\bigg(\frac{\tr Y_n(z)}{2 \sqrt{\det Y_n(z)}} \bigg) \bigg)^2,
	\]
	hence
	\[
		\prod_{M \leq n \leq m}
		\bigg|\frac{\lambda^+_n(z)}{\lambda^-_n(z)}\bigg|
		=
		\prod_{M \leq n \leq m}
		\bigg|\xi_+\bigg(\frac{\tr Y_n(z)}{2 \sqrt{\det Y_n}}\bigg)\bigg|^2.
	\]
	Since $|\xi_+(w)| \geq 1$ for $w \in \CC$, by generalized Bernoulli's inequality we obtain
	\begin{align*}
		1 +
		\sum_{M \leq n \leq m} \bigg(\bigg|\xi_+\bigg(\frac{\tr Y_n(z)}{2 \sqrt{\det Y_n(z)}}\bigg)\bigg| - 1\bigg)
		&\leq
		\prod_{M \leq n \leq m} \bigg|\xi_+\bigg(\frac{\tr Y_n(z)}{2 \sqrt{\det Y_n(z)}}\bigg)\bigg| \\
		&\leq
		\exp\bigg\{
		\sum_{M \leq n \leq m}
		\bigg(\bigg|\xi_+\bigg(\frac{\tr Y_n(z)}{2 \sqrt{\det Y_n(z)}}\bigg)\bigg| - 1\bigg) \bigg\},
	\end{align*}
	from which the conclusion easily follows.
\end{proof}

We next investigate the case when $\calY = \tilde{\varepsilon} \Id$ with $\tilde{\varepsilon} \in \{-1, 1\}$. We assume that for 
certain sequence $(\tilde{\gamma}_n : n \in \NN_0)$ of positive numbers tending to infinity, the sequence
\begin{equation}
	\label{eq:71}
	R_n = \tilde{\gamma}_n (\tilde{Y}_n - \tilde{\varepsilon} \Id)
\end{equation}
converges to $\calR(z)$ locally uniformly on $U$. Let $K$ be a compact subset of
\[
	\big\{z \in U : \discr \calR(z) \notin (-\infty, 0] \big\}.
\]
There is $M \geq 1$ such that for all $n \geq M$ and $z \in K$, $\discr R_n(z) \neq 0$.
Since
\[
	\discr \tilde{Y}_n(z) = \tilde{\gamma}_n^{-2} \discr R_n(z),
\] 
for each $n \geq M$ and $z \in K$ the matrix $\tilde{Y}_n(z)$ has two distinct eigenvalues. Since $\det \calY = 1$, by possible
increasing $M$, we can guarantee that for all $n \geq M$ and $z \in K$,
\[
	\det \tilde{Y}_n(z) \notin (-\infty, 0].
\]
Hence, the eigenvalues of $\tilde{Y}_n(z)$ are given by formulas \eqref{eq:63}. Moreover, the eigenvalues of the matrix $R_n(z)$
are
\begin{equation}
	\label{eq:73}
	\zeta_n^+(z) = 
		\tilde{\gamma}_n (\lambda_n^+(z) - \tilde{\varepsilon}),
	\quad\text{and}\quad
	\zeta_n^-(z) = 
		\tilde{\gamma}_n (\lambda_n^{-} (z) - \tilde{\varepsilon}).
\end{equation}
In view of Remark \ref{rem:3}, by the next proposition we conclude that the functions $\lambda_n^+$ and $\lambda_n^-$ are 
holomorphic in $\intr(K)$.

\begin{proposition}
	\label{prop:3}
	Let $U$ be a nonempty open subset of $\CC$ such that
	\[
		\text{if} \quad z \in U \quad \text{then}\quad t \Re z + (1-t) z \in U, 
		\quad\text{for all} \quad t \in [0, 1).
	\]
	Suppose that $(\tilde{Y}_n : n \geq 1)$ is a sequence of mappings defined on 
	$U \cup (\cl{U} \cap \RR)$, differentiable on $U$, with values in
	$\GL(2, \CC)$, and such that $\tilde{Y}_n(x) \in \GL(2, \RR)$ and $\det \tilde{Y}_n(x) > 0$ for all $x \in \cl{U} \cap \RR$
	and $n \in \NN$. Assume that for certain sequence $(\tilde{\gamma}_n : n \in \NN_0)$ of positive numbers tending to infinity
	and $\tilde{\varepsilon} \in \{-1, 1\}$, the sequences $(R_n : n \geq 1)$ and $(R_n' : n \geq 1)$, where
	\begin{equation}
		\label{eq:72}
		R_n = \tilde{\gamma}_n(\tilde{Y}_n - \tilde{\varepsilon} \Id),
	\end{equation}
	converges to $\calR$ and $\calR'$ locally uniformly on $U$, respectively. If 
	\begin{equation} 
		\label{eq:72a}
		\lim_{n \to \infty} \tilde{\gamma}_n^2 (\det \tilde{Y}_n)' = 0
	\end{equation}
	locally uniformly on $U$, then
	\begin{equation}
		\label{eq:74}
		\lim_{n \to \infty}
		\tilde{\gamma}_n^2
		\frac{\ud}{\ud z} 
		\bigg( \frac{\tr \tilde{Y}_n}{2 \sqrt{\det \tilde{Y}_n}} \bigg)
		=
		\frac{\tilde{\varepsilon}}{8} (\discr \calR)'
	\end{equation}
	locally uniformly on $U$. Moreover, if $\Re (\discr \calR)' \neq 0$ on $U$ and $K$ is a compact subset $U$,
	then there are $M \geq 1$ and $c > 0$ such that for all $n \geq M$ and $z \in K$,
	\[
		\bigg| \Im \bigg(
		\frac{\tr \tilde{Y}_n(z)}{2 \sqrt{\det \tilde{Y}_n(z)}}
		\bigg) \bigg| \geq c \frac{|\Im z|}{\tilde{\gamma}_n^2}.
	\]
\end{proposition}
\begin{proof}
	A direct computations leads to
	\[
		\bigg( \frac{\tr \tilde{Y}_n}{2 \sqrt{\det \tilde{Y}_n}} \bigg)' =
		\frac{1}{2} \frac{(\tr \tilde{Y}_n') (\det \tilde{Y}_n) - \tfrac{1}{2}(\tr \tilde{Y}_n) 
		(\det \tilde{Y}_n)'}{(\det \tilde{Y}_n)^{3/2}}.
	\]
	Thus, by \eqref{eq:72a} we get
	\begin{equation} \label{eq:72b}
		\lim_{n \to \infty} 
		\tilde{\gamma}_n^2 \bigg( \frac{\tr \tilde{Y}_n}{2 \sqrt{\det \tilde{Y}_n}} \bigg)' =
		\frac{1}{2} \lim_{n \to \infty} \tilde{\gamma}_n^2 \tr \tilde{Y}_n'.
	\end{equation}
	It remains to compute the last limit. By \eqref{eq:72}, we have
	\begin{align*}
		\det R_n 
		&= \tilde{\gamma}_n^2 \det \big(\tilde{Y}_n - \tilde{\varepsilon} \Id\big) \\
		&= \tilde{\gamma}_n^2 \big(\det \tilde{Y}_n - \tilde{\varepsilon} \tr \tilde{Y}_n + 1\big).
	\end{align*}
	Hence,
	\[
		\tilde{\gamma}_n^2 \tr \tilde{Y}_n' = -\tilde{\varepsilon} (\det R_n)' 
		+ \tilde{\varepsilon} \tilde{\gamma}_n^2 (\det \tilde{Y}_n)',
	\]
	and by \eqref{eq:72a},
	\begin{equation} 
		\label{eq:72c}
		\lim_{n \to \infty} \tilde{\gamma}_n^2 \tr \tilde{Y}_n' = -\tilde{\varepsilon} (\det \calR)'
	\end{equation}
	locally uniformly on $U$. Finally, by \eqref{eq:72} we have $\tr R_n' = \tilde{\gamma}_n \tr \tilde{Y}_n'$,
	and consequently, by \eqref{eq:72c}
	\[
		\tr \calR' = \lim_{n \to \infty} \frac{1}{\tilde{\gamma}_n} \tilde{\gamma}_n^2 \tr \tilde{Y}_n' = 0
	\]
	locally uniformly on $U$, which leads to $(\discr \calR)' = -4 (\det \calR)'$. Thus, by \eqref{eq:72c} and \eqref{eq:72b} 
	the formula \eqref{eq:74} follows.

	Next, let $z = x + i y \in K$. By the mean value theorem we have
	\[
		\frac{\tr \tilde{Y}_n(z)}{2 \sqrt{\det \tilde{Y}_n(z)}} - 
		\frac{\tr \tilde{Y}_n(x)}{2\sqrt{\det \tilde{Y}_n(x)}}  = iy 
		\left. \frac{{\rm d}}{{\rm d} z}
		\bigg( \frac{\tr \tilde{Y}_n}{2 \sqrt{\det \tilde{Y}_n}} \bigg) \right|_{z = x+i\xi}
	\]
	where $\xi$ depends on $n$ and belongs to the segment between $0$ and $y$. Since $\tilde{Y}_n(x)$ is real, 
	$\det \tilde{Y}_n(x) > 0$,
	we have
	\begin{equation}
		\label{eq:75}
		\Im \bigg( \frac{\tr \tilde{Y}_n(z)}{2\sqrt{\det \tilde{Y}_n(z)}} \bigg) 
		= (\Im z) \Re \left. \frac{{\rm d}}{{\rm d} z}
		\bigg( \frac{\tr \tilde{Y}_n}{2 \sqrt{\det \tilde{Y}_n}} \bigg) \right|_{z = x+i\xi}.
	\end{equation}
	To complete the proof, it is enough to invoke \eqref{eq:74}. 
\end{proof}

\begin{remark}
	\label{rem:4}
	By \eqref{eq:74} and \eqref{eq:75}, we conclude that the sign of
	\[
		\Im \bigg( \frac{\tr \tilde{Y}_n(z)}{2\sqrt{\det \tilde{Y}_n(z)}} \bigg)
	\]
	is the same as the sign of the product $\tilde{\varepsilon} (\Im z) \Re (\discr \calR)'(z)$.
\end{remark}

\begin{proposition}
	\label{prop:4}
	Suppose that $(\tilde{Y}_n : n \geq 1)$ is a sequence of continuous mappings defined on an open set $U \subset \CC$,
	with values in $\GL(2, \CC)$. Assume that for certain sequence $(\tilde{\gamma}_n : n \in \NN_0)$ of positive
	numbers tending to infinity and $\tilde{\varepsilon} \in \{-1, 1\}$, the sequence
	\[
		R_n(z) = \tilde{\gamma}_n(\tilde{Y}_n(z) - \tilde{\varepsilon} \Id),
	\]
	converges to $\calR(z)$ locally uniformly with respect to $z \in U$. Let $K$ be a compact subset of
	$
		\{z \in U :  \discr \calR(z) \notin (-\infty, 0] \}.
	$
	If there is $M \geq 1$ such that
	\begin{equation} 
		\label{eq:76}
		\bigg|\Im \bigg( \frac{\tr \tilde{Y}_n(z)}{2 \sqrt{\det \tilde{Y}_n(z)}} \bigg) \bigg| > 0, \quad z \in K, n \geq M,
	\end{equation}
	then
	\[
		\lim_{n \to \infty} \zeta_n^+ = 
		\frac{\tr \calR + \tilde{\varepsilon} \sqrt{\discr \calR}}{2}
		\quad \text{and} \quad
		\lim_{n \to \infty} \zeta_n^- = 
		\frac{\tr \calR - \tilde{\varepsilon} \sqrt{\discr \calR}}{2}
	\]
	uniformly on $K$.
\end{proposition}
\begin{proof}
	In view of \eqref{eq:73} it is enough to show that
	\begin{equation}
		\label{eq:125}
		\lim_{n \to \infty}
		\tilde{\gamma}_n(\lambda^+_n(z) - \tilde{\varepsilon}) = \frac{\tr \calR + \tilde{\varepsilon} \sqrt{\discr \calR}}{2}
	\end{equation}
	and
	\begin{equation}
		\label{eq:126}
		\lim_{n \to \infty}
		\tilde{\gamma}_n(\lambda^-_n(z) - \tilde{\varepsilon}) = \frac{\tr \calR - \tilde{\varepsilon} \sqrt{\discr \calR}}{2}.
	\end{equation}
	We shall prove \eqref{eq:125} only; the proof of the assertion \eqref{eq:126} is analogous. By \eqref{eq:63} we have
	\begin{equation}
		\label{eq:77}
		\tilde{\gamma}_n(\lambda^+_n(z) - \tilde{\varepsilon})
		= 
		\tilde{\gamma}_n \big( \sqrt{\det \tilde{Y}_n(z)} \xi_+(w_n) - \tilde{\varepsilon} \big)
	\end{equation}
	where
	\[
		w_n = \frac{\tr \tilde{Y}_n(z)}{2 \sqrt{\det \tilde{Y}_n(z)}}.
	\]
	Let us notice that
	\[
		\tr \tilde{Y}_n = \tilde{\gamma}_n^{-1} \tr R_n + 2 \tilde{\varepsilon},
	\]
	thus
	\[
		\lim_{n \to \infty} w_n = \tilde{\varepsilon}.
	\]
	We write
	\begin{equation} 
		\label{eq:78}
		\lambda^+_n(z) - \tilde{\varepsilon} =
		\big( \sqrt{\det \tilde{Y}_n(z)} - 1 \big) \xi_+(w_n) 
		+
		(\xi_+(w_n) - w_n)
		+ 
		(w_n - \tilde{\varepsilon}),
	\end{equation}
	we need to analyze the three terms on the right-hand side of \eqref{eq:78} separately. Let us begin with the first term.
	Observe that
	\[
		\det \tilde{Y}_n = 
		\tilde{\gamma}_n^{-2} \det R_n + \tilde{\varepsilon} \tilde{\gamma}_n^{-1} \tr R_n + 1,
	\]
	thus
	\begin{align}
		\nonumber
		\lim_{n \to \infty} \tilde{\gamma}_n \big(\sqrt{\det \tilde{Y}_n} - 1\big) 
		&=\lim_{n \to \infty} \tilde{\gamma}_n \big(\det \tilde{Y}_n - 1\big) \big(\sqrt{\det \tilde{Y}_n} + 1\big)^{-1} \\
		\label{eq:79}
		&=\frac{\tilde{\varepsilon}}{2} \tr \calR.
	\end{align}
	Next, to analyze the third term let us recall that
	\[
		\tilde{\gamma}_n^2 (w_n^2 - 1)
		=
		\tilde{\gamma}_n^2
		\frac{\discr \tilde{Y}_n}{4 \det \tilde{Y}_n} \\
		=
		\frac{\discr R_n}{4 \det \tilde{Y}_n}.
	\]
	Thus
	\begin{equation} \label{eq:207}
		\lim_{n \to \infty}
		\tilde{\gamma}_n^2 (w_n^2 - 1)
		=
		\tfrac{1}{4} \discr \calR(z),
	\end{equation}
	and consequently,
	\begin{equation} \label{eq:80}
		\lim_{n \to \infty} \tilde{\gamma}_n^2 (w_n - \tilde{\varepsilon})
		=
		\lim_{n \to \infty} \tilde{\gamma}_n^2 (w_n^2 - 1) \frac{1}{w_n + \tilde{\varepsilon}}
		=
		\frac{\tilde{\varepsilon}}{8} \discr \calR(z).
	\end{equation}
	In particular,
	\begin{equation}
		\label{eq:81}
		\lim_{n \to \infty} \tilde{\gamma}_n (w_n - \tilde{\varepsilon}) =
		\lim_{n \to \infty} \frac{1}{\tilde{\gamma}_n} \tilde{\gamma}_n^2 (w_n - \tilde{\varepsilon}) = 0.
	\end{equation}
	We claim the following holds true.
	\begin{claim}	
		\label{clm:1}
		We have 
		\[
			\lim_{n \to \infty} 
			\tilde{\gamma}_n
			\bigg( \xi_+ \Big( \frac{\tr \tilde{Y}_n}{2 \sqrt{\det \tilde{Y}_n}} \Big) 
			- \frac{\tr \tilde{Y}_n}{2 \sqrt{\det \tilde{Y}_n}} \bigg)
			=
			\tilde{\varepsilon}
			\frac{\sqrt{\discr \calR}}{2}
		\]
		uniformly on $K$.
	\end{claim}	
	First, let us see how one can use Claim \ref{clm:1} to complete the proof of the theorem. We have
	\[
		\lim_{n \to \infty} \tilde{\gamma}_n (\xi_+(w_n) - w_n) = \tilde{\varepsilon} \frac{\sqrt{\discr \calR(z)}}{2}.
	\]
	Now, by putting \eqref{eq:79}, \eqref{eq:81} and \eqref{eq:79} into \eqref{eq:78} we get
	\[
		\lim_{n \to \infty} \tilde{\gamma}_n (\lambda^+_n(z) - \tilde{\varepsilon}) =
		\frac{\tilde{\varepsilon}}{2} \tr \calR(z) \tilde{\varepsilon} + \tilde{\varepsilon} \frac{\sqrt{\discr \calR(z)}}{2}
		=
		\frac{\tr \calR(z) + \tilde{\varepsilon} \sqrt{\discr \calR(z)}}{2},
	\]
	which in view of \eqref{eq:77} implies the conclusion.

	It remains to prove our claim. Let $z \in K$ and suppose that $\tilde{\varepsilon} = 1$. Then by \eqref{eq:76}, 
	$w_n \in \CC \setminus \RR$, and
	\begin{align*}
		\tilde{\gamma}_n \big(\xi_+(w_n) - w_n\big) 
		&= \tilde{\gamma}_n \sqrt{w_n-1}\sqrt{w_n+1} \\
		&= \sqrt{\tilde{\gamma}_n^2(w_n - 1)}{ \sqrt{w_n + 1}},
	\end{align*}
	thus the claim follows from \eqref{eq:80}. If $\tilde{\varepsilon} = -1$, by \eqref{eq:57}, we get
	\begin{align*}
		\lim_{n \to \infty} \tilde{\gamma}_n \exp f(w_n)
		&= - \lim_{n \to \infty} \tilde{\gamma}_n \exp f(-w_n) \\
		&= -\sqrt{\discr \calR(z)}.
	\end{align*}
	This completes the proof of Claim \ref{clm:1} and the theorem follows.
\end{proof}

Thanks to Proposition \ref{prop:4}, for each $z \in K$ the eigenvalues of $\calR(z)$ are
\[
	\zeta^+(z) = \frac{\tr \calR(z) + \tilde{\varepsilon} \sqrt{\discr \calR(z)}}{2}
	\quad\text{and}\quad
	\zeta^-(z) = \frac{\tr \calR(z) - \tilde{\varepsilon} \sqrt{\discr \calR(z)}}{2}.
\]

\begin{proposition} 
	\label{prop:9}
	Suppose that the hypotheses of Proposition~\ref{prop:3} are satisfied. Let $K$ be a compact subset of
	\[
		\big\{z \in U : \discr \calR \in \CC \setminus (-\infty, 0] \text{ and }
		\Re (\discr \calR)'(z) \neq 0 \big\}.
	\]
	Then
	\begin{equation}
		\label{eq:112}
		\lim_{j \to \infty} \lambda_j^+(z) = \tilde{\varepsilon}
	\end{equation}
	uniformly with respect to $z \in K$. Moreover, if
	\begin{equation}
		\label{eq:67}
		\lim_{n \to \infty} \tilde{\gamma}_n^2 (\det \tilde{Y}_n)' = 0,
	\end{equation}
	uniformly on $K$, then
	\begin{equation}
		\label{eq:113}
		\lim_{j \to \infty}
		\tilde{\gamma}_j
		\frac{(\lambda_j^+)'(z)}{\lambda_j^+(z)}  =
		\frac{1}{4} \frac{(\discr \calR)'(z)}{\sqrt{\discr \calR(z)}}
	\end{equation}
	uniformly with respect to $z \in K$.
\end{proposition}
\begin{proof}
	By Proposition \ref{prop:3}, the condition \eqref{eq:76} is satisfied. Hence, the equation \eqref{eq:112} immediately 
	follows from \eqref{eq:73} and \eqref{eq:125}.

	To prove the second assertion, let us recall that by \eqref{eq:63} we have
	\[
		\lambda_j^+(z) = \sqrt{\det \tilde{Y}_j(z)} \xi_+ \big( w_j(z) \big)
	\]
	where
	\[
		w_j(z) = \frac{\tr \tilde{Y}_j(z)}{2 \sqrt{\det \tilde{Y}_j(z)}}.
	\]
	Hence,
	\begin{align}
		\nonumber
		\frac{(\lambda_j^+)'(z)}{\lambda_j^+(z)} 
		&=
		\frac{(\det \tilde{Y}_j)'(z)}{2 \det \tilde{Y}_j(z)} +
		\frac{\xi_+' \big( w_j(z) \big)}{\xi_+ \big( w_j(z) \big)} w_j'(z) \\
		\label{eq:114}
		&=
		\frac{(\det \tilde{Y}_j)'(z)}{2 \det \tilde{Y}_j(z)} +
		\frac{w_j'(z)}{\xi_+ \big( w_j(z) \big) - w_j(z)}
	\end{align}
	where in the last equality we used \eqref{eq:60}. We need to analyze the two terms on the right-hand side of \eqref{eq:114}.
	The first term is easy, namely by \eqref{eq:67}
	\[
		\lim_{j \to \infty} \tilde{\gamma}_j (\det \tilde{Y}_j)' =
		\lim_{j \to \infty} \frac{1}{\tilde{\gamma}_j} \tilde{\gamma}_j^2 (\det \tilde{Y}_j)' = 0
	\]
	uniformly on $K$. Therefore, in view of \eqref{eq:114}
	\begin{equation} 
		\label{eq:115}
		\lim_{j \to \infty}
		\tilde{\gamma}_j
		\frac{(\lambda_j^+)'(z)}{\lambda_j^+(z)} 
		=
		\lim_{j \to \infty}
		\frac{\tilde{\gamma}_j^2}{\tilde{\gamma}_j}
		\frac{w_j'(z)}{\xi_+ \big( w_j(z) \big) - w_j(z)}.
	\end{equation}
	By \eqref{eq:74} we get
	\begin{equation}
		\label{eq:116}
		\lim_{j \to \infty}
		\tilde{\gamma}_j^2 w_j'(z) =
		\frac{\tilde{\varepsilon}}{8} (\discr \calR)'(z).
	\end{equation}
	By Claim~\ref{clm:1} we have
	\begin{equation} 
		\label{eq:117}
		\lim_{j \to \infty}
		\tilde{\gamma}_j
		\Big( \xi_+ \big( w_j(z) \big) - w_j(z) \Big) =
		\tilde{\varepsilon} \frac{\sqrt{\discr \calR(z)}}{2}.
	\end{equation}
	By putting \eqref{eq:116} and \eqref{eq:117} into \eqref{eq:115} leads to \eqref{eq:113}.
\end{proof}

\subsection{Applications to periodic modulations} \label{sec:6.3}
In this section we separately study diagonalization procedure for $N$-transfer matrix of $N$-periodically
modulated Jacobi parameters depending on the form of the limit matrix $\frakX_0$.
We start with the following observation.
\begin{lemma}
	\label{lem:6}
	Suppose that $(a_n : n \in \NN_0)$ and $(b_n : n \in \NN_0)$ are $N$-periodically modulated Jacobi parameters
	such that $\tr \frakX_0'(0) \neq 0$.
	Let $K$ be a compact subset of $\CC_\varsigma$ with nonempty interior where $\varsigma = \sign{\tr \frakX_0'(0)}$. 
	Then there is $M \geq 1$ such that for all $n \geq M$ and $z \in K$,
	\[
		\Im \bigg(\frac{\tr X_n(z)}{2 \sqrt{\det X_n(z)}} \bigg) 
		> 0.
	\]
\end{lemma}
\begin{proof}
	Let $z = x+iy \in K$. For a fixed $k \in \{0, 1, \ldots, N-1\}$ we set
	\[
		Y_j = X_{jN+k}, \quad j \in \NN_0.
	\]
	By the mean value theorem we have
	\[
		\tr Y_n(x+iy) - \tr Y_n(x) = iy \tr Y_n'(x+i\xi)
	\]
	where $\xi$ depends on $x$ and $n$, and belongs to $(0, y)$. Since $\tr Y_n(x)$ is real, 
	\begin{equation}
		\label{eq:88}
		\Im \big( \tr Y_n(x+iy) \big) = y \Re \big( \tr Y_n'(x+i\xi) \big).
	\end{equation}
	Now, by \cite[Corollary 3.10]{ChristoffelI} we have
	\begin{equation}
		\label{eq:89}
		\lim_{n \to \infty} \frac{a_{nN}}{\alpha_0} \tr Y_n'(z) = 
		\tr \frakX'_0(0) \in \RR \setminus \{0\}
	\end{equation}
	locally uniformly with respect to $z \in \CC$. Hence there is $M \geq 1$ such that for all $n \geq M$, 
	\[
		\sign{\Re (\tr Y_n'(x+i \xi))} = \sign{\tr \frakX'_0(0)},
	\]
	and the lemma follows.
\end{proof}

\subsubsection{Case~\ref{eq:PI}: $|\tr \frakX_0(0)| < 2$}
\label{sec:6.3.1}
We set
\[
	Y_n = X_{nN}, \quad n \geq 0.
\]
Thanks to Lemma \ref{lem:6}, if $K$ is a compact subset of $\CC_\varsigma$ with nonempty interior, where
$\varsigma = \sign{\tr \frakX_0'(0)}$, then the functions $\lambda_n^+$ and $\lambda_n^-$ are continuous on $K$ and holomorphic 
on $\intr(K)$. In the following proposition we compute the limit as $n$ tends to infinity.
\begin{proposition} 
	\label{prop:6}
	Suppose that $(a_n : n \in \NN_0)$ and $(b_n : n \in \NN_0)$ are $N$-periodically modulated Jacobi parameters such that
	$\tr \frakX_0(0) \in (-2,2)$. Let $K$ be a compact subset of $\CC_\varsigma$ with nonempty interior where
	$\varsigma = \sign{\tr \frakX_0'(0)}$. Then
	\begin{equation}
		\label{eq:90}
		\lim_{n \to \infty}
		\lambda_n^+(z)
		=
		\frac{\tr \frakX_0(0) + i \sqrt{-\discr \frakX_0(0)}}{2}
	\end{equation}
	uniformly with respect to $z \in K$. Moreover,
	\begin{equation}
		\label{eq:91}
		\lim_{n \to \infty}
	   	\frac{a_{nN}}{\alpha_0} \frac{(\lambda_n^+)'(z)}{\lambda_n^+(z)}
		=
		\frac{\tr \frakX_0'(0)}{i \sqrt{-\discr \frakX_0(0)}}
	\end{equation}
	locally uniformly with respect to $z \in \intr(K)$.
\end{proposition}
\begin{proof}
	Observe that, by Lemma \ref{lem:6}, there is $M > 0$ such that for all $n \geq M$ and $z \in K$,
	\begin{equation}
		\label{eq:92}
		\Im \big(\lambda_n^+(z)\big) > 0.
	\end{equation}
	Indeed, by \eqref{eq:63} we have
	\begin{equation}
		\label{eq:93}
		\lambda_n^+(z) = \sqrt{\det Y_n(z)} \xi_+ \big( w_n(z) \big),
	\end{equation}
	where
	\[
		w_n(z) = \frac{\tr Y_n(z)}{2 \sqrt{\det Y_n(z)}}.
	\]
	Since
	\begin{equation}
		\label{eq:94}
		\det Y_n = \frac{a_{nN-1}}{a_{nN+N-1}} > 0,
	\end{equation}
	we get
	\[
		\Im \big( \lambda_n^+(z) \big) = 
		\sqrt{\det Y_n(z)} 
		\Im 
		\Big( 
			\xi_+ \big( w_n(z) \big) 
		\Big).
	\]
	Hence, by Lemma \ref{lem:6} and \eqref{eq:61} we conclude \eqref{eq:92}.
	
	Next, let us notice that the limit
	\begin{equation}
		\label{eq:95}
		\lim_{n \to \infty} \lambda_n^+(z) = 
		\lim_{n \to \infty}
		\xi_+ \big( w_n(z) \big) 
	\end{equation}
	exists and satisfies 
	\[
		\lambda^2 - \big( \tr \frakX_0(0) \big) \lambda + 1 = 0
	\]
	which is the characteristic equation of $\frakX_0(0)$. The equation has two solutions
	\[
		\frac{\tr \frakX_0(0) - i \sqrt{-\discr \frakX_0(0)}}{2}
		\quad\text{and}\quad
		\frac{\tr \frakX_0(0) + i \sqrt{-\discr \frakX_0(0)}}{2}.
	\]
	In view of \eqref{eq:92}, we must have
	\[
		\lim_{n \to \infty} \lambda_n^+(z) = 
		\frac{\tr \frakX_0(0) + i \sqrt{-\discr \frakX_0(0)}}{2}
	\]
	as claimed in \eqref{eq:90}.
	
	Let us proceed to the proof of \eqref{eq:91}. By \eqref{eq:93}, \eqref{eq:94} and \eqref{eq:60} we get
	\begin{equation}
		\label{eq:96}
		\frac{(\lambda_n^+)'(z)}{\lambda_n^+(z)} =
		\frac{\xi_+' \big( w_n(z) \big)}{\xi_+ \big( w_n(z) \big)} w_n'(z) =
		\frac{w_n'(z)}{\xi_+ \big( w_n(z) \big) - w_n(z)}.
	\end{equation}
	Now, by \eqref{eq:95} and \eqref{eq:90} we have
	\begin{equation}
		\label{eq:97}
		\lim_{n \to \infty} 
		\Big( \xi_+ \big( w_n(z) \big) - w_n(z) \Big) =
		\frac{i}{2} \sqrt{-\discr \frakX_0(0)}.
	\end{equation}
	Next, by \cite[Corollary 3.10]{ChristoffelI}
	\[
		\lim_{n \to \infty} \frac{a_{nN}}{\alpha_0} w_n'(z) =
		\frac{\tr \frakX'_0(0)}{2}.
	\]
	Thus, by putting it and \eqref{eq:97} into \eqref{eq:96} we obtain \eqref{eq:91}.
\end{proof}

\begin{proposition} 
	\label{prop:7}
	Suppose that $(a_n : n \in \NN_0)$ and $(b_n : n \in \NN_0)$ are $N$-periodically modulated Jacobi parameters such that 
	$\tr \frakX_0(0) \in (-2,2)$. If the Carleman's condition is satisfied, then for any compact $K \subset \CC_\varsigma$
	where $\varsigma = \sign{\tr \frakX_0'(0)}$, there exists $M \geq 1$ such that
	\[
		\inf_{z \in K}
		\prod_{j \geq M}		
		\bigg| \frac{\lambda_{j}^+(z)}{\lambda_{j}^-(z)}\bigg| = \infty.
	\]
\end{proposition}
\begin{proof}
	The conclusion follows by Proposition \ref{prop:2}, provided we can justify \eqref{eq:64}. To do so, let us define
	\[
		w_j(z) = \frac{\tr Y_j(z)}{2 \sqrt{\det Y_j(z)}}, \quad z \in K.
	\]
	Notice that
	\[
		\lim_{j \to \infty} w_j(z) = \frac{\tr \frakX_0(0)}{2} \in (-1,1)
	\]
	uniformly with respect to $z \in K$. Thus, in view of \eqref{eq:196}, and $\det Y_j = \frac{a_{jN-1}}{a_{jN+N-1}} \to 1$, 
	it is enough to prove that for some $M \geq 1$
	\begin{equation} \label{eq:200}
		\sum_{j \geq M} \inf_{z \in K} |\Im \tr Y_j(z)| = \infty.
	\end{equation}
	To do so, we use \eqref{eq:88} and \eqref{eq:89} to conclude that there
	are $M \geq 1$ and $c > 0$ such that for all $z \in K$ and $j \geq M$,
	\[
		 | \Im \tr Y_j(z)| \geq c \frac{1}{a_{jN}}.
	\]
	By \eqref{eq:int:13}, Carleman's condition implies \eqref{eq:200}, what we needed to show.
\end{proof}

We are now in the position to introduce diagonalization procedure of $Y_n$. Since orthogonal polynomials
have only real zeros, by \eqref{eq:4} we get
\[
	[Y_n(z)]_{12} \neq 0
\]
for all $z \in \CC$ and $n \geq 1$. Let $K$ be a compact subset of $\CC_\varsigma$ with nonempty interior
where $\varsigma = \sign{\tr \frakX_0'(0)}$. Since $\discr \frakX_0(0) < 0$, there are $M \geq 1$ and $\delta > 0$ such 
that for all $n \geq M$ and $z \in K$,
\begin{equation}
	\label{eq:99}
	|\discr Y_n(z)| > \delta.
\end{equation}
Then for each $z \in K$ and $n \geq M$ the matrix $Y_n(z)$ can be diagonalized
\[
	Y_n(z) = C_n(z) D_n(z) C_n^{-1}(z)
\]
where
\begin{equation}
	\label{eq:101}
	C_n(z) =
	\begin{pmatrix}
		1 & 1 \\
		\frac{\lambda^+_n(z) - [Y_n(z)]_{11}}{[Y_n(z)]_{12}} & \frac{\lambda^-_n(z) - [Y_n(z)]_{11}}{[Y_n(z)]_{12}}
	\end{pmatrix}, \qquad
	D_n(z) =
	\begin{pmatrix}
		\lambda^+_n(z) & 0 \\
		0 & \lambda^-_n(z)
	\end{pmatrix}.
\end{equation}
To see this, we notice that \eqref{eq:99} implies that $|\lambda_n^+(z)| > |\lambda_n^-(z)|$, thus $C_n(z)$ is 
invertible. Moreover, if $\lambda$ satisfies \eqref{eq:62}, then
\[
	\big( Y_n(z) - \lambda \Id \big)
	\begin{pmatrix}
		1 \\
		\frac{\lambda - [Y_n(z)]_{11}}{[Y_n(z)]_{12}}
	\end{pmatrix}
	= 
	\begin{pmatrix}
		0 \\
		\frac{-\lambda^2 + \tr Y_n(z) \lambda - \det Y_n(z)}{[Y_n(z)]_{12}}
	\end{pmatrix}
	= 
	\begin{pmatrix}
		0 \\
		0
	\end{pmatrix}.
\]
Lastly, the mappings $C_n(z)$ and $D_n(z)$ are continuous on $K$ and holomorphic on $\intr(K)$.
\begin{lemma} 
	\label{lem:7}
	Suppose that $(a_n : n \in \NN_0)$ and $(b_n : n \in \NN_0)$ are $N$-periodically modulated Jacobi parameters such that 
	$\tr \frakX_0(0) \in (-2,2)$. Let $K$ be a compact subset of $\CC_\varsigma$, where $\varsigma = \sign{\tr \frakX_0'(0)}$,
	with nonempty interior. If
	\begin{equation} 
		\label{eq:102}
		\bigg( \frac{a_{n-1}}{a_n} : n \in \NN \bigg),
		\bigg( \frac{b_n}{a_n} : n \in \NN_0 \bigg),
		\bigg( \frac{1}{a_n} : n \in \NN_0 \bigg) \in \calD_1^N,
	\end{equation}
	then there is $M \geq 1$ such that
	\[
		(C_n : n \geq M) \in \calD_1 \big( K, \GL(2, \CC) \big).
	\]
\end{lemma}
\begin{proof}
	By \eqref{eq:102} we have
	\[
		(B_n : n \geq 0) \in \calD_1^N \big( K, \GL(2, \CC) \big).
	\]
	Hence, by \cite[Corollary 1(i)]{SwiderskiTrojan2019}
	\begin{equation}
		\label{eq:103}
		(X_n : n \geq 0) \in \calD_1^N \big( K, \GL(2, \CC) \big).
	\end{equation}
	In particular, 
	\[
		(Y_n : n \geq 0) \in \calD_1 \big( K, \GL(2, \CC) \big),
	\]
	and by \cite[Lemma 2]{SwiderskiTrojan2019},
	\[
		(\tr Y_n : n \geq 0), (\det Y_n : n \geq 0) \in \calD_1(K, \CC).
	\]
	Since $\det Y_n$ locally uniformly approaches $1$, we have
	\[
		\bigg( \frac{\tr Y_n}{2 \sqrt{\det Y_n}} : n \geq 0 \bigg) \in \calD_1(K, \CC).
	\]
	Now, by Remark \ref{rem:3} and  \cite[Lemma 2]{SwiderskiTrojan2019} we obtain
	\begin{equation} 
		\label{eq:104}
		\big( \lambda_{n}^+ : n \geq M \big), \big( \lambda_{n}^- : n \geq M \big) \in \calD_1(K, \CC).
	\end{equation}
	Let us recall that for any $i \in \{0, 1, \ldots, N-1 \}$
	\[
		\lim_{n \to \infty} Y_{n}(z) = \frakX_0(0)
	\]
	uniformly with respect to $z \in K$. Since $\discr \frakX_0(0) < 0$ we get $[\frakX_0(0)]_{12} \neq 0$. Hence, in view of 
	\eqref{eq:101}, \eqref{eq:103}, and \eqref{eq:104} the conclusion follows.
\end{proof}

\subsubsection{Case~\ref{eq:PIIa}: $\frakX_0(0) = \varepsilon \Id$ for some $\varepsilon \in \{-1,1\}$}
\label{sec:6.3.2}
Let
\[
	Y_j = X_{jN}, \quad j \geq 0.
\]
Suppose that the sequence $(R_j : j \in \NN_0)$ where
\begin{equation}
	\label{eq:105}
	R_j = a_{jN+N-1}(Y_j - \varepsilon \Id)
\end{equation}
converges to $\calR_{0}$ locally uniformly on $\CC$. Observe that we are in the setup \eqref{eq:71} with
$\tilde{\gamma}_j = a_{jN+N-1}$, $\tilde{\varepsilon} = \varepsilon$ and $\tilde{Y}_j = Y_j$. By \cite[Corollary 1]{PeriodicIII}, 
$\discr \calR_0$ is a concave quadratic polynomial with real coefficients. 
The proof of the next result is based on the proof of \cite[Lemma 2]{Geronimo1986}.
\begin{proposition}
	\label{prop:11}
	$\discr \calR_0(z)$ is a quadratic polynomial with real roots.
\end{proposition}
\begin{proof}
	Indeed, by \cite[Corollary 1]{PeriodicIII}, $\discr \calR_0(z)$ is a quadratic polynomial. Moreover, by \eqref{eq:105}, 
	we have $\discr R_{jN} = a_{(j+1)N-1}^{-2} \discr X_{jN}$, thus
	\[
		\discr \calR_0(z) = \lim_{j \to \infty} a_{(j+1)N-1}^{-2} \discr X_{jN}(z).
	\]
	locally uniformly with respect to $z \in \CC$. Observe that $\discr X_{jN}(z)$ is a real polynomial of degree $2N$.
	In view of Hurwitz' theorem, to conclude the proof it is enough to show that $\discr X_{jN}$ has only real zeros. 
	By \eqref{eq:122},
	\[
		\det X_n = \frac{a_{n-1}}{a_{n+N-1}},
	\]
	thus by \eqref{eq:4} we get
	\begin{align*}
		\discr X_n 
		&= \Big(p_N^{[n]} - \frac{a_{n-1}}{a_n} p_{N-2}^{[n+1]} \Big)^2
		- 4 \frac{a_{n-1}}{a_{n+N-1}},
	\end{align*}
	Let 
	\[
		Q^-_n(z) = p_N^{[n]}(z) - \frac{a_{n-1}}{a_n} p_{N-2}^{[n+1]}(z) - 2 \sqrt{\frac{a_{n-1}}{a_{n+N-1}}}
	\]
	and
	\[
		Q^+_n(z) = p_N^{[n]}(z) - \frac{a_{n-1}}{a_n} p_{N-2}^{[n+1]}(z) + 2 \sqrt{\frac{a_{n-1}}{a_{n+N-1}}}.
	\]
	Then the set of zeros of $\discr X_n$ is contained in the union 
	\[	
		\big\{z \in \CC : Q^-_n(z) = 0\big\} \cup \big\{z \in \CC  : Q^+_n(z) = 0\big\}.
	\]
	Let us consider the polynomial $Q^-_n$. Suppose that $p_{N-1}^{[n]}(x_0) = 0$ for certain $x_0 \in \RR$. Then by \eqref{eq:4},
	\[
		\frac{a_{n-1}}{a_{n+N-1}} 
		=
		\det X_n(x_0)
		= -\frac{a_{n-1}}{a_n} p_{N-2}^{[n+1]}(x_0) p_N^{[n]} (x_0),
	\]
	that is
	\[
		\frac{a_{n-1}}{a_n} p_{N-2}^{[n+1]}(x_0) = - \frac{a_{n-1}}{a_{n+N-1}} \frac{1}{p_N^{[n]}(x_0)}
	\]
	and so
	\begin{align*}
		Q^-_n(x_0) 
		&= p_N^{[n]}(x_0)  + \frac{a_{n-1}}{a_{n+N-1}} \frac{1}{p_N^{[n]}(x_0)} - 2 \sqrt{\frac{a_{n-1}}{a_{n+N-1}}}  \\
		&= \frac{1}{p_N^{[n]}(x_0)}\Big( p_N^{[n]}(x_0) - \sqrt{\frac{a_{n-1}}{a_{n+N-1}}} \Big)^2.
	\end{align*}
	Let us recall that for a fixed $n \in \NN$, the polynomial $p_{N-1}^{[n]}$ has $N-1$ real and simple zeros, say
	$x_{1}, \ldots, x_{N-1}$. Moreover, for each $j = 1, \ldots, N-1$ we have $\sign{p_{N}^{[n]}(x_j)} = (-1)^{j+N}$. 
	Hence, $\sign{Q^-_n(x_j)} = (-1)^{j+N}$, for each $j = 1, \ldots, N-1$. Because $Q^-_n$ has the leading coefficient positive
	and $\sign{ Q^-_n(x_{N-1})} < 0$, we conclude that $Q^-_n$ has $N-1$ real zeros. Since the coefficients of $Q^-_n$ are real,
	it must have $N$ real zeros. Analogously, we prove that $Q^+_n$ has $N$ real zeros. Consequently, all zeros
	of $\discr X_n$ are real, and the proposition follows.
\end{proof}

Next, by Remark \ref{rem:4}, for each $K$ a compact subset of 
\[
	\big\{z \in \CC \setminus \RR : \Re (\discr \calR_0)'(z) \neq 0 \big\},
\]
there is $M \geq 1$ such that for all $j \geq M$ and $z \in K$,
\[
	\mathrm{sign} \bigg( \Im\bigg(\frac{\tr Y_j(z)}{\sqrt{\det Y_j(z)}} \bigg) \bigg) = 
	\mathrm{sign}\Big( \varepsilon (\Im z) \Re (\discr \calR_0)'(z) \Big).
\]
Hence, if $K$ is a compact subset of
\[
	U = \big\{z \in \CC_+ : \discr \calR_{0}(z) \in \CC \setminus (-\infty, 0] 
	\text{ and }  \varepsilon \Re (\discr \calR_0)'(z) > 0 \big\},
\]
then for all $j \geq M$ and $z \in K$,
\[
	\frac{\tr Y_j(z)}{\sqrt{\det Y_j(z)}} \in \CC_+.
\]
By Proposition \ref{prop:3}, the eigenvalues of $Y_j$ are given by formulas \eqref{eq:63} and 
both are continuous on $K$ and holomorphic in $\intr(K)$ provided that $j \geq M$. Moreover, the eigenvalues of $R_j$ 
are given by \eqref{eq:73} and both are continuous on $K$ and holomorphic in $\intr(K)$, if $j \geq M$.

Since $[Y_n]_{12}$ is a polynomial having only real zeros, we must have $[R_n]_{12} \neq 0$ on $K$. Hence, for each 
$n \geq M$ and $z \in K$, we can write
\[
	R_n(z) = C_n(z) \tilde{D}_n(z) C_n^{-1}(z)
\]
where
\[
	C_n(z) = 
	\begin{pmatrix}
		1  & 1 \\
		\frac{\zeta_n^+(z) - [R_n(z)]_{11}}{[R_n(z)]_{12}} &
		\frac{\zeta_n^-(z) - [R_n(z)]_{22}}{[R_n(z)]_{12}}
	\end{pmatrix}
	,\quad\text{and}\quad
	\tilde{D}_n(z) =
	\begin{pmatrix}
		\zeta^+_n(z) & 0\\
		0 & \zeta^-_n(z)
	\end{pmatrix}.
\]
Thus setting
\[
	D_n = \begin{pmatrix}
		\lambda^+_n & 0 \\
		0 & \lambda^-_n
	\end{pmatrix},
\]
we obtain
\[
	Y_n(z) = C_n(z) D_n(z) C_n^{-1}(z).
\]
Let us observe that $C_n$ and $D_n$ are continuous on $K$ and holomorphic mappings on $\intr(K)$.

Next, from \cite[Proposition 9]{PeriodicIII} we easily deduce the following lemma.
\begin{lemma}
	\label{lem:8}
	Suppose that $(a_n : n \in \NN_0)$ and $(b_n : n \in \NN_0)$ are $N$-periodically modulated Jacobi parameters such that	
	$\frakX_0(0) = \varepsilon \Id$ with $\varepsilon \in \{-1, 1\}$. Assume that
	\[
		\bigg( a_n \Big( \frac{\alpha_{n-1}}{\alpha_n} - \frac{a_{n-1}}{a_n} \Big) : n \in \NN \bigg),
		\bigg( a_n \Big( \frac{\beta_n}{\alpha_n} - \frac{b_n}{a_n} \Big): n \in \NN_0 \bigg),
		\bigg( \frac{1}{a_n} : n \in \NN_0 \bigg) \in \calD_1^N.
	\]
	Then the sequence $(R_{n} : n \in \NN)$ where
	\[
		R_n(z) = a_{nN+N-1}(Y_n(z) - \varepsilon \Id),
	\]
	converges locally uniformly with respect to $z \in \CC$. Moreover, for each $K$ a compact subset of
	\[
		U = \big\{z \in \CC_+ : \discr \calR_{0}(z) \in \CC \setminus (-\infty, 0] 
		\text{ and }  \varepsilon \Re (\discr \calR_0)'(z) > 0 \big\},
	\]
	there is $M \geq 1$ such that
	\[
		(C_n : n \geq M) \in \calD_1 \big( K, \Mat(2, \CC) \big).
	\]
\end{lemma}

\subsubsection{Case~\ref{eq:PIIb}: $\frakX_0(0)$ is a non-trivial parabolic element}
\label{sec:6.3.3}
Let $\gamma = (\gamma_n : n \in \NN_0)$ be an auxiliary positive sequence tending to infinity satisfying \eqref{eq:83a}
and \eqref{eq:83b}. Recall that Jacobi parameters $(a_n : n \in \NN_0)$ and $(b_n : n \in \NN_0)$ are $\gamma$-tempered 
$N$-periodically modulated such that $\frakX_0(0)$ is a non-trivial parabolic element if
\begin{equation}
	\label{eq:108a}
	\bigg( \sqrt{\gamma_n} \Big(\frac{\alpha_{n-1}}{\alpha_n} - \frac{a_{n-1}}{a_n}\Big) : n \in \NN \bigg),
	\bigg( \sqrt{\gamma_n} \Big(\frac{\beta_n}{\alpha_n} - \frac{b_n}{a_n}\Big) : n \in \NN_0 \bigg),
	\bigg( \frac{\gamma_n}{a_n} : n \in \NN_0 \bigg) \in \calD_1^N.
\end{equation}
We additionally assume that
\begin{equation}
	\label{eq:108b}
	\bigg(
	\gamma_n\big(1- \varepsilon [\frakX_n(0)]_{11}\big)\Big(\frac{\alpha_{n-1}}{\alpha_n} - \frac{a_{n-1}}{a_n}\Big)
	-
	\gamma_n \varepsilon \Big(\frac{\beta_n}{\alpha_n} - \frac{b_n}{a_n}\Big)
	:  n \in \NN \bigg) \in \calD_1^N
\end{equation}
where
\[
	\varepsilon = \sign{\tr\frakX_0(0)}.
\]
By \eqref{eq:108a}, there are two $N$-periodic sequences $(\fraks_n : n \in \ZZ)$ and $(\frakr_n : n \in \ZZ)$ such that
\[
	\lim_{n \to \infty} \bigg|\sqrt{\alpha_n \gamma_n} \Big(\frac{\alpha_{n-1}}{\alpha_n} - \frac{a_{n-1}}{a_n}\Big) -
	\fraks_n \bigg| = 0,
	\qquad\text{and}\qquad
	\lim_{n \to \infty} \bigg|\sqrt{\alpha_n \gamma_n} \Big(\frac{\beta_n}{\alpha_n} - \frac{b_n}{a_n}\Big) - \frakr_n\bigg| = 0.
\]
Moreover, there is $\mathfrak{t} \in \{0, 1\}$ such that
\[
	\lim_{n \to \infty} \frac{\gamma_n}{a_n} = \mathfrak{t}.
\]
Lastly, by \eqref{eq:108b} there is $N$-periodic sequence $(\fraku_n : n \in \ZZ)$ such that
\begin{equation}
	\lim_{n \to \infty}
	\Big|
	\gamma_n\big(1- \varepsilon [\frakX_n(0)]_{11}\big)\Big(\frac{\alpha_{n-1}}{\alpha_n} - \frac{a_{n-1}}{a_n}\Big)
	-
	\gamma_n \varepsilon \Big(\frac{\beta_n}{\alpha_n} - \frac{b_n}{a_n}\Big) - \fraku_n
	\Big| = 0.
\end{equation}
Let us define
\begin{equation}
	\label{eq:189}
	\tau(z) = \frac{1}{4} \mathfrak{S}^2 - \upsilon(z)
\end{equation}
where
\[
	\upsilon(z) = \mathfrak{t} \varepsilon z (\tr \frakX_0'(0)) - \mathfrak{U},
\]
and
\[
	\mathfrak{U} = \sum_{i' = 0}^{N-1} \frac{\fraku_{i'}}{\alpha_{i'-1}},
	\qquad
	\mathfrak{S} = \sum_{i' = 0}^{N-1} \frac{\fraks_{i'}}{\alpha_{i'-1}}.
\]
Since $\frakX_0(0)$ in a nontrivial parabolic element of $\SL(2, \RR)$, we have $\tr \frakX_0'(0) \neq 0$, see 
\cite[Proposition 2.1]{jordan2}. 

If $\mathfrak{t} = 1$, then 
\[
	x_0 = - \varepsilon \frac{\mathfrak{U} + \frac{1}{4} \mathfrak{S}^2}{\tr \frakX_0'(0)}
\]
is the only solution to $\tau(z) = 0$. Moreover, with no lose of generality we can assume that $\gamma_n \equiv a_n$.
We set
\begin{equation} \label{eq:204}
	U = \big\{z \in \CC : \Re \tau(z) < 0 \big\}
\end{equation}
and
\[
	\vartheta_j(z) = 
	\sqrt{\frac{\alpha_{N-1}}{\gamma_{jN+N-1}}} \sqrt{-\tau(z)}, \quad \text{ for } z \in U.
\]
If $\mathfrak{t} = 0$, then $\tau(z) \neq 0$ for all $z \in \CC$. Let $U = \CC$,
\[
	\vartheta_j(z) = \sqrt{\frac{\alpha_{N-1}}{\gamma_{jN+N-1}} |\tau(0)|}, \quad \text{ for } z \in U.
\]
In both cases we set
\[
	Z_j = 
	T_{0}
	\begin{pmatrix}
		1 & 1 \\
		\ue^{\vartheta_j} & \ue^{-\vartheta_j}
	\end{pmatrix}
\]
where $T_0$ is an invertible matrix such that
\[
	\frakX_0(0) = \varepsilon T_0 
	\begin{pmatrix}
		0 & 1 \\
		-1 & 2
	\end{pmatrix}
	T_0^{-1}.
\]
Finally, for $z \in U$, we define
\[
	Y_j(z) = Z_{j+1}^{-1}(z) X_{jN}(z) Z_j(z), \quad j \geq 1.
\]
Given $K$ a compact subset of $U$, we follow the proof of \cite[Theorem 3.1]{jordan2} to show that
\begin{equation}
	\label{eq:26}
	Z_j^{-1} Z_{j+1} = 
	\Id 
	+ 
	\sqrt{\frac{\alpha_{N-1}}{\gamma_{jN+N-1}}} Q_j
\end{equation}
where $(Q_j)$ is a sequence of holomorphic mappings from $\calD_1(K, \Mat(2, \CC))$ convergent uniformly on $K$ to zero.
Moreover, by following the proof of \cite[Theorem 3.2]{jordan2}, one can prove that
\begin{equation} \label{eq:138}
	Y_j = 
	Z_{j+1}^{-1} X_{jN} Z_j = 
	\varepsilon \bigg(\Id +  \sqrt{\frac{\alpha_{N-1}}{\gamma_{jN+N-1}}} R_j\bigg)
\end{equation}
where $(R_j)$ is a sequence of holomorphic mappings from $\calD_1(K, \Mat(2, \CC))$ convergent uniformly on $K$ to
\begin{equation}
	\label{eq:48}
	\calR_{0}(z) =
	\frac{\sqrt{-\tau(z)}}{2}
	\begin{pmatrix}
		1 & -1 \\
		1 & -1
	\end{pmatrix}
	-
	\frac{\upsilon(z)}{2 \sqrt{-\tau(z)}}
	\begin{pmatrix}
		1 & 1 \\
		-1 & -1
	\end{pmatrix}
	-
	\frac{\mathfrak{S}}{2}
	\begin{pmatrix}
		1 & -1\\
		-1 & 1
	\end{pmatrix}.
\end{equation}
In particular, $\discr \calR_{0}(z) = 4 \tau(z)$ and $\det \calR_{0}(z) = \upsilon(z)$ for $z \in K$. We shall restrict our attention to the case when $\Lambda_- = \big\{x \in \RR : \tau(x) < 0 \big\} \neq \emptyset$. Notice that if $\Lambda_- = \emptyset$, but $\Lambda_+ = \{ x \in \RR : \tau(x) > 0 \} \neq \emptyset$, then according to \cite[Theorem A]{jordan2} one has that either $A$ is self-adjoint and $\sigmaEss(A) = \emptyset$ or $A$ is not self-adjoint. So by Proposition~\ref{prop:12} and Remark~\ref{rem:5} we have $\tfrac{1}{\rho_n} \nu_n \xrightarrow{v} 0$ for any sequence $\rho_n \to \infty$.

Our next step is to define the eigenvalues $\lambda_j^-(z)$ and $\lambda_j^+(z)$ of $Y_j(z)$ for $z \in K$, and to compute
the limit
\[
	\lim_{j \to \infty} \frac{a_{jN+N-1}}{\sqrt{\alpha_{N-1} \gamma_{jN+N-1}}} \frac{\big(\lambda_j^+(z)\big)'}{\lambda_j^+(z)}.
\]
At this stage we split the analysis according to whether $\mathfrak{t} = 0$ or $\mathfrak{t} = 1$. 

\vspace{1ex}
\noindent
{Case \bf 1. } Suppose that $\mathfrak{t} = 0$, thus $\tau$ is a constant. As we explained already, for the purposes of the present article 
the only interesting case is when $\tau < 0$. In view
of \cite[Corollary 3.10]{PeriodicIII}, 
\[
	\lim_{j \to \infty} \frac{a_{jN}}{\alpha_0} X_{jN}' = \frakX_0'(0),
\]
with $\tr \frakX_0'(0) \neq 0$. Since $Z_j(z)$ is independent of $z$, we have $(\det Y_j)' \equiv 0$. Moreover, by \eqref{eq:26}
\begin{align}
	\nonumber
	\lim_{j \to \infty} \frac{a_{jN}}{\alpha_0} \tr Y_j'(z) 
	&= 
	\lim_{j \to \infty} \frac{a_{jN}}{\alpha_0} \tr \big(X_{jN}'(z) Z_j Z_{j+1}^{-1}\big) \\
	\nonumber
	&=
	\lim_{j \to \infty} \frac{a_{jN}}{\alpha_0} \tr X_{jN}'(z) 
	+
	\lim_{j \to \infty}  \frac{a_{jN}}{\alpha_0} \sqrt{\frac{\alpha_{N-1}}{\gamma_{jN+N-1}}}
	\tr \big(X_{jN}'(z) Q_j(z) \big) \\
	\label{eq:38}
	&=
	\tr \frakX_0'(0).
\end{align}
Consequently, we can repeat the proof of Lemma \ref{lem:6} to justify the following statement.
\begin{lemma}
	\label{lem:4}
	Let $N \geq 1$. Let $(\gamma_n : n \geq 0)$ be a sequence of positive numbers tending to infinity and satisfying 
	\eqref{eq:83a} and \eqref{eq:83b}. Let $(a_n : n \geq 0)$ and $(b_n : n \geq 0)$ be $\gamma$-tempered $N$-periodically
	modulated Jacobi parameters such that $\frakX_0(0)$ is non-trivial parabolic element. Assume that $\mathfrak{t} = 0$.
	Let $K$ be a compact subset of $\CC_\varsigma$ with nonempty interior where $\varsigma = \sign{\tr \frakX_0'(0)}$. 
	Then there is $M \geq 1$ such that for all $j \geq M$, and $z \in K$,
	\[
		\Im \bigg(\frac{\tr Y_j(z)}{2 \sqrt{\det Y_j(z)}} \bigg) > 0.
	\]
\end{lemma}

Let $K$ be a compact subset of $\CC_\varsigma$ with nonempty interior where $\varsigma = \sign{\tr \frakX_0'(0)}$.
By Lemma \ref{lem:4} for $K$, the eigenvalues of $Y_j$ given 
by formulas \eqref{eq:63} are continuous on $K$ and holomorphic in $\intr(K)$, provided that $j \geq M$.

In view of \eqref{eq:138} we are in the setup \eqref{eq:71} with $\tilde{\gamma}_j = \sqrt{\gamma_{jN+N-1}/\alpha_{N-1}}$,
$\tilde{\varepsilon} = 1$, and $\tilde{Y}_j = \varepsilon Y_j$. Hence, the eigenvalues of $R_j$ are 
\[
	\zeta_j^+(z) = \sqrt{\frac{\gamma_{jN+N-1}}{\alpha_{N-1}}} \big(\varepsilon \lambda_j^+(z) - 1\big), 
	\quad\text{and}\quad
	\zeta_j^-(z) = \sqrt{\frac{\gamma_{jN+N-1}}{\alpha_{N-1}}} \big(\varepsilon \lambda_j^-(z) - 1\big),
\]
and both are continuous on $K$ and holomorphic in $\intr(K)$, if $j \geq M$.

Next, let us recall that by \eqref{eq:63},
\[
	\lambda_j^+(z) =  \sqrt{\det Y_j(z)} \xi_+\big(w_j(z)\big)
\]
where
\[
	w_j(z) = \frac{\tr Y_j(z)}{2 \sqrt{\det Y_j(z)}}.
\]
Since $(\det Y_j)' \equiv 0$, we get
\begin{align}
	\nonumber
	\frac{(\lambda_j^+)'(z)}{\lambda_j^+(z)} 
	&= \frac{(\det Y_j)'(z)}{2 \det Y_j(z)} + \frac{\xi_+'\big(w_j(z)\big)}{\xi_+\big(w_j(z)\big)} w_j'(z) \\
	\label{eq:41}
	&= \frac{w_j'(z)}{\xi_+\big(w_j(z)\big) - w_j(z)}.
\end{align}
In view of \eqref{eq:38} we obtain
\begin{equation} \label{eq:205}
	\lim_{j \to \infty} \frac{a_{jN}}{\alpha_0} w_j'(z) = \frac{\tr \frakX_0'(0)}{2}.
\end{equation}
Since
\[
	\frac{\gamma_{jN+N-1}}{\alpha_{N-1}} (w_j^2 - 1) 
	= 
	\frac{\gamma_{jN+N-1}}{\alpha_{N-1}} \frac{\discr Y_j}{4 \det Y_j} 
	= \frac{\discr R_{jN}}{4 \det Y_j}
\]
we have
\begin{equation} \label{eq:206}
	\lim_{j \to \infty} \frac{\gamma_{jN+N-1}}{\alpha_{N-1}} (w_j^2 - 1)
	=
	\frac{\discr \calR_0}{4}.
\end{equation}
By \eqref{eq:138}, we get
\[
	\lim_{j \to \infty} w_j = \varepsilon,
\]
therefore,
\[
	\lim_{j \to \infty} \frac{\gamma_{jN+N-1}}{\alpha_{N-1}} (w_j - \varepsilon)
	=
	\lim_{j \to \infty} \frac{\gamma_{jN+N-1}}{\alpha_{N-1}} (w_j^2 - 1) \frac{1}{w_j + \varepsilon} 
	=
	\frac{\discr \calR_0}{8 \varepsilon}.
\]
By \eqref{eq:88} and \eqref{eq:38}, we get
\[
	\lim_{j \to \infty} a_{jN} \Im w_j = 0,
\]
thus
\begin{align*}
	\lim_{j \to \infty} \frac{\Im w_j}{\Re (w_j - \varepsilon)}
	&=
	\lim_{j \to \infty} 
	\frac{\gamma_{jN+N-1}}{a_{jN}} 
	\frac{a_{jN} \Im w_j}{\gamma_{jN+N-1} (\Re w_j - \varepsilon)} = 0.
\end{align*}
Since $\sign{\Re(w_j - \varepsilon)} = -\varepsilon$, we obtain
\begin{align*}
	\lim_{j \to \infty} 
	\sqrt{\frac{\gamma_{jN+N-1}}{\alpha_{N-1}}} \sqrt{w_j - \varepsilon}
	&= 
	\lim_{j \to \infty}
	\sqrt{\frac{\gamma_{jN+N-1}}{\alpha_{N-1}} \abs{w_j - \varepsilon}}
	\exp\bigg\{\frac{i}{2} \arctan\bigg(\frac{\Im w_j}{\Re w_j - \varepsilon}\bigg)
	+i \pi \frac{\varepsilon+1}{4}\bigg\} \\
	&=
	\sqrt{-\discr \calR_0} \exp\bigg\{ i \pi \frac{\varepsilon+1}{4}\bigg\}.
\end{align*}
On the other hand, $\sign{\Re(w_j + \varepsilon)} = \varepsilon$ and so
\begin{align*}
	\lim_{j \to +\infty} \sqrt{w_j + \varepsilon} 
	&=
	\lim_{j \to +\infty} \sqrt{\abs{w_j+\varepsilon}} 
	\exp\bigg\{i \arctan\bigg(\frac{\Im w_j}{\Re w_j + \varepsilon}\bigg) + i \pi \frac{1-\varepsilon}{4} \bigg\} \\
	&=
	\sqrt{2} \exp\bigg\{ i \pi \frac{1-\varepsilon}{4}\bigg\}.
\end{align*}
Thus
\begin{align*}
	\lim_{j \to \infty} 
	\sqrt{\frac{\gamma_{jN+N-1}}{\alpha_{N-1}}} 
	\big(\xi_+(w_j) - w_j \big) 
	&=
	\lim_{j \to \infty}
	\sqrt{\frac{\gamma_{jN+N-1}}{\alpha_{N-1}}} \sqrt{w_j - \varepsilon} \sqrt{w_j + \varepsilon} \\
	&=
	\frac{i}{2} \sqrt{-\discr \calR_0}. 
\end{align*}
Hence, by \eqref{eq:41}
\begin{align*}
	\lim_{j \to \infty}
	\frac{a_{jN+N-1}}{\sqrt{\alpha_{N-1} \gamma_{jN+N-1}}}
	\frac{(\lambda_j^+)'}{\lambda_j^+}
	&=
	\lim_{j \to \infty} 
	\frac{a_{jN+N-1}}{a_{jN}} \frac{\alpha_0}{\alpha_{N-1}} 
	\frac{ \frac{a_{jN}}{\alpha_0} w_j'}
	{\sqrt{\frac{\gamma_{jN+N-1}}{\alpha_{N-1}}} \big(\xi_+(w_j) - w_j \big)} \\
	&=
	i \frac{\tr \frakX_0'(0)}{\sqrt{-\discr \calR_0}} 
\end{align*}
uniformly on $K$. 

\vspace{1ex}
\noindent
{Case \bf 2.} Suppose that $\mathfrak{t} = 1$, that is $\gamma_n \equiv a_n$. Since
\[
	\det Y_j = 
	\det X_{jN} \det Z_j (\det Z_{j+1})^{-1} = 
	\frac{a_{jN-1}}{a_{jN+N-1}}
	\frac{\ue^{\vartheta_j} - \ue^{-\vartheta_j}}{\ue^{\vartheta_{j+1}} - \ue^{-\vartheta_{j+1}}},
\]
we get
\begin{align*}
	(\det Y_j)' 
	&=  
	\frac{a_{jN-1}}{a_{jN+N-1}}
	\frac{\sinh \vartheta_j}{\sinh \vartheta_{j+1}}
	\big(
	\vartheta_j' \cdot 
	\coth \vartheta_j 
	-
	\vartheta_{j+1}' \cdot 
	\coth \vartheta_{j+1} 
	\big) \\
	&=
	\frac{1}{2}
	\sqrt{\frac{a_{jN-1}}{a_{jN+N-1}}}
	\frac{\sinh \vartheta_j}{\vartheta_j}
	\frac{\vartheta_{j+1}}{\sinh \vartheta_{j+1}}
	\frac{\tau'(z)}{\tau(z)}
	\big(
	\vartheta_j \cdot 
	\coth \vartheta_j 
	-
	\vartheta_{j+1} \cdot 
	\coth \vartheta_{j+1} 
	\big).
\end{align*}
Recall that
\[
	\frac{\sinh z}{z} = 1 + \calO(|z|^2), \quad\text{and}\quad
	z \cdot \coth z = 1 + \frac{z^2}{3} + \calO(|z|^4),
\]
thus
\begin{align*}
	(\det Y_j)'(z)
	&=
	\frac{\alpha_{N-1}}{6}
	\sqrt{\frac{a_{jN-1}}{a_{jN+N-1}}}
	\Big(1 + \calO\big(\gamma_{jN}^{-1} \big) \Big)
	\tau'(z)
	\bigg(\frac{1}{a_{jN-1}} - \frac{1}{a_{jN+N-1}} + \calO\big(a_{jN-1}^{-2}\big) \bigg) \\
	&=
	\frac{\alpha_{N-1}}{6}
	\sqrt{\frac{a_{jN-1}}{a_{jN+N-1}}}
	\bigg(\frac{1}{a_{jN-1}} - \frac{1}{a_{jN+N-1}} + \calO\big(a_{jN-1}^{-2}\big)\bigg)
	\tau'(z).
\end{align*}
Hence, in view of \eqref{eq:83b}
\begin{equation}
	\label{eq:111}
	\lim_{j \to \infty} a_{jN+N-1}(\det Y_j)'(z) = 0
\end{equation}
uniformly with respect to $z \in K$. 

Next, we recall that we are in the setup \eqref{eq:71} with $\tilde{\gamma}_j = \sqrt{\gamma_{jN+N-1}/\alpha_{N-1}}$,
$\tilde{\varepsilon} = 1$ and $\tilde{Y}_j = \varepsilon Y_j$. Since
\begin{equation}
	\label{eq:188}
	( \discr \calR_{0} )'(z) 
	= 4 \tau'(z) 
	\equiv
	4 \varepsilon (\tr \frakX_0'(0)) \in \RR \setminus \{0\},
\end{equation}
we can apply Proposition~\ref{prop:3}. By Remark \ref{rem:4}, for $K$ a compact subset of $U \cap \CC_\varsigma$
with nonempty interior where $\varsigma = \sign{ \tr \frakX_0'(0)}$, we have
\[
	\Im \bigg(\frac{\tr Y_j(z)}{2 \sqrt{\det Y_j(z)}} \bigg) > 0.
\]
Hence, the eigenvalues of $Y_j$ given by formulas \eqref{eq:63} are continuous on $K$ and holomorphic in $\intr(K)$,
provided that $j \geq M$. The same holds true for eigenvalues of $R_j$, that are given as
\[
	\zeta_j^+(z) = \sqrt{\frac{a_{jN+N-1}}{\alpha_{N-1}}} \big(\varepsilon \lambda_j^+(z) - 1\big), 
	\quad\text{and}\quad
	\zeta_j^-(z) = \sqrt{\frac{a_{jN+N-1}}{\alpha_{N-1}}} \big(\varepsilon \lambda_j^-(z) - 1\big).
\]
Since $\discr \calR_0(z) \in \CC \setminus (-\infty, 0]$ for $z \in K$, by Proposition \ref{prop:9}, we get
\begin{align*}
	\lim_{j \to \infty} \frac{a_{jN+N-1}}{\sqrt{\alpha_{N-1} \gamma_{jN+N-1}}} \frac{(\lambda_j^+)'(z)}{\lambda_j^+(z)}
	&=
	\lim_{j \to \infty} \sqrt{\frac{a_{jN+N-1}}{\alpha_{N-1}}} \frac{(\lambda_j^+)'(z)}{\lambda_j^+(z)} \\
	&=
	\frac{1}{4} \frac{(\discr \calR_0)'(z)}{\sqrt{\discr \calR_0(z)}},
\end{align*}
thus by \eqref{eq:188},
\begin{equation}
	\label{eq:40}
	\lim_{j \to \infty} \frac{a_{jN+N-1}}{\sqrt{\alpha_{N-1} \gamma_{jN+N-1}}} \frac{(\lambda_j^+)'(z)}{\lambda_j^+(z)}
	=
	\varepsilon \frac{\tr \frakX_0'(0)}{\sqrt{\discr \calR_0(z)}}.
\end{equation}

Finally, we are going to construct the appropriate diagonalization. We claim that $[\calR_0]_{12} \neq 0$ on $K$. Indeed,
if $[\calR_0(z)]_{12} = 0$, then by \eqref{eq:48},
\[
	\sqrt{-\tau(z)} + \frac{\upsilon(z)}{\sqrt{-\tau(z)}} = \mathfrak{S}
\]
thus
\[
	-\tau(z) + 2 \upsilon(z) - \frac{\upsilon^2(z)}{\tau(z)} = \mathfrak{S}^2
\]
and so
\[
	(\tau(z) + \upsilon(z))^2 - 4 \tau(z) \upsilon(z) = -\mathfrak{S}^2 \tau(z),
\]
which by \eqref{eq:189} can be written as
\[
	\frac{\mathfrak{S}^4}{16} = -\tau(z) \big(\mathfrak{S}^2 - 4 \upsilon(z)\big) = - 4 \tau^2(z).
\]
Consequently,
\[
	\tau^2(z) = -\frac{1}{64} \mathfrak{S}^4,
\]
which implies that $\tau(z) \in i \RR$ and so $\Re \tau(z) = 0$, which is impossible, because $K \subset U$, where $U$ is defined in \eqref{eq:204}.

Since $[\calR_0]_{12} \neq 0$ on $K$, there are $M \geq 1$ and $\delta > 0$ such that for all $j \geq M$ and $z \in K$,
\[
	\big| [R_j(z)]_{12} \big| > \delta, \quad\text{and} \quad \big|\discr R_j(z) \big| > \delta.
\]
Then for each $j \geq M$ and $z \in K$, the matrix $R_j(z)$ can be diagonalized 
\[
	R_j(z) = C_{j}(z) \tilde{D}_j(z) C_{j}^{-1}(z)
\]
where
\[
	C_{j}(z) = 
	\begin{pmatrix}
		1 & 1 \\
		\frac{\zeta_j^+(z) - [R_j(z)]_{11}}{[R_j(z)]_{12}} & \frac{\zeta_j^-(z) - [R_j(z)]_{11}}{[R_j(z)]_{12}}
	\end{pmatrix}
\]
and
\[
	\tilde{D}_j =
	\begin{pmatrix}
		\zeta_{j}^+(z) & 0 \\
		0 & \zeta_{j}^-(z)
	\end{pmatrix}.
\]
Therefore,
\[
	Y_j(z) = C_{j}(z) D_{j} C_{j}^{-1}(z)
\]
where
\[
	D_j (z) = \begin{pmatrix}
		\lambda_j^+(z) & 0 \\
		0 & \lambda_j^-(z)
	\end{pmatrix}.
\]
Let us notice that $C_n$ and $D_n$ are continuous on $K$ and holomorphic mappings on $\intr(K)$.

\section{Asymptotics of generalized eigenvectors} \label{sec:7}
In view of Remark~\ref{rem:2}, to prove the convergence of $\calC \big[\tfrac{1}{\rho_n} \nu_n \big]$ it is enough to understand the asymptotic behavior
of $(p_n : n \geq 0)$ as well as its derivative. Our main tool to establish this goal will be a uniform discrete variant of
Levinson's theorem.

\subsection{Discrete Levinson's type theorem} \label{sec:7.1}
\begin{theorem} 
	\label{thm:10}
	Let $K$ be a compact subset of $\CC$ with nonempty interior. Let $(Y_n : n \geq M)$ be a sequence of mappings 
	defined on $K$ with with values in $\GL(2,\CC)$ which are continuous on $K$ and holomorphic on $\intr(K)$.
	Suppose that for each $n \geq M$,
	\begin{equation}
		\label{eq:45}
		Y_n(z) = C_n(z) D_n(z) C_n^{-1}(z), \quad z \in K
	\end{equation}
	where for every $n \geq M$ the mappings $C_n$ and $D_n$ are holomorphic on $K$. Assume that
	\begin{enumerate}[label=\rm (\alph*), start=1, ref=\alph*]
		\item 
		\label{en:2:1}
		$\begin{aligned}
		(C_n : n \geq M) \in \calD_1 \big( K, \Mat(2, \CC) \big);
		\end{aligned}$
		\item 
		for every $n \geq M$ and $z \in K$ one has $D_n(z) = \diag \big( \lambda_n^+(z), \lambda_n^-(z) \big)$ with 
		$|\lambda_n^+(z)| \geq |\lambda_n^-(z)|$;
		\item 
		\label{en:2:3}
		there exist mappings $C_\infty$ and $D_\infty$ defined on $K$ with values in $\GL(2,\CC)$, such that
		\[
			\lim_{n \to \infty} 
			\sup_K 
			\|C_n - C_\infty \| = 0 
			\quad \text{and} \quad
			\lim_{n \to \infty} 
			\sup_K
			\| D_n - D_\infty \| = 0.
		\]
		\end{enumerate}
		If 
		\begin{equation}
			\label{eq:46}
			\inf_{z \in K}
			\prod_{n=M}^\infty  
			\Big| \frac{\lambda_n^+(z)}{\lambda_{n}^-(z)} \Big| = \infty
			\qquad \text{or} \qquad
			\sup_{z \in K}
			\prod_{n=M}^\infty  
			\Big| \frac{\lambda_n^+(z)}{\lambda_{n}^-(z)} \Big| < \infty,
		\end{equation}
		then there exist sequences $(\Phi^-_n : n \geq M)$ and $(\Phi^+_n : n \geq M)$ of mappings defined on $K$ with
		values in $\CC^2$ such that
		\begin{enumerate}[label=\rm (\arabic*), start=1, ref=\arabic*]
			\item 
			\label{en:3:1}
			for each $n \geq M$, $\Phi_n^-$ and $\Phi_n^+$ are holomorphic in $\intr(K)$ and continuous on $K$; 
			\item 
			for each $z \in K$, the sequences $(\Phi^-_n(z) : n \geq M)$ and $(\Phi^+_n(z) : n \geq M)$ are 
			linearly independent solution of
			\begin{equation}
				\label{eq:47}
				\Phi_{n+1}(z) = Y_n(z) \Phi_n(z);
			\end{equation}
			\item 
			\label{en:3:3}
			$\begin{aligned}
				\lim_{n \to \infty} 
				\sup_{z \in K} 
				\bigg\| 
				\frac{\Phi^-_n(z)}{\prod_{j=M}^{n-1} \lambda_j^-(z)} - 
				C_{\infty}(z) e_1 
				\bigg\| = 0,
				\quad \text{and} \quad
				\lim_{n \to \infty} 
				\sup_{z \in K} 
				\bigg\| 
				\frac{\Phi^+_n(z)}{\prod_{j=M}^{n-1} \lambda_j^+(z)} -
				C_{\infty}(z) e_2 
				\bigg\| = 0.
			\end{aligned}$
		\end{enumerate}
\end{theorem}
\begin{proof}
	Let $\Psi_n = C_n^{-1} \Phi_n$. By \eqref{eq:45}, the equation \eqref{eq:47} takes the form
	\[
		\Psi_{n+1} 
		= C_{n+1}^{-1} C_{n} D_{n} \Psi_j.
	\]
	Since
	\[
		C_{n+1}^{-1} C_{n} = \Id - C_{n+1}^{-1} (C_{n+1} - C_{n})
	\]
	we have
	\begin{equation}
		\label{eq:49}
		\Psi_{n+1} = (D_{n} + R_n) \Psi_n, \quad n \geq M
	\end{equation}
	where
	\[
		R_n = -C_{n+1}^{-1} (C_{n+1} - C_n)D_n.
	\]
	In view of \eqref{en:2:1} and \eqref{en:2:3}, we get
	\[
		\sum_{n \geq M} \sup_{K} \| R_n \| < \infty.
	\]
	By non-singularity of $D_\infty$ and \eqref{eq:46} the sequence $(D_{n} : n \geq M)$ satisfies the uniform Levinson's 
	condition, see \cite[Definition 2.1]{Silva2004}. Therefore, in view of \cite[Theorem 4.1]{Silva2004} there are sequences 
	$(\Psi^-_n : n \geq M)$ and $(\Psi^-_n : n \geq M)$ satisfying \eqref{eq:49} and such that
	\[
		\lim_{n \to \infty} 
		\sup_{K} 
		\bigg\| 
		\frac{\Psi^-_n}{\prod_{M \leq j < n} \lambda^-_{j}}  - e_1
		\bigg\| 
		= 0,
		\quad
		\text{and}
		\quad
		\lim_{n \to \infty} 
		\sup_{K} 
		\bigg\| 
		\frac{\Psi^+_n}{\prod_{M \leq j < n} \lambda^+_{j}}  - e_2
		\bigg\| 
		= 0.
	\]
	Furthermore, for any $z \in K$ the sequences $(\Psi^-_n(z) : n \geq M)$ and $(\Psi^+_n(z) : n \geq M)$ are linearly 
	independent. We set
	\[
		\Phi_n^- = C_n \Psi_n^-, \quad\text{and}\quad
		\Phi_n^+ = C_n \Phi_n^+.
	\]
	Since for each $z \in K$ the matrix $C_\infty(z)$ is invertible, the sequences $(\Phi^-_n : n \geq M)$ and 
	$(\Phi^+_n : n \geq M)$ satisfy \eqref{en:3:1}--\eqref{en:3:3}.
\end{proof}

\subsection{Asymptotic behavior of a basis of generalized eigenvectors} \label{sec:7.2}
\begin{theorem}[Case \ref{eq:PI}] 
	\label{thm:5}
	Suppose that $(a_n: n \in \NN_0)$, $(b_n : n \in \NN_0)$ are $N$-periodically modulated Jacobi parameters such that
	$\tr \frakX_0(0) \in (-2,2)$. If
	\[
		\bigg( \frac{a_{n-1}}{a_n} : n \in \NN \bigg),
		\bigg( \frac{b_n}{a_n} : n \in \NN \bigg),
		\bigg( \frac{1}{a_n} : n \in \NN \bigg) \in \calD_1^N,
	\]
	and the Carleman's condition is satisfied, then for each compact set $K \subset \CC_\varsigma$ where 
	$\varsigma = \sign{\tr \frakX_0'(0)}$, there are $M \geq 1$, and a~basis of generalized eigenvectors 
	$\{u^+(z), u^-(z)\}$ associated with $z \in K$ which are continuous on $K$ and holomorphic on $\intr(K)$, 
	such that for each $i_0 \in \NN_0$, the limits
	\[
		\varphi_{i_0}^+
		= 
		\lim_{k \to \infty}
		\frac{u^+_{kN+i_0}}{\prod_{M \leq j < k} \lambda_{j}^+}
		\quad\text{and}\quad
		\varphi_{i_0}^-
		=
		\lim_{k \to \infty}
		\frac{u^-_{kN+i_0}}{\prod_{M \leq j < k} \lambda_{j}^-},
	\]
	exist uniformly on $K$ where $\lambda_n^+$ and $\lambda_n^-$ are defined by \eqref{eq:63} for $Y_j = X_{jN}$.
	Moreover, for all $z \in K$ and $i_0 \in \{0, 1, \ldots, N-1\}$,
	\begin{equation}
		\label{eq:42}
		|\varphi_{i_0}^+ (z)| + |\varphi_{i_0+1}^+(z)| > 0,
		\quad\text{and}\quad
		|\varphi_{i_0}^- (z)| + |\varphi_{i_0+1}^-(z)| > 0.
	\end{equation}
\end{theorem}
\begin{proof}
	We are going to study solutions of the difference equation
	\begin{equation}
		\label{eq:139}
		\phi_{n+1} = B_n \phi_n, \quad n \geq 0.
	\end{equation}
	Fix $K$ a compact set of $\CC_\varepsilon$. Let $M \geq 1$ be determined in Lemma~\ref{lem:7}. Set 
	$\Phi_j = \phi_{jN}$. Then 
	\begin{equation} 
		\label{eq:140}
		\Phi_{j+1} = Y_{j} \Phi_j, \quad j \geq M,
	\end{equation}
	where $Y_j = X_{jN}$. Since $\tr \frakX_0(0) \in (-2,2)$, in view of Section~\ref{sec:6.3.1} we can apply Theorem \ref{thm:10}.
	Hence there are sequences $(\Phi^-_j : j \geq M)$ and $(\Phi^+_j : j \geq M)$ satisfying \eqref{eq:140} and such that
	\begin{equation}
		\label{eq:141}
		\lim_{k \to \infty} 
		\sup_{K} 
		\bigg\| 
		\frac{\Phi^-_k}{\prod_{M \leq j < k} \lambda^-_{j}}  - C_\infty e_1
		\bigg\| = 0,
		\quad\text{and}\quad
		\lim_{k \to \infty} 
		\sup_{K} 
		\bigg\| 
		\frac{\Phi^+_k}{\prod_{M \leq j < k} \lambda^+_{j}}  - C_\infty e_2
		\bigg\| = 0.
	\end{equation}
	Moreover, for each $z \in K$, the sequences $(\Phi^{-}_j(z) : j \geq M)$ and $(\Phi^{+}_j(z) : j \geq M)$ are linearly 
	independent. We set
	\[
		\phi^-_{0} = B_0^{-1} B_1^{-1} \ldots B_{MN-1}^{-1} \Phi^-_M,
	\]
	and
	\begin{equation}
		\label{eq:142}
		\phi_{n+1}^- = B_n \phi_n^-, \quad n \geq 0.
	\end{equation}
	Then $(\phi^-_n : n \geq 0)$ satisfies \eqref{eq:139}. In view of \eqref{eq:140} and \eqref{eq:142}, 
	for all $k \geq M$ and $i_0 \in  \{0, 1, \ldots, N-1 \}$,
	\[
		\phi_{kN+i_0}^- =
		B_{kN+i_0-1} B_{kN+i_0-2} \ldots B_{kN} \Phi^-_{k}.
	\]
	Since for $i_0' \in \{0, 1, \ldots, N-1 \}$,
	\[
		\lim_{j \to \infty} B_{jN+i_0'}(z) = \frakB_{i_0'}(0)
	\]
	uniformly with respect to $z \in K$, by \eqref{eq:141} we obtain
	\[
		\lim_{k \to \infty} 
		\sup_{K} 
		\bigg\| 
		\frac{\phi^-_{kN+i_0}}{\prod_{M \leq j < k} \lambda^-_{j}}  - 
		\frakB_{i_0-1}(0) \frakB_{i_0-2}(0) \ldots \frakB_0(0) C_\infty e_1
		\bigg\| 
		= 0.
	\]
	We set
	\[
		u_n^-(z) =
		\begin{cases}
			\langle \phi^-_1(z), e_1 \rangle & \qquad n = 0, \\
			\langle \phi^-_n(z), e_2 \rangle & \qquad n \geq 1.
		\end{cases}
	\]
	In view of \eqref{eq:139}, $u^-(z) = (u^-_n(z) : n \geq 0)$ is a generalized eigenvector associated with $z \in K$,
	provided that $(u^-_0(z), u^-_1(z)) \neq 0$. Suppose, that $(u^-_0(z), u^-_1(z)) = 0$, that is $\phi_1^-(z) = 0$.
	Hence, by \eqref{eq:142} we deduce that $\phi_n^-(z) = 0$ for all $n \geq 1$. Since $\Phi_j^-(z) = \phi^-_{jN}(z)$, 
	in view of \eqref{eq:141} we conclude that $C_\infty e_1 = 0$, which is impossible since $C_\infty(z)$ is invertible. 
	Analogously, we construct the sequence $u^+(z) = (u^+_n(z) : n \geq 0)$, $z \in K$.

	Next, we claim that for each $z \in K$, the sequences $u^-(z)$ and $u^+(z)$ are linearly independent. To see this
	let us suppose that for certain constants $\alpha, \beta \in \CC$, 
	\[
		\alpha u_n^-(z) + \beta u_n^+(z) = 0, \qquad \text{for all } n \geq 0.
	\]
	By \eqref{eq:139}, we have $\langle \phi^\pm_n(z), e_1 \rangle = u^\pm_{n-1}(z)$, hence
	\[
		\alpha \phi_n^-(z) + \beta \phi_n^+(z) = 0, \qquad \text{for all } n \geq 1.
	\]
	Recall that $\phi^\pm_{kN} = \Phi^\pm_k$ for $k \geq M$, thus by linear independence of 
	$(\Phi^{-}_j(z) : j \geq M)$ and $(\Phi^{+}_j(z) : j \geq M)$ we obtain $\alpha=\beta=0$, that is $u^-$ and $u^+$
	are linearly independent.

	Finally, to show \eqref{eq:42}, we observe that for each $z \in K$ and $i_0 \in \{0, 1, \ldots, N-1 \}$, the matrix
	\[
		\frakB_{i_0-1}(0) \frakB_{i_0-2}(0) \ldots \frakB_0(0) C_\infty(z)
	\]
	is invertible, thus
	\begin{equation}
		\label{eq:144}
		\lim_{k \to \infty} 
		\frac{\phi^-_{kN+i_0}(z)}{\prod_{M \leq j < k} \lambda^-_{j}(z)} \neq 0.
	\end{equation}
	Since $\langle \phi^-_n(z), e_1 \rangle = u_{n-1}^-(z)$, by \eqref{eq:144} we conclude that 
	$(\vphi_{i_0-1}(z), \vphi_{i_0}(z)) \neq 0$. This completes the proof of the theorem.
\end{proof}

A similar reasoning leads to the following statement.
\begin{theorem}[Case \ref{eq:PIIa}] 
	\label{thm:5s}
	Suppose that $(a_n : n \in \NN_0)$ and $(b_n : n \in \NN_0)$ are $N$-periodically modulated Jacobi parameters such that	
	$\frakX_0(0) = \varepsilon \Id$ with $\varepsilon \in \{-1, 1\}$. Assume that
	\[
		\bigg( a_n \Big( \frac{\alpha_{n-1}}{\alpha_n} - \frac{a_{n-1}}{a_{n-1}} \Big): n \in \NN \bigg),
		\bigg( a_n \Big( \frac{\beta_n}{\alpha_n} - \frac{b_n}{a_n} \Big): n \in \NN_0 \bigg),
		\bigg( \frac{1}{a_n} : n \in \NN_0 \bigg) \in \calD_1^N.
	\]
	Set
	\[
		h(z) = \lim_{j \to \infty} a_{jN+N-1}^2 \discr X_{jN}(z), \quad z \in \CC.
	\]
	Then for each $K$ a compact subset of
	\[
		\big\{
		z \in \CC_+ :
		h(z) \in \CC \setminus (-\infty, 0] \text{ and } \varepsilon \Re \big( h'(z) \big) > 0
		\big\},
	\]
	there are $M \geq 1$ and a~basis of generalized eigenvectors $\{u^+(z), u^-(z)\}$ associated with $z \in K$
	which are continuous on $K$ and holomorphic on $\intr(K)$, such that for any $i_0 \in \NN_0$, the limits
	\[
		\varphi_{i_0}^+
		= 
		\lim_{k \to \infty}
		\frac{u^+_{kN+i_0}}{\prod_{M \leq j < k} \lambda_{j}^+}
		\quad\text{and}\quad
		\varphi_{i_0}^-
		=
		\lim_{k \to \infty}
		\frac{u^-_{kN+i_0}}{\prod_{M \leq j < k} \lambda_{j}^-},
	\]
	exist uniformly on $K$ where $\lambda_n^+$ and $\lambda_n^-$ are defined by \eqref{eq:63} for $Y_j = X_{jN}$. Moreover, 
	for all $z \in K$ and $i_0 \in \{0, 1, \ldots, N-1\}$,
	\[
		|\varphi_{i_0}^+ (z)| + |\varphi_{i_0+1}^+(z)| > 0
		\quad\text{and}\quad
		|\varphi_{i_0}^- (z)| + |\varphi_{i_0+1}^-(z)| > 0.
	\]
\end{theorem}
\begin{proof}
	In order to deduce the Levinson's condition
	\begin{equation}
		\label{eq:68}
		\inf_{z \in K}
		\prod_{j = 1}^\infty 
		\bigg| \frac{\lambda^+_n(z)}{\lambda^-_n(z)}\bigg| = \infty,
	\end{equation}
	we use Proposition \ref{prop:2}. To do so it is enough to check that \eqref{eq:64} holds true. By Proposition \ref{prop:15},
	\[
		\abs{\xi_+(w_j)} - 1 \geq  \frac{\abs{\Im w_j}}{2 \sqrt{\abs{w_j - 1}\abs{w_j + 1}}}
	\]
	where
	\[
		w_j = \frac{\tr Y_j}{2 \sqrt{ \det Y_j}}.
	\]
	Since
	\[
		\det Y_j = \frac{a_{jN+N-1}}{a_{jN-1}},
	\]
	for
	\[
		U = \big\{ z \in \CC_+ : h(z) \in \CC \setminus (-\infty, 0] \text{ and } \varepsilon \Re \big( h'(z) \big) > 0 \big\},
	\]
	by Proposition \ref{prop:3}, there is $M_1 \geq 1$, such that for all $j \geq M_1$ and $z \in K$,
	\[
		\abs{\Im w_j} \geq c \frac{\abs{\Im z}}{a_{jN+N-1}^2}.
	\]
	In view of \eqref{eq:80}, there $M_2 \geq M_1$, such that for all $j \geq M_2$ and $z \in K$,
	\[
		\abs{w_j - 1} \abs{w_j + 1} \leq c \frac{\abs{h(z)}}{a_{jN+N-1}^2}.
	\]
	Consequently,
	\[
		\abs{\xi_+(w_j)} - 1\geq c \frac{\abs{\Im z}}{\sqrt{\abs{h(z)}}} \frac{1}{a_{jN+N-1}},
	\]
	which implies \eqref{eq:68} because $K \cap \RR = \emptyset$ and by \cite[Proposition 4]{ChristoffelII} the Carleman's condition for $A$ is satisfied. Next, the diagonalization described in Section~\ref{sec:6.3.2}
	allows us to apply Theorem \ref{thm:10}. From this point on the reasoning follows the same line as in the proof 
	of Theorem~\ref{thm:5}.
\end{proof}

Finally, the following theorem describes a basis of generalized eigenvectors in the case \ref{eq:PIIb}.
\begin{theorem}[Case \ref{eq:PIIb}] 
	\label{thm:5h}
	Let $N \geq 1$. Let $(\gamma_n : n \geq 0)$ be a sequence of positive numbers tending to infinity and satisfying 
	\eqref{eq:83a} and \eqref{eq:83b}. Let $(a_n : n \geq 0)$ and $(b_n : n \geq 0)$ be $\gamma$-tempered $N$-periodically
	modulated Jacobi parameters such that $\frakX_0(0)$ is non-trivial parabolic element. Suppose that \eqref{eq:108b} 
	holds true with $\varepsilon = \sign{\tr \frakX_0(0)}$. Let
	\[
		h(z) = \lim_{j \to \infty} \gamma_{jN+N-1} \discr X_{jN}(z), \quad z \in \CC.
	\]
	If $\Lambda_- = \big\{x \in \RR : \tau(x) < 0 \big\} \neq \emptyset$ and
	\begin{equation} \label{eq:208}
		\sum_{k=0}^\infty 
		\frac{\sqrt{\alpha_k \gamma_{k}}}{a_{k}} = \infty,
	\end{equation}
	then for each compact subset $K$ of 
	\[
		U=\big\{z \in \CC_\varsigma : h(\Re z) < 0 \big\}
	\]
	where $\varsigma = \sign{\tr \frakX_0'(0)}$, there are $M \geq 1$ and a~basis of generalized eigenvectors 
	$\{u^+(z), u^-(z)\}$ associated with $z \in K$ which are continuous on $K$ and holomorphic on $\intr(K)$, such
	that for any $i_0 \in \NN_0$, the limits
	\[
		\varphi_{i_0}^+
		= 
		\lim_{k \to \infty}
		\frac{u^+_{kN+i_0}}{\prod_{M \leq j < k} \lambda_{j}^+}
		\quad\text{and}\quad
		\varphi_{i_0}^-
		=
		\lim_{k \to \infty}
		\frac{u^-_{kN+i_0}}{\prod_{M \leq j < k} \lambda_{j}^-},
	\]
	exist uniformly on $K$ where $\lambda_n^+$ and $\lambda_n^-$ are defined by \eqref{eq:63} with 
	$Y_j = Z_{j+1}^{-1} X_{jN} Z_j$. Moreover, for all $z \in K$ and $i_0 \in \{0, 1, \ldots, N-1\}$,
	\begin{equation}
		\label{eq:43}
		|\varphi_{i_0}^+ (z)| + |\varphi_{i_0+1}^+(z)| > 0
		\quad\text{and}\quad
		|\varphi_{i_0}^- (z)| + |\varphi_{i_0+1}^-(z)| > 0.
	\end{equation}
\end{theorem}
\begin{proof}
	Let us consider the equation
	\begin{equation}
		\label{eq:173}
		\Psi_{j+1} = Y_j \Psi_j, \quad j \geq M.
	\end{equation}
	We are going to check Levinson's condition, that is
	\begin{equation}
		\label{eq:69}
		\inf_{z \in K}
		\prod_{j = 1}^\infty 
		\bigg|\frac{\lambda_j^+(z) }{\lambda_j^-(z)} \bigg| = \infty.
	\end{equation}
	By Proposition \ref{prop:15}, we can estimate
	\begin{equation} \label{eq:210}
		\abs{\xi_+(w_j)} - 1 \geq c \frac{\abs{\Im w_j}}{\sqrt{\abs{w_j-1}\abs{w_j+1}}}.
	\end{equation}
	where
	\[
		w_j = \frac{\tr Y_j}{2 \sqrt{\det Y_j}}.
	\]
	If $\frakt=1$, then by \eqref{eq:111} the hypotheses of Proposition~\ref{prop:3} are satisfied. Thus, there is $M_1 \geq 1$, such that for all $j \geq M_1$ and $z \in K$,
	\[
		\abs{\Im w_j} \geq c \frac{\abs{\Im z}}{a_{jN+N-1}}.
	\]
	The same inequality is also satisfied for $\frakt=0$ by \eqref{eq:75} together with \eqref{eq:205}. Next, in view of \eqref{eq:206} and \eqref{eq:207} for $\frakt=0$ and $\frakt=1$, respectively, 
	there is $M_2 \geq M_1$, such that for all $j \geq M_2$ and $z \in K$,
	\[
		\abs{w_j - 1}\abs{w_j+1} \leq c \frac{\abs{h(z)}}{\gamma_{jN+N-1}}.
	\]
	Hence,
	\begin{equation} \label{eq:209}
		\frac{\abs{\Im w_j}}{\sqrt{\abs{w_j-1}\abs{w_j+1}}}
		\geq 
		c \abs{\Im z} \frac{\sqrt{\gamma_{jN+N-1}}}{a_{jN+N-1}}.
	\end{equation}
	Notice that for any $i_0 \in \{1,2,\ldots,N\}$ we have by \eqref{eq:83a} and \eqref{eq:int:13}
	\begin{align*}
		\lim_{j \to \infty} 
		\frac{\sqrt{\alpha_{jN+N-1} \gamma_{jN+N-1}}}{a_{jN+N-1}}
		\frac{a_{jN+i_0-1}}{\sqrt{\alpha_{jN+i_0-1} \gamma_{jN+i_0-1}}} 
		&=
		\lim_{j \to \infty}
		\sqrt{\frac{\gamma_{jN+N-1}}{\gamma_{jN+i_0-1}} \frac{\alpha_{N-1}}{\alpha_{i_0-1}}}
		\frac{a_{jN+i_0-1}}{a_{jN+N-1}} \\
		&=
		\frac{\alpha_{N-1}}{\alpha_{i_0-1}} \frac{\alpha_{i_0-1}}{\alpha_{N-1}} = 1.
	\end{align*}
	Thus, by \eqref{eq:208} we have that for any $i_0 \in \{1,2,\ldots,N\}$ we have
	\begin{equation} \label{eq:208'}
		\sum_{j=0}^\infty 
		\frac{\sqrt{\alpha_{jN+i_0-1} \gamma_{jN+i_0-1}}}{a_{jN+i_0-1}}
		= \infty.
	\end{equation}
	Therefore, by \eqref{eq:209}, \eqref{eq:210} and \eqref{eq:208'} we have
	\[
		\sum_{j = 0}^\infty \inf_{z \in K} \big(\abs{\xi_+(w_j)} - 1 \big) = \infty,
	\]
	which by Proposition \ref{prop:2} implies \eqref{eq:69}.

	Next, the diagonalization described in Section~\ref{sec:6.3.3} allows us to apply Theorem \ref{thm:10}. Hence,
	there are two sequences $(\Psi^+_j : j \geq M)$ and $(\Psi^-_j : j \geq M)$ satisfying \eqref{eq:173} and such that
	\begin{equation}
		\label{eq:44}
		\lim_{k \to \infty} \sup_K 
		\bigg\| \frac{\Psi^-_k}{\prod_{M \leq j < k} \lambda_j^-} - C_\infty e_1 \bigg\| = 
		\lim_{k \to \infty} \sup_K
		\bigg\| \frac{\Psi^+_k}{\prod_{M \leq j < k} \lambda_j^+} - C_\infty e_2 \bigg\| = 0.
	\end{equation}
	Since $C_\infty(z) e_1$ and $C_\infty(z) e_2$ are eigenvectors of the matrix $\calR_0(z)$ defined by the formula 
	\eqref{eq:48}, and $\tau(z) \neq 0$ for $z \in K$, we must have
	\begin{equation}
		\label{eq:85}
		\langle C_\infty(z) e_1, e_1 + e_2 \rangle \neq 0,
		\quad\text{and}\quad
		\langle C_\infty(z) e_2, e_1 + e_2 \rangle \neq 0.
	\end{equation}
	Next, by \eqref{eq:44}, the sequences $\Phi^\pm_j = Z_j \Psi^\pm_{j}$ satisfy
	\[
		\Phi_{j+1} = X_{jN} \Phi_j, \quad j \geq M.
	\]
	From this point on, the construction of generalized eigenvectors $u^-$ and $u^+$ follows the same line as in the
	proof of Theorem \ref{thm:5}. If $(u^-_0(z), u^-_1(z)) = 0$, then $\phi^-_1(z) = 0$, and so $\phi^-_n(z) = 0$ for all
	$n \geq 1$. Hence, $\Psi^-_j(z) = Z_j^{-1} \phi^-_{jN}(z) = 0$ for all $j \geq M$. Consequently, by \eqref{eq:44},
	$C_\infty(z) e_1 = 0$ which is impossible since $C_\infty(z)$ is invertible.

	Let $z \in K$, and suppose that for certain $\alpha, \beta \in \CC$,
	\[
		\alpha u^-_n(z) + \beta u^+_n(z) = 0, \quad\text{for all } n \geq 0.
	\]
	Since $\langle \phi^\pm_n(z), e_1 \rangle = u^\pm_{n-1}(z)$, we have
	\[
		\alpha \phi^-_n(z) + \beta \phi^+_n(z) = 0, \quad\text{for all } n \geq 1,
	\]
	thus using $\Psi_j(z) = Z_j^{-1}(z) \phi^\pm_{jN}(z)$, we get
	\[
		\alpha \Psi^-_j(z) + \beta \Psi^+_j(z) = 0, \quad\text{for all } j \geq M.
	\]
	Because the sequences $(\Psi^-_j(z) : j \geq M)$ and $(\Psi^+_j(z) : j \geq M)$ are linearly independent, we deduce that 
	$\alpha = \beta = 0$. Consequently, $u^-(z)$ and $u^+(z)$ are linearly independent.

	It remains to prove \eqref{eq:43}. Suppose that there are $z \in K$ and $i_0 \in \{0, 1, \ldots, N-1\}$,
	such that $\vphi_{i_0}^-(z) = \vphi_{i_0-1}^-(z) = 0$. Since $\langle \phi_n^-(z), e_1 \rangle = u_{n-1}^-(z)$, we must
	have
	\[
		\lim_{k \to \infty} \frac{\phi^-_{kN+i_0}(z)}{\prod_{M \leq j < k} \lambda^-_j(z)} = 0.
	\]
	Hence,
	\begin{equation}
		\label{eq:50}
		\frakB_{i_0-1}(0) \frakB_{i_0-2}(0) \ldots \frakB_0(0) Z_\infty(z) C_\infty(z) e_1 = 0,
	\end{equation}
	where
	\[
		Z_\infty(z) = 
		\begin{pmatrix}
			1 & 1 \\
			1 & 1
		\end{pmatrix}.
	\]
	However, \eqref{eq:50} implies that $\langle C_\infty(z) e_1, e_1 + e_2 \rangle = 0$ which contradicts \eqref{eq:85}.
	This completes the proof.
\end{proof}

\subsection{Applications to orthogonal polynomials} \label{sec:7.3}
Let us fix a compact set $K \subset \CC$. Suppose that $\{ u^-(z), u^+(z) \}$ is a basis of generalized eigenvectors associated 
with $z \in K$ which is continuous on $K$ and holomorphic on $\intr(K)$. Then every generalized eigenvector $u(z)$ associated 
with $z \in K$ can be expressed as
\begin{equation} 
	\label{eq:119}
	u_n(z) = f(z) u^+_n(z) + g(z) u^-_n(z), \quad n \geq 0
\end{equation}
for certain functions $f,g : K \to \CC$. Let us observe that from \eqref{eq:119} and the linearity of the Wronskian we obtain
\begin{align*}
	\Wrk( u(z), u^-(z))
	&= 
	f(z) \Wrk(u^+(z), u^-(z)) + g(z) \Wrk(u^-(z), u^-(z)) \\
	&=
	f(z) \Wrk(u^+(z), u^-(z)).
\end{align*}
Similarly we obtain
\[
	\Wrk(u(z), u^+(z)) = -g(z) \Wrk(u^+(z), u^-(z)).
\]
Since $u^+(z)$ and $u^-(z)$ are linearly independent one has
\[
	0 \neq \Wrk(u^+(z), u^-(z)) 
	=
	a_0
	\det 
	\begin{pmatrix}
		u^+_{0}(z) & u^-_{0}(z) \\
		u^+_1(z) & u^-_{1}(z)
	\end{pmatrix}.
\]
Consequently, 
\begin{equation} 
	\label{eq:120}
	f(z) = \frac{\Wrk(u(z), u^-(z))}{\Wrk(u^+(z), u^-(z))},
	\quad\text{and}\quad
	g(z) = -\frac{\Wrk(u(z), u^+(z))}{\Wrk(u^+(z), u^-(z))}.
\end{equation}
In particular, if $u(z)$ is continuous on $K$ and holomorphic on $\intr(K)$, then both $f$ and $g$ also have these properties. 

By \eqref{eq:119} and \eqref{eq:120} if we have grasp of asymptotic properties of $u^- (z)$ and $u^+(z)$, then we have some understanding of the asymptotic properties of $(u_n(z) : n \geq 0)$. Namely we have the following result.
\begin{proposition} 
	\label{prop:10}
	Let $\calA$ be a Jacobi matrix and let $N \geq 1$.
	Let $K \subset \CC$ be compact and let $\{ u^+(z), u^-(z) \}$ be a basis of generalized eigenvectors associated with $z \in K$. Suppose there are: $M \geq 1$, sequences $(\lambda^+_j(z) : j \geq M), (\lambda^-_j(z) : j \geq M)$ such that $0 < \abs{\lambda_j^-(z)} \leq \abs{\lambda_j^+(z)}$ for any $j \geq M$ and any $z \in K$, and such that the limits
	\begin{equation} 
		\label{eq:146}
		\varphi_{i}^+
		= 
		\lim_{k \to \infty}
		\frac{u^+_{kN+i}}{\prod_{M \leq j < k} \lambda_{j}^+}
		\quad\text{and}\quad
		\varphi_{i}^-
		=
		\lim_{k \to \infty}
		\frac{u^-_{kN+i}}{\prod_{M \leq j < k} \lambda_{j}^-},
	\end{equation}
	exist uniformly on $K$ for any $i \in \{0,1, \ldots, N-1 \}$. If
	\begin{equation} 
		\label{eq:147}
		\inf_{z \in K}
		\prod_{M \leq j}  
		\bigg| \frac{\lambda_j^+(z)}{\lambda_j^-(z)} \bigg| = \infty,
	\end{equation}
	and $u(z)$ satisfies \eqref{eq:119} and \eqref{eq:120}, then for any $i \in \{0,1,\ldots,N-1\}$
	\begin{equation} 
		\label{eq:151}
		\varphi_{i}(z) = 
		\lim_{k \to \infty}
		\frac{u_{kN+i}(z)}{\prod_{M \leq j < k} \lambda_{j}^+(z)}
		=
		\varphi_{i}^+(z)
		\frac{\Wrk(u(z), u^-(z))}{\Wrk(u^+(z), u^-(z))},
	\end{equation}
	exists uniformly with respect to $z \in K$.
\end{proposition}
\begin{proof}
	By \eqref{eq:119} for $i \in \{0, 1, \ldots, N-1 \}$, we have
	\begin{equation} 
		\label{eq:148}
		\frac{u_{kN+i}(z)}{\prod_{M \leq j < k} \lambda_{j}^+(z)} =
		f(z)
		\frac{u^+_{kN+i}(z)}{\prod_{M \leq j < k} \lambda_{j}^+(z)} +
		g(z)
		\frac{u^-_{kN+i}(z)}{\prod_{M \leq j < k} \lambda_{j}^+(z)}.
	\end{equation}
	Since
	\[
		\bigg|
		\frac{u^-_{kN+i}(z)}{\prod_{M \leq j < k}  \lambda_{j}^+(z)}
		\bigg|
		\leq
		\bigg|\frac{u^-_{kN+i}(z)}{\prod_{M \leq j < k} \lambda_{j}^-(z)} \bigg|
		\cdot
		\sup_{z \in K}
		\prod_{M \leq j < k}
		\bigg| \frac{\lambda_j^-(z)}{\lambda_j^+(z)} \bigg|
	\]
	by \eqref{eq:146} and \eqref{eq:147} we obtain
	\[
		\lim_{k \to \infty}
		\frac{u^-_{kN+i}(z)}{\prod_{M \leq j < k}  \lambda_{j}^+(z)}
		=0,
	\]
	uniformly with respect to $z \in K$. Now \eqref{eq:148} together with \eqref{eq:120} leads to \eqref{eq:151}. 
\end{proof}

The following proposition gives some sufficient conditions which imply that for $z_0 \in \CC \setminus \RR$ the sequences $u^-(z_0)$ and $p(z_0)$ are linearly independent.
\begin{proposition} \label{prop:14}
Suppose that $A$ is a self-adjoint Jacobi operator and let $z_0 \in \CC \setminus \RR$. Suppose that the hypotheses of Proposition~\ref{prop:10} are satisfied for $K = \{z_0\}$. If $\varphi_{i_0}^+(z_0) \neq 0$ for some $i_0 \in \{0,1,\ldots,N-1\}$, then all generalized eigenvectors $u(z_0) \in \ell^2$ associated to $z_0$ are multiples of $u^-(z_0)$. Moreover, $u^-(z_0) \in \ell^2$, and consequently, $u^-(z_0)$ and $p(z_0)$ are linearly independent.
\end{proposition}
\begin{proof}
Let $u(z_0)$ be a generalized eigenvector associated with $z_0$. Then by Proposition~\ref{prop:10} we get for any $i \in \{0,1,\ldots,N-1\}$
\[
	\lim_{k \to \infty} 
	\frac{u_{kN+i}(z_0)}{\prod_{M \leq j < k} \lambda_j^+(z_0)} =
	\varphi_i(z_0).
\]
In particular, there is a constant $c>0$ such that for any $i \in \{0,1,\ldots,N-1\}$
\[
	|u_{kN+i}(z_0)| \leq 
	c \prod_{M \leq j < k} |\lambda^+_j(z_0)|^2, \quad k > M.
\]
By summing it up we get
\begin{equation} \label{eq:191}
	\sum_{k=M+1}^\infty |u_{kN+i}(z_0)|^2 \leq
	c \sum_{k=M+1}^\infty \prod_{M \leq j < k} |\lambda^+_j(z_0)|^2, \quad i \in \{0,1,\ldots,N-1\}.
\end{equation}

We claim that if $u(z_0) \in \ell^2$, then necessarily $\varphi_i(z_0) = 0$ for any $i \in \{0,1,\ldots, N-1\}$. If not, then $\varphi_{i_0}(z_0) \neq 0$ for some $i_0 \in \{0,1,\ldots,N-1\}$. Thus, there are constants $c'>0$ and $M' > M$ such that 
\[
	c' \prod_{M \leq j < k} |\lambda_{j}^+(z_0)|^2 \leq 
	|u_{kN+i_0}(z_0)|^2, \quad k \geq M'
\]
By summing it up we get
\[
	c' \sum_{k=M'}^\infty \prod_{M \leq j < k} |\lambda_{j}^+(z_0)|^2 \leq
	\sum_{k=M'}^\infty |u_{kN+i_0}(z_0)|^2 \leq 
	\sum_{k=0}^\infty |u_{k}(z_0)|^2.
\]
Therefore, we get
\[
	\sum_{k=M'}^\infty \prod_{M \leq j < k} |\lambda_{j}^+(z_0)|^2 < \infty
\]
and consequently, by \eqref{eq:191}, all generalized eigenvectors must belong to $\ell^2$. In particular, $(p_n(z_0) : n \in \NN_0) \in \ell^2$ which contradicts the self-adjointness of $A$, see \cite[Corollary 6.7]{Schmudgen2017}. 

Since $\varphi^+_{i_0}(z_0) \neq 0$, formula \eqref{eq:151} implies that we have $\varphi_{i_0}(z_0) = 0$ if and only if $u(z)$ is a constant multiple of $u^-(z_0)$. On the other hand, it is easy to verify that the sequence
\[
	v(z_0) = (A- z_0 \Id)^{-1} e_0
\]
is a non-trivial generalized eigenvector belonging to $\ell^2$. Therefore, also $u^-(z_0) \in \ell^2$ what we needed to prove.
\end{proof}

\begin{proposition}
	\label{prop:8}
	Let $\calA$ be a Jacobi matrix and let $N \geq 1$.
	Assume that there is an open set $U \subset \CC \setminus \RR$, sequence $(\lambda_j : j \geq M)$ of holomorphic 
	functions $\lambda_j: U \rightarrow \CC \setminus \{0\}$ such that for certain $i_0 \in \{0, 1, \ldots, N-1\}$,
	\[
		\lim_{k \to \infty} \frac{p_{kN+i_0}(z)}{\prod_{M \leq j < k} \lambda_j(z)} = \vphi(z)
	\]
	locally uniformly with respect to $z \in U$. If $\vphi \neq 0$ on $U$ then for each positive sequence
	$(\rho_n : n \geq 1)$ tending to infinity
	\[
		\lim_{k \to \infty} 
		\bigg|
		\calC \big[ \tfrac{1}{\rho_{kN+i_0-1}}\nu_{kN+i_0-1} \big](z) + 
		\frac{1}{\rho_{kN+i_0-1}}
		\sum_{j = M}^{k-1} \frac{\lambda_j'(z)}{\lambda_j(z)}
		\bigg|
		=0
	\]
	locally uniformly with respect to $z \in U$ where $(\nu_n : n \in \NN_0)$ is defined in \eqref{eq:25}.
\end{proposition}
\begin{proof}
	Let us define
	\[
		\phi_k(z) = \frac{p_{kN+i_0}(z)}{\prod_{j=M}^{k-1} \lambda_j(z)}.
	\]
	By the hypotheses, $\phi_k : U \to \CC \setminus \{0\}$ is a sequence of holomorphic functions. Since
	\[
		\lim_{k \to \infty} \phi_k(z) = \varphi(z)
	\]
	locally uniformly with respect to $z \in U$, \cite[Theorem II.5.3]{Stein2003} implies that
	\[
		\lim_{k \to \infty} \phi_k'(z) = \varphi'(z)
	\]
	locally uniformly with respect to $z \in U$. Since $\vphi \neq 0$ on $U$,
	\begin{equation} 
		\label{eq:39}
		\lim_{k \to \infty} 
		\frac{\phi_k'(z)}{\phi_k(z)} =
		\frac{\varphi'(z)}{\varphi(z)}
	\end{equation}
	locally uniformly with respect to $z \in U$. Observe that
	\[
		\frac{\phi_k'(z)}{\phi_k(z)} =
		\frac{p_{kN+i_0}'(z)}{p_{kN+i_0}(z)} -
		\sum_{j=M}^{k-1} \frac{\lambda_j'(z)}{\lambda_j(z)}.
	\]
	Because $\rho_n$ tends to infinity, by \eqref{eq:39} we obtain
	\[
		\lim_{k \to \infty}
		\frac{1}{\rho_{kN+i_0-1}}
		\bigg|
		\frac{p_{kN+i_0}'(z)}{p_{kN+i_0}(z)} -
		\sum_{j=M}^{k-1} \frac{\lambda_j'(z)}{\lambda_j(z)}
		\bigg| = 0
	\]
	locally uniformly with respect to $z \in U$. In view of Remark~\ref{rem:2} the proof is completed.
\end{proof}

\section{Weighted weak convergence for periodic modulations} \label{sec:8}
In this section we study orthogonal polynomials with $N$-periodically modulated recurrence coefficients. We shall examine
the vague convergence of the sequences $\big(\tfrac{1}{\rho_n} \eta_n : n \in \NN_0 \big)$ and $\big(\tfrac{1}{\rho_n} \nu_n : n \in \NN_0 \big)$ defined in \eqref{eq:25}, where a positive sequence $(\rho_n : n \in \NN_0)$ which will be chosen based on properties of
the matrix $\frakX_0(0)$.

The next theorem shows that in the setup of Example~\ref{ex:I} for $r=1$ we have a stronger conclusion than what can be obtained
from Corollary~\ref{cor:1}.
\begin{theorem}[Case \ref{eq:PI}] 
	\label{thm:7}
	Let $N \geq 1$. Suppose that $(a_n: n \in \NN_0)$ and $(b_n : n \in \NN_0)$ are $N$-periodically modulated Jacobi parameters
	such that $\tr \frakX_0(0) \in (-2,2)$. Assume that the Carleman's condition is satisfied and
	\[
		\bigg( \frac{a_{n-1}}{a_n} : n \in \NN \bigg),
		\bigg( \frac{b_n}{a_n} : n \in \NN \bigg),
		\bigg( \frac{1}{a_n} : n \in \NN \bigg) \in \calD_1^N.
	\]
	Let $(\eta_n : n \in \NN_0)$ and $(\nu_n : n \in \NN_0)$ be defined in \eqref{eq:25} and set
	\[
		\rho_n = \sum_{j=0}^n \frac{\alpha_j}{a_j}.
	\]
	Then for each function $f \in \calC_0(\RR)$ satisfying $\sup_{x \in \RR} (1 + x^2) |f(x)| < \infty$, we have
	\begin{equation}
		\label{eq:161}
		\lim_{n \to \infty} 
		\frac{1}{\rho_n}
		\int_\RR f \ud \nu_n =
		\lim_{n \to \infty} 
		\frac{1}{\rho_n}
		\int_\RR f \ud \eta_n = \int_\RR f \ud \omega
	\end{equation}
	where the measure $\omega$ is purely absolutely continuous with the density
	\[
		\frac{\ud \omega}{\ud x} \equiv 
		\frac{|\tr \frakX_0'(0)|}{\pi N \sqrt{-\discr \frakX_0(0)}}.
	\]
\end{theorem}
\begin{proof}
	Our strategy is to apply Corollary~\ref{cor:3}. In view of Proposition \ref{prop:8} we need to understand the asymptotic
	behavior of the sequence $(p_n(z) : n \geq 0)$ for $z \in \CC \setminus \RR$.
	
	Let $z_0 \in \CC_\varepsilon$, where $\varepsilon = \sign{\tr \frakX_0'(0)}$. Let $K \subset \CC_\varepsilon$ be a compact 
	set such that $z_0 \in \intr(K)$. We set $U = \intr(K)$. By Theorem~\ref{thm:5} there are linearly independent generalized 
	eigenvectors $\{ u^-(z), u^+(z) \}$ 
	associated with $z \in K$, such that for certain $M \geq 1$ and each $i \in \{0,1,\ldots,N-1\}$ the limits
	\[
		\varphi_{i}^+
		= 
		\lim_{k \to \infty}
		\frac{u^+_{kN+i_0}}{\prod_{M \leq j < k} \lambda_{j}^+}
		\quad\text{and}\quad
		\varphi_{i}^-
		=
		\lim_{k \to \infty}
		\frac{u^-_{kN+i_0}}{\prod_{M \leq j < k} \lambda_{j}^-},
	\]
	exist uniformly on $K$ where $\lambda_n^+$ and $\lambda_n^-$ are defined by \eqref{eq:63} with $Y_n = X_{nN}$. Moreover, 
	by \eqref{eq:42} there is $i_0 \in \{0,1,\ldots,N-1\}$ such that $\varphi_{i_0}^+(z_0) \neq 0$. By continuity there is 
	$V \subset U$, an open neighborhood of $z_0$, such that $\varphi_{i_0}^+(z) \neq 0$ for all $z \in V$.
	In view of Proposition~\ref{prop:7} we have
	\[
		\inf_{z \in K}
		\prod_{j \geq M}  
		\bigg| \frac{\lambda^+_j(z)}{\lambda_j^-(z)} \bigg| = \infty.
	\]
	Thus, by Proposition~\ref{prop:10} we have that for any $i \in \{0,1,\ldots, N-1\}$
	\begin{equation} \label{eq:201}
		\varphi_{i}(z) = 
		\lim_{k \to \infty} \frac{p_{kN+i}(z)}{\prod_{M \leq j < k} \lambda_j^+(z)}
		= \varphi_i^+(z) \frac{\Wrk(p(z), u^-(z))}{\Wrk(u^+(z), u^-(z))}.
	\end{equation}
	Since Carleman's condition implies that $A$ is self-adjoint, 
	by Proposition~\ref{prop:14}, we have that
	\[
		\Wrk(p(z), u^-(z)) \neq 0, \quad z \in V.
	\]
	Thus, by \eqref{eq:201} we have that $\varphi_{i_0}(z) \neq 0$ for any $z \in V$.
	In view of Proposition \ref{prop:8}, we obtain
	\[
		\lim_{k \to \infty} 
		\calC \big[ \tfrac{1}{\rho_{kN+i_0-1}}\nu_{kN+i_0-1} \big](z_0) = 
		-\lim_{k \to \infty} 
		\frac{1}{\rho_{kN+i_0-1}} 
		\sum_{j=M}^{k-1} \frac{(\lambda_j^+)'(z_0)}{\lambda_j^+(z_0)}.
	\]
	For $k \geq 0$, we set
	\[
		\rho_{0;k} = \sum_{j=0}^k \frac{1}{a_{jN}}.
	\]
	By \cite[formula (4.12)]{ChristoffelI}, for all $n \in \NN_0$, we have
	\begin{equation} 
		\label{eq:158}
		\lim_{k \to \infty} 
		\alpha_0
		\frac{\rho_{0;k}}{\rho_{kN+n}} = 
		\frac{1}{N},
	\end{equation}
	thus
	\[
		\lim_{k \to \infty} 
		\frac{1}{\rho_{kN+i_0-1}} 
		\sum_{j=M}^{k-1} \frac{(\lambda_j^+)'(z_0)}{\lambda_j^+(z_0)}
		=
		\frac{1}{N} 
		\lim_{k \to \infty} 
		\frac{1}{\alpha_{0} \rho_{0;k}} 
		\sum_{j=M}^{k-1} \frac{(\lambda_j^+)'(z_0)}{\lambda_j^+(z_0)}.
	\]
	Now, using the Stolz--Ces\`aro theorem, we obtain
	\begin{equation}
		\label{eq:123}
		\lim_{k \to \infty}
		\frac{1}{\alpha_0 \rho_{0;k}}
		\sum_{j=M}^{k-1} \frac{(\lambda_j^+)'(z_0)}{\lambda_j^+(z_0)}
		=
		\frac{1}{N}
		\lim_{k \to \infty}
		\frac{a_{kN}}{\alpha_0} \frac{(\lambda_k^+)'(z_0)}{\lambda_k^+(z_0)},
	\end{equation}
	therefore
	\begin{equation}
		\label{eq:157}
		\lim_{k \to \infty} 
		\calC \big[ \tfrac{1}{\rho_{kN+i_0-1}} \nu_{kN+i_0-1} \big](z_0) = 
		-\frac{1}{N}
		\lim_{k \to \infty} 
		\frac{a_{kN}}{\alpha_0} \frac{(\lambda_k^+)'(z_0)}{\lambda_k^+(z_0)}.
	\end{equation}
	Our next task is to compute the right-hand side of \eqref{eq:157}. 	By Proposition~\ref{prop:6}, we get
	\[
		\lim_{k \to \infty} 
		\frac{a_{kN}}{\alpha_0} \frac{(\lambda_k^+)'(z_0)}{\lambda_k^+(z_0)}
		=
		\frac{\abs{\tr \frakX_0'(0)}}{i \sqrt{-\discr \frakX_0(0)}},
	\]
	thus by \eqref{eq:157} we get
	\[
		\lim_{k \to \infty} 
		\calC \big[ \tfrac{1}{\rho_{kN+i_0-1}} \nu_{kN+i_0-1} \big](z_0) 
		= 
		-\frac{\tr \frakX_0'(0)}{i N \sqrt{-\discr \frakX_0(0)}}, 
		\quad\text{for all } z_0 \in \CC_{\varepsilon}.
	\]
	If $\varepsilon = -1$, we use the property $\overline{\calC[\eta](\overline{z})} = \calC[\eta](z)$ valid for any positive finite measure $\eta$ and any $z \in \CC \setminus \RR$ to finally get
	\[
		\lim_{k \to \infty} 
		\calC \big[\tfrac{1}{\rho_{kN+i_0-1}} \nu_{kN+i_0-1} \big](z) 
		= 
		-\frac{|\tr \frakX_0'(0)|}{i N \sqrt{-\discr \frakX_0(0)}}, \quad\text{for all } z \in \CC_+.
	\]
	In view of \eqref{eq:158} for $L \geq 0$ we have
	\begin{equation}
		\label{eq:121}
		\lim_{k \to \infty}
		\frac{\rho_{kN+i_0-1}}{\rho_{kN+i_0+L}}
		=
		\lim_{k \to \infty} 
		\frac{\rho_{0; k}}{\rho_{kN+i_0+L}} 
		\frac{\rho_{kN+i_0-1}}{\rho_{0; k}}
		= 1,
	\end{equation}
	thus, by Theorem~\ref{thm:4} we get
	\[
		\lim_{n \to \infty} 
		\calC \big[ \tfrac{1}{\rho_n} \nu_{n} \big](z) =
		-\frac{|\tr \frakX_0'(0)|}{i N \sqrt{-\discr \frakX_0(0)}}, \quad\text{for all } z \in \CC_+.
	\]
	Now, in view of Remark~\ref{rem:6} we can apply Corollary~\ref{cor:3} with any $z_0 \in \CC_+$ (see \eqref{eq:180}). Thus, by Theorem~\ref{thm:12}\eqref{thm:12:a}, we obtain \eqref{eq:161}. The proof is complete.
\end{proof}

The following theorem shows that in the setup of Example~\ref{ex:IIa} we have stronger conclusion compared to what we could 
obtain by applying Corollary~\ref{cor:1}. Moreover, we can drop the hypothesis \eqref{eq:11}.
\begin{theorem}[Case \ref{eq:PIIa}]
	\label{thm:8}
	Suppose that $(a_n : n \in \NN_0)$ and $(b_n : n \in \NN_0)$ are $N$-periodically modulated Jacobi parameters such that	
	$\frakX_0(0) = \varepsilon \Id$ with $\varepsilon \in \{-1, 1\}$. Assume that
	\[
		\bigg( a_n \Big( \frac{\alpha_{n-1}}{\alpha_n} - \frac{a_{n-1}}{a_{n-1}} \Big): n \in \NN \bigg),
		\bigg( a_n \Big( \frac{\beta_n}{\alpha_n} - \frac{b_n}{a_n} \Big): n \in \NN_0 \bigg),
		\bigg( \frac{1}{a_n} : n \in \NN_0 \bigg) \in \calD_1^N.
	\]
	Set $\Lambda_- = \{ x \in \RR : h(x) < 0 \}$ where
	\[
		h(z) = \lim_{j \to \infty} a_{jN+N-1}^2 \discr X_{jN}(z), \quad z \in \CC
	\]
	Let $(\eta_n : n \in \NN_0)$ and $(\nu_n : n \in \NN_0)$ be defined in \eqref{eq:25} and set
	\[
		\rho_n = \sum_{j=0}^n \frac{\alpha_j}{a_j}.
	\]
	Then for each function $f \in \calC_0(\RR)$ satisfying $\sup_{x \in \RR} (1+x^2) |f(x)| < \infty$ we have
	\[
		\lim_{n \to \infty} 
		\frac{1}{\rho_n} \int_\RR f \ud \nu_n =  
		\lim_{n \to \infty} 
		\frac{1}{\rho_n}
		\int_\RR f \ud \eta_n = \int_\RR f \ud \omega
	\]
	where the measure $\omega$ is purely absolutely continuous with the density
	\[
		\frac{\ud \omega}{\ud x}
		= 
		\frac{1}{4 \pi N \alpha_{N-1}} \frac{|h'(x)|}{\sqrt{-h(x)}} \mathds{1}_{\Lambda_-}(x).
	\]
\end{theorem}
\begin{proof}
	Thanks to Proposition~\ref{prop:11}, $h$ is a quadratic real polynomial with real roots. Hence, for $z \in \CC \setminus \RR$,
	$\Im h(z) = 0$ if and only if $h'(\Re z) \neq 0$. Let $z_0 \in \scrD$ where
	\[
		\scrD
		=
		\big\{z \in \CC _+ : h(z) \in \CC \setminus (-\infty, 0] \text{ and } \varepsilon h'(\Re z) > 0 \big\}.
	\]
	Let $K \subset \scrD$ be a compact set such that $z_0 \in \intr(K)$. 
	Using Theorem~\ref{thm:5s} in place of Theorem~\ref{thm:5}, 
	we can repeat the line of reasoning used in the proof of Theorem~\ref{thm:7} to conclude that there is 
	$i_0 \in \{0, 1, \ldots, N-1\}$ such that
	\[
		\lim_{k \to \infty} 
		\calC \big[ \tfrac{1}{\rho_{kN+i_0-1}} \nu_{kN+i_0-1} \big](z_0) =
		-
		\lim_{k \to \infty}
		\frac{1}{\rho_{kN+i_0-1}} \sum_{j = M}^{k-1} 
		\frac{(\lambda_k^+)'(z_0)}{\lambda_k^+(z_0)}.
	\]
	Next, by Proposition \ref{prop:9}
	\[
		\lim_{k \to \infty} a_{kN+N-1} \frac{(\lambda_k^+)'(z_0)}{\lambda_k^+(z_0)} = 
		\frac{1}{4} \frac{h'(z_0)}{\sqrt{h(z_0)}},
	\]
	and since
	\[
		\lim_{k \to \infty} \frac{a_{kN}}{a_{kN+N-1}} = \frac{\alpha_0}{\alpha_{N-1}},
	\]
	by \eqref{eq:123} we obtain
	\begin{align*}
		\lim_{k \to \infty} 
		\calC \big[ \tfrac{1}{\rho_{kN+i_0-1}} \nu_{kN+i_0-1} \big](z_0) 
		&=
		-\frac{1}{N \alpha_0} \lim_{k \to \infty} \frac{a_{kN}}{a_{kN+N-1}} a_{kN +N - 1 }
		\frac{(\lambda_k^+)'(z_0)}{\lambda_k^+(z_0)} \\
		&=
		-\frac{1}{4 N \alpha_{N-1}} \frac{h'(z_0)}{\sqrt{h(z_0)}}.
	\end{align*}
	Now, using \eqref{eq:121} and Theorem~\ref{thm:4}, we conclude that
	\[
		\lim_{n \to \infty} 
		\calC \big[\tfrac{1}{\rho_n} \nu_n \big](z) 
		=
		-\frac{1}{4 N \alpha_{N-1}} \frac{h'(z)}{\sqrt{h(z)}},
		\quad\text{for } z\in \scrD.
	\]
	Since $\overline{\calC[\eta](\overline{z})} = \calC[\eta](z)$ for any positive finite measure~$\eta$ 
	and any $z \in \CC \setminus \RR$, we get
	\[
		\lim_{n \to \infty} 
		\calC \big[\tfrac{1}{\rho_n} \nu_n \big](z) 
		=
		-\frac{1}{4 N \alpha_{N-1}} \frac{h'(z)}{\sqrt{h(z)}},
		\quad\text{for } z\in \scrD \cup \overline{\scrD}.
	\]
	In view of Lemma \ref{lem:3}, there is an analytic function $g: \CC_+ \rightarrow
	\CC_+ \cup \RR$ 
	such that
	\[
		g(z) = 
		\lim_{n \to \infty} 
		\calC \big[\tfrac{1}{\rho_n} \nu_n \big](z), \quad z \in \CC_+
	\]
	and
	\begin{equation} \label{eq:212}
		g(z) = -\frac{1}{4 N \alpha_{N-1}} \frac{h'(z)}{\sqrt{h(z)}},
		\quad\text{for } z \in \CC_+ \cap (\scrD \cup \overline{\scrD}).
	\end{equation}
	Notice that the right-hand side of \eqref{eq:212} depends analytically on $z \in \CC_+ \setminus h^{-1}((-\infty,0])$. However, since $h$ is a real polynomial of degree $2$ with real roots we can verify that
	$h^{-1}(\RR) \subset \RR \cup \{ z \in \CC : h'(\Re z) = 0 \}$. So the right-hand side of \eqref{eq:212} is actually analytic on $\{ z \in \CC_+ : h'(\Re z) \neq 0 \}$, so the equality \eqref{eq:212} holds on the set $\{z \in \CC_+ : h'(\Re z) \neq 0 \}$.
	Now, in view of Remark~\ref{rem:6} we can apply Corollary~\ref{cor:3} with some explicit $z_0 \in \CC_+$ 
	(see \eqref{eq:186}). Thus, by Theorem~\ref{thm:12}\eqref{thm:12:c}, we complete the proof.
\end{proof}

The following theorem covers a more general setup than Example~\ref{ex:IIb}. 
\begin{theorem}[Case \ref{eq:PIIb}] 
	\label{thm:9}
	Let $N \geq 1$. Let $(\gamma_n : n \geq 0)$ be a sequence of positive numbers tending to infinity and satisfying 
	\eqref{eq:83a} and \eqref{eq:83b}. Let $(a_n : n \geq 0)$ and $(b_n : n \geq 0)$ be $\gamma$-tempered $N$-periodically
	modulated Jacobi parameters such that $\frakX_0(0)$ is non-trivial parabolic element. Suppose that \eqref{eq:108b} 
	holds true with $\varepsilon = \sign{\tr \frakX_0(0)}$. Then the following limit exists
	\[
		h(z) = \lim_{j \to \infty} \gamma_{jN+N-1} \discr X_{jN}(z), \quad z \in \CC
	\]
	and defines a polynomial of degree at most $1$. 
	Let $(\eta_n : n \in \NN_0)$ and $(\nu_n : n \in \NN_0)$ be defined in \eqref{eq:25} and set	
	\[
		\rho_n = \sum_{j=0}^n \frac{\sqrt{\alpha_j \gamma_j}}{a_j}.
	\]
	If $\Lambda_- = \{ x \in \RR : h(x) < 0 \} \neq \emptyset$ and $\rho_n \to \infty$, 
	then for each $f \in \calC_0(\RR)$ satisfying $\sup_{x \in \RR} (1+x^2)|f(x)| < \infty$, we have
	\[
		\lim_{n \to \infty} 
		\frac{1}{\rho_n}
		\int_\RR f \ud \nu_n =  
		\lim_{n \to \infty} 
		\frac{1}{\rho_n}
		\int_\RR f \ud \eta_n = \int_\RR f \ud \omega
	\]
	where the measure $\omega$ is purely absolutely continuous with the density
	\[
		\frac{\ud \omega}{\ud x}
		=
		\frac{\sqrt{\alpha_{N-1}}}{\pi N} \frac{|\tr \frakX_0'(0)|}{\sqrt{-h(x)}} \ind{\Lambda_-}(x).
	\]
\end{theorem}
\begin{proof}
	The proof follows the same line of reasoning as Theorem \ref{thm:8}. Since $h$ is a linear map with real coefficients,
	$\Im h(z) = 0$ if only if $\Im z = 0$. Let $z_0 \in \scrD$ where
	\[
		\scrD = \big\{z \in \CC_\varsigma : h(\Re z) < 0\big\}
	\]
	with $\varsigma = \sign{\tr \frakX_0'(0)}$.

	Let $K$ be a compact subset of $\scrD$ such that $z_0 \in \intr(K)$. Using Theorem~\ref{thm:5h}
	in place of Theorem~\ref{thm:5}, we can repeat the line of reasoning used in the proof of 
	Theorem~\ref{thm:7} to conclude that there is 
	$i_0 \in \{0, 1, \ldots, N-1\}$ such that
	\[
		\lim_{k \to \infty} 
		\calC \big[\tfrac{1}{\rho_{kN+i_0-1}} \nu_{kN+i_0-1} \big](z_0) =
		-\lim_{k \to \infty}
		\frac{1}{\rho_{kN+i_0-1}} \sum_{j = M}^{k-1} 
		\frac{(\lambda_k^+)'(z_0)}{\lambda_k^+(z_0)}
	\]
	where $\lambda_k^+$ are eigenvalues of $Y_k = Z_{k+1}^{-1} X_{kN} Z_k$. 

	At this stage we have two cases depending on the value of $\mathfrak{t} \in \{0, 1\}$. We treat $\mathfrak{t} = 1$
	only. By \eqref{eq:40}, we get
	\[
		\lim_{k \to \infty} 
		\frac{a_{kN+N-1}}{\sqrt{\alpha_{N-1} \gamma_{kN+N-1}}}
		\frac{(\lambda_k^+)'(z_0)}{\lambda_k^+(z_0)}
		=
		\varepsilon \sqrt{\alpha_{N-1}}
		\frac{\tr \frakX_0'(0)}{\sqrt{h(z_0)}}.
	\]
	For
	\[
		\rho_{0;k} = \sum_{j = 0}^k \frac{\sqrt{\gamma_{jN}}}{a_{jN}}
	\]
	by the Stolz--Ces\`aro theorem, 
	\begin{align*}
		\lim_{k \to +\infty} 
		\frac{\rho_{0; k}}{\rho_{kN+i_0-1}} 
		&=\lim_{k \to +\infty} 
		\frac{\frac{\sqrt{\gamma_{jN}}}{a_{jN}}}
		{\sum_{j = 0}^{N-1}
		\frac{\sqrt{\alpha_{i_0+j} \gamma_{kN+i_0+j}}}{a_{kN+i_0+j}}} \\
		&=\frac{1}{N \sqrt{\alpha_0}}
	\end{align*}
	where we have also used
	\[
		\lim_{k \to \infty} 
		\frac{\gamma_{kN}}{\gamma_{kN+i_0+j}}
		=
		\lim_{k \to \infty}
		\frac{a_{kN}}{a_{kN+i_0+j}}
		=
		\frac{\alpha_0}{\alpha_{i_0+j}}, \quad j=0,1,\ldots,N-1,
	\]
	Next, another application of the Stolz--Ces\`aro theorem leads to
	\[
		\lim_{k \to \infty}
		\frac{1}{\sqrt{\alpha_0} \rho_{0;k}}
		\sum_{j=M}^{k-1} \frac{(\lambda_j^+)'(z_0)}{\lambda_j^+(z_0)}
		=
		\lim_{k \to \infty}
		\frac{a_{kN}}{\sqrt{\alpha_0 \gamma_{kN}}}
		\frac{(\lambda_k^+)'(z_0)}{\lambda_k^+(z_0)}.
	\]
	Hence,
	\begin{align*}
		\lim_{k \to \infty} 
		\calC \big[\tfrac{1}{\rho_{kN+i_0-1}} \nu_{kN+i_0-1} \big](z_0) 
		&=
		-\sqrt{\alpha_0} \lim_{k \to +\infty}
		\frac{\rho_{0; k}}{\rho_{kN+i_0-1}} 
		\frac{1}{\sqrt{\alpha_0} \rho_{0; k}} \sum_{j=M}^{k-1} \frac{(\lambda_j^+)'(z_0)}{\lambda_j^+(z_0)} \\
		&=
		-\frac{1}{N}
		\lim_{k \to \infty} 
		\frac{a_{kN}}{a_{kN+N-1}}
		\sqrt{ \frac{\gamma_{kN+N-1}}{\gamma_{kN}} }
		\sqrt{\frac{\alpha_{N-1}}{\alpha_0}}
		\frac{a_{kN+N-1}}{\sqrt{\alpha_{N-1} \gamma_{kN+N-1}}}
		\frac{(\lambda_k^+)'(z_0)}{\lambda_k^+(z_0)} \\
		&=
		-\varepsilon \frac{\sqrt{\alpha_{N-1}}}{N}
		\frac{\tr \frakX_0'(0)}{\sqrt{h(z_0)}}.	
	\end{align*}
	By \cite[Corollary 3.3]{jordan2} we have $h(x) = 4 \alpha_{N-1} \tau(x)$. Thus, by \eqref{eq:188} we have 
	$h'(x) = 4 \varepsilon \alpha_{N-1} (\tr \frakX_0'(0))$. Now, using \eqref{eq:121} and Theorem \ref{thm:4}, 
	we conclude that
	\[
		\lim_{n \to \infty} 
		\calC \big[\tfrac{1}{\rho_n} \nu_n \big](z) 
		=
		-\frac{1}{4 N \sqrt{\alpha_{N-1}}} \frac{h'(z)}{\sqrt{h(z)}},
		\quad\text{for all } z\in \scrD.
	\]
	Since $\overline{\calC[\eta](\overline{z})} = \calC[\eta](z)$ for any positive finite measure~$\eta$ 
	and any $z \in \CC \setminus \RR$, we get
	\[
		\lim_{n \to \infty} 
		\calC \big[\tfrac{1}{\rho_n} \nu_n \big](z) 
		=
		-\frac{1}{4 N \sqrt{\alpha_{N-1}}} \frac{h'(z)}{\sqrt{h(z)}},
		\quad\text{for all } z \in \scrD \cup \overline{\scrD}.
	\]
	In view of Lemma \ref{lem:3}, there is an analytic function $g: \CC_+ \rightarrow \CC_+ \cup \RR$ such that
	\[
		g(z) = 
		\lim_{n \to \infty} 
		\calC \big[\tfrac{1}{\rho_n} \nu_n \big](z), \quad z \in \CC_+
	\]
	and
	\begin{equation} \label{eq:211}
		g(z) = -\frac{1}{4 N \sqrt{\alpha_{N-1}}} \frac{h'(z)}{\sqrt{h(z)}},
		\quad\text{for all } z \in \CC_+ \cap (\scrD \cup \overline{\scrD}).
	\end{equation}
	Notice that the right-hand side of \eqref{eq:211} depends analytically on $z \in \CC_+ \setminus h^{-1}((-\infty,0])$.
	However, $h^{-1}((-\infty,0]) \subset h^{-1}(\RR) \subset \RR$, so the right-hand side of \eqref{eq:211} is actually analytic on $\CC_+$,
	and consequently, the equality \eqref{eq:211} holds for any $z \in \CC_+$.
	By Remark~\ref{rem:6}, we can apply Corollary~\ref{cor:3} with some explicit $z_0 \in \CC_+$ (see \eqref{eq:187}). 
	Thus, by Theorem~\ref{thm:12}\eqref{thm:12:b}, we complete the proof. 

	Analogously one can treat the case $\mathfrak{t} = 0$.
\end{proof}

The following corollary gives a constructive proof of absolute continuity of $\mu$ on $\Lambda_-$, which has been proven 
in \cite{jordan2} by means of subordinacy theory.
\begin{corollary}
	Under the hypotheses of Theorem~\ref{thm:9} the measure $\mu$ is purely absolutely continuous on $\Lambda_-$ with
	the density
	\[
		\mu'(x) = 
		\frac{\sqrt{\alpha_{N-1}}}{\pi N} 
		\frac{\abs{\tr \frakX_0'(0)}}{\sqrt{-h(x)}} \frac{1}{g(x)}, 
		\quad x \in \Lambda_-
	\]
	where
	\[
		g(x) = 
		\lim_{n \to \infty} 
		\frac{1}{\rho_n} K_n(x,x), \quad x \in \Lambda_-.
	\]
	In particular, $\mu'$ is continuous and positive on $\Lambda_-$.
\end{corollary}
\begin{proof}
	It readily follows from \cite[Theorem B]{jordan2} and Corollary~\ref{cor:3}.
\end{proof}

\appendix

\section{Computation of the limiting measure from its Cauchy transform}
The following result, which is a compilation of \cite[Theorems 7.37 and 7.46]{Lukic2022} is crucial in our task of recovering a measure from its Cauchy transform.
\begin{theorem} \label{thm:11}
Let $\omega$ be a positive finite measure on $\RR$. Then the limit
\begin{equation} \label{eq:177}
	w(x) = \frac{1}{\pi} 
	\lim_{\epsilon \downarrow 0} \Im \big( \calC[\omega](x+i\epsilon) \big)
\end{equation}
exists Lebesgue a.e. and $\omega$ a.e. with value in $[0, \infty]$ and
\begin{equation} \label{eq:178}
	\omega = w \ud x + \omega_s,
\end{equation}
where $\omega_s$ is a singular measure with respect to the Lebesgue measure and satisfies $\supp(\omega_s) \subseteq w^{-1}[\{\infty\}]$. Moreover, for any $x_0 \in \RR$ and any $\delta \in (0,\pi)$
\begin{equation} \label{eq:179}
	\omega(\{x_0\}) = 
	\lim_{\substack{z \to x_0\\\delta \leq \Arg(z-x_0) \leq \pi -\delta}}
	(x_0-z) \calC[\omega](z)
\end{equation}
\end{theorem}

In view of formulas \eqref{eq:177} and \eqref{eq:178} it is crucial to understand the boundary behavior of the imaginary part of the Cauchy transform of the measure $\tilde{\omega}_{z_0}$ defined by \eqref{eq:28}. The following result says that it can be expressed simply in terms of the function $g$ itself.

\begin{proposition} 
Let $z_0 = x_0 + i y_0$ for some $x_0 \in \RR$ and $y_0 > 0$ be given and let the measure $\tilde{\omega}_{z_0}$ be defined by the Cauchy transform~\eqref{eq:29}. Then
\begin{align}
	\label{eq:174}
	\lim_{\epsilon \downarrow 0} 
	\Im \big( \calC[\tilde{\omega}_{z_0}](x+i\epsilon) \big) =
	\frac{1}{|x-z_0|^2} 
	\lim_{\epsilon \downarrow 0} 
	\Im \big( g(x+i\epsilon) \big).
\end{align}
\end{proposition}
\begin{proof}
Indeed, by \eqref{eq:29} we have
\[
	\Im \big( \calC[\tilde{\omega}_{z_0}](z) \big) =
	\Im \bigg( \frac{g(z)}{(z-x_0)^2 + y_0^2} \bigg) -
	\Re \big( g(z_0) \big) \Im \bigg( \frac{1}{(z-x_0)^2 + y_0^2} \bigg) -
	\frac{\Im \big( g(z_0) \big)}{y_0} \Im \bigg( \frac{z-x_0}{(z-x_0)^2 + y_0^2} \bigg).
\]
Thus, by the continuity we have
\begin{equation} 
	\label{eq:175}
	\lim_{\epsilon \downarrow 0} \Im \big( \calC[\tilde{\omega}_{z_0}](x+i\epsilon) \big) =
	\lim_{\epsilon \downarrow 0} \Im \bigg( \frac{g(x+i\epsilon)}{(x-x_0 + i\epsilon)^2 + y_0^2} \bigg)
\end{equation}
Next, 
\begin{align}
	\label{eq:176}
	\Im \bigg( \frac{g(x+i\epsilon)}{(x-x_0+i\epsilon)^2 + y_0^2} \bigg) 
	&= 
	\frac{\Re \big( (x-x_0+i\epsilon)^2 + y_0^2 \big)}{|(x-x_0+i\epsilon)^2 + y_0^2|^2} \Im \big( g(x+i\epsilon) \big) \\
	\nonumber
	&\phantom{=}-
	\frac{\Im \big( (x-x_0+i\epsilon)^2 + y_0^2 \big)}{|(x-x_0+i\epsilon)^2 + y_0^2|^2} \Re \big( g(x+i\epsilon) \big).
\end{align}
Since $g$ is Herglotz function, we have
\[
	\lim_{\epsilon \to 0^+} \epsilon \Re \big( g(x+i\epsilon) \big) = 0,
\]
see \cite[Theorem 2.3(iv)]{Gesztesy2000}. Thus, the second term on the right-hand side of \eqref{eq:176} tends to $0$. Therefore, by the continuity we get
\[
	\lim_{\epsilon \downarrow 0} \Im \bigg( \frac{g(x+i\epsilon)}{(x-x_0+i\epsilon)^2 + y_0^2} \bigg) =
	\frac{1}{(x-x_0)^2 + y_0^2} \lim_{\epsilon \downarrow 0} \Im \big( g(x+i\epsilon) \big),
\]
which in view of \eqref{eq:175} implies \eqref{eq:174}.
\end{proof}

In the following Theorem we compute the measure $\tilde{\omega}_{z_0}$ in the cases we have encountered in this article.

\begin{theorem} \label{thm:12}
Let $z_0 = x_0 + i y_0$ for some $x_0 \in \RR$ and $y_0 > 0$ be given and let the measure 
$\tilde{\omega}_{z_0}$ be defined by the Cauchy transform~\eqref{eq:29}. Define $\ud \omega = |x-z_0|^2 \ud \tilde{\omega}_{z_0}(x)$.
\begin{enumerate}[label=\rm (\arabic*), start=1, ref=\arabic*]
\item \label{thm:12:a}
If $g(z) \equiv i$, then 
\[
	\frac{\ud \omega}{\ud x} \equiv \frac{1}{\pi}.
\]

\item \label{thm:12:b}
If $g(z) = \frac{-h'(z)}{\sqrt{h(z)}}$, for $h$ being a real polynomial of degree $1$, then for $\Lambda_- = h^{-1}[(-\infty, 0)]$ we have
\[
	\frac{\ud \omega}{\ud x} = 
	\frac{1}{\pi}
	\frac{\abs{h'(x)}}{\sqrt{-h(x)}}
	\ind{\Lambda_-}(x).
\]

\item \label{thm:12:c}
If $g(z) = \frac{-h'(z)}{\sqrt{h(z)}}$, for $h$ being a real polynomial of degree $2$ with negative leading coefficient and having only real zeros, and $z \in \CC_+ \setminus \scrD$, where $\scrD = \{ z \in \CC \setminus \RR : h'(\Re z) = 0 \}$, then for $\Lambda_- = h^{-1}[(-\infty, 0)]$ we have
\[
	\frac{\ud \omega}{\ud x} =
	\frac{1}{\pi} 
	\frac{|h'(x)|}{\sqrt{-h(x)}} \mathds{1}_{\Lambda_-}(x).
\]
\end{enumerate}
\end{theorem}
\begin{proof} Let $z_0 = x_0 + i y_0$ for $x_0 \in \RR$ and $y_0 > 0$ be given.

\noindent
\textbf{Proof of \eqref{thm:12:a}}. Since $\Im g(z) \equiv 1$, the formula~\eqref{eq:174} gives
\[
	\lim_{\epsilon \downarrow 0} 
	\frac{1}{\pi}
	\Im \big( \calC[\tilde{\omega}_{z_0}](x + i \epsilon) \big) =
	\frac{1}{\pi} \frac{1}{(x-x_0)^2 + y_0^2}, \quad x \in \RR.
\]
Thus, by Theorem~\ref{thm:11} the measure $\tilde{\omega}_{z_0}$ is purely absolutely continuous and
\[
	\frac{\ud \tilde{\omega}_{z_0}}{\ud x} = 
	\frac{1}{\pi} \frac{1}{(x-x_0)^2 + y_0^2},
\]
which gives the requested formula for the measure $\omega$. Notice that we have
\begin{equation} \label{eq:180}
	\tilde{\omega}_{z_0}(\RR) = 
	\frac{1}{\pi} \int_\RR \frac{1}{(x-x_0)^2 + y_0^2} \ud x =
	\frac{1}{y_0} = 
	\frac{1}{y_0} \Im \big( g(z_0) \big).
\end{equation}

\noindent
\textbf{Proof of \eqref{thm:12:b}}. Again, we shall use the formula~\eqref{eq:174}. Because $h$ is linear, we can write
\[
	h(x + i \epsilon) = h(x) + i h'(0) \epsilon.
\]
Now, we have three cases.

\noindent
{\bf Case $h(x) > 0$}: Then
\[
	\Log h(x+i\epsilon) = \log \abs{h(x+i\epsilon)} + i \arctan\bigg(\frac{h'(0) \epsilon}{h(x)} \bigg),
\]
thus
\[
	\lim_{\epsilon \to 0^+}
	\Im \bigg( \frac{h'(x+i\epsilon)}{ \sqrt{h(x+i\epsilon)}} \bigg) = 0.
\]

\noindent
{\bf Case $h(x) < 0$}: Then
\[
	\Log h(x+i\epsilon) = \log \abs{h(x+i\epsilon)} + i \arctan\bigg(\frac{h'(0) \epsilon}{h(x)} \bigg) 
	+
	i
	\begin{cases}
		\pi &\text{if } h'(0) \geq 0, \\
		-\pi &\text{if } h'(0) < 0.
	\end{cases}
\]
Hence,
\[
	\lim_{\epsilon \to 0^+} 
	\Im \bigg(\frac{-h'(0)}{\sqrt{h(x+i\epsilon)}}\bigg)
	=
	\frac{\abs{h'(0)}}{\sqrt{-h(x)}}.
\]

\noindent
{\bf Case $h(x) = 0$}: Then $h(x+i\epsilon) = i h'(0) \epsilon$, and so
\[
	\big| \sqrt{h(x + i \epsilon)} \big| = \sqrt{\epsilon \abs{h'(0)}}.
\]
Hence,
\[
	\lim_{\epsilon \to 0^+} \epsilon \bigg|\frac{h'(0)}{\sqrt{h(x+i\epsilon)}}\bigg| = 0,
\]
which in view of \eqref{eq:29} and \eqref{eq:179} shows that there is no atom of $\tilde{\omega}_{z_0}$ at $h(x) = 0$. Summarizing we obtain
\[
	\frac{\ud \tilde{\omega}_{z_0}}{\ud x} = 
	\frac{1}{\pi}
	\frac{1}{(x-x_0)^2 + y_0^2}
	\frac{\abs{h'(0)}}{\sqrt{-h(x)}}
	\ind{\Lambda_-}(x),
\]
which gives the requested formula for $\omega$. Now, let us write $h$ in the form $h(x) = ax+b$ for some $a \neq 0$ and $b \in \RR$. Notice that
\[
	\frac{\ud \tilde{\omega}_{-b/a+i}}{\ud x} =
	\frac{\sqrt{|a|}}{\pi} 
	\frac{1}{(x+\tfrac{b}{a})^2 + 1} 
	\frac{1}{\sqrt{-\sign{a}(x+\tfrac{b}{a})}} \mathds{1}_{\Lambda_-}(x).
\]
Then by using the variable $s = -\sign{a}(x + \frac{b}{a})$ we get
\[
	\tilde{\omega}_{-b/a+i}(\RR) 
	=
	\frac{\sqrt{|a|}}{\pi}
	\int_{0}^{\infty}
	\frac{1}{s^2 + 1} 
	\frac{1}{\sqrt{s}} \ud s
	= 
	\sqrt{\frac{|a|}{2}}.
\]
On the other hand, 
\[
	g \big( -\tfrac{b}{a} + i \big) 
	= 
	\frac{-a}{\sqrt{a i}} 
	=
	\frac{-\sign{a} \sqrt{|a|}}{\sqrt{\sign{a} i}}
	=
	\sqrt{\frac{|a|}{2}} \cdot
	\begin{cases}
		-1 + i & a > 0 \\
		 1 + i & a < 0
	\end{cases}.
\]
Thus,
\begin{equation} \label{eq:187}
	\tilde{\omega}_{-b/a+i}(\RR) = 
	\Im \Big( g \big( -\tfrac{b}{a} + i \big) \Big).
\end{equation}

\noindent
\textbf{Proof of \eqref{thm:12:c}}. Again, we shall use the formula~\eqref{eq:174}. Since $h$ a polynomial of degree two with negative leading coefficient, we have $2 a = h''(0) < 0$, hence
\[
	h(x+i\epsilon) = h(x) + i \epsilon h'(x) - a \epsilon^2.
\]
We start with $h(x) \neq 0$ and $h'(x) \neq 0$. Then
\[
	\lim_{\epsilon \to 0^+} \epsilon \Re \sqrt{h(x+i\epsilon)} = 0.
\]
Moreover,
\[
	g(x + i \epsilon) = 
	\frac{h'(x+i\epsilon)}{\abs{h(x+i\epsilon)}} 
	\Big(\Re \sqrt{h(x+i\epsilon)} - i \Im \sqrt{h(x+i\epsilon)}\Big),
\]
thus
\[
	\Im \big( g(x + i\epsilon) \big) = 
	\frac{1}{\abs{h(x+i\epsilon)}}
	\Big( 2a\epsilon \Re \sqrt{h(x+i\epsilon)} - 
	h'(x) \Im \sqrt{h(x+i\epsilon)} \Big).
\]
Therefore, we have
\begin{equation} \label{eq:181}
	\lim_{\epsilon \downarrow 0}
	\Im \big( g(x + i\epsilon) \big) =
	\frac{-h'(x)}{\abs{h(x)}} 
	\lim_{\epsilon \downarrow 0} \Im \sqrt{h(x+i\epsilon)},
\end{equation}
provided that $h'(x) \neq 0$ and $h(x) \neq 0$. To compute the last limit, we consider two cases.

\noindent
{\bf Case $h(x) < 0$}: Then for $\epsilon$ sufficiently small we have $h(x) - a \epsilon^2 < 0$, and
\[
	\Log h(x+i\epsilon) = \log \abs{h(x+i\epsilon)} + i \arctan\bigg(\frac{\epsilon h'(x)}{h(x) - a \epsilon^2}\bigg)
	+ 
	\begin{cases}
		i \pi & \text{if } h'(x) > 0, \\
		-i\pi & \text{if } h'(x) < 0.
	\end{cases}
\]
Hence,
\[
	\sqrt{h(x+i\epsilon)} = -\sqrt{\abs{h(x+i\epsilon)}} 
	\exp\bigg(\frac{i}{2} \arctan\bigg(\frac{\epsilon h'(x)}{h(x) - a \epsilon^2}\bigg)\bigg)
	\sign{h'(x)} i,
\]
and so
\begin{equation} \label{eq:182}
	\lim_{\epsilon \downarrow 0} 
	\Im \sqrt{h(x+i\epsilon)} = 
	-\sqrt{-h(x)} \sign{ h'(x) }.
\end{equation}

\noindent
{\bf Case $h(x) > 0$ and $h'(x) \neq 0$}: Then
\[
	\Log h(x+i\epsilon) = \log \abs{h(x+i\epsilon)} + i \arctan\bigg(\frac{\epsilon h'(x)}{h(x) - \epsilon^2}\bigg),
\]
and thus
\[
	\sqrt{h(x+i\epsilon)} = \sqrt{\abs{h(x+i\epsilon)}} 
	\exp\bigg(\frac{i}{2} \arctan\bigg(\frac{\epsilon h'(x)}{h(x) - \epsilon^2}\bigg)\bigg).
\]
Therefore,
\begin{equation} \label{eq:183}
	\lim_{\epsilon \downarrow 0}  \Im \sqrt{h(x+i\epsilon)} = 0.
\end{equation}

To completely describe the measure $\tilde{\omega}_{z_0}$ it remains to show that there are no atoms at $h(x)=0$ nor $h'(x) = 0$. It is enough to prove that the right-hand side of \eqref{eq:179} is equal to zero.

\noindent
{\bf Case $h(x) = 0$ and $h'(x) \neq 0$}:
Then
\[
	h(x + i\epsilon) = i \epsilon h'(x) - a \epsilon^2.
\]
Since
\[
	\big| \sqrt{h(x+i\epsilon)} \big| = 
	\sqrt{\abs{h(x+i\epsilon)}} \geq 
	\sqrt{\epsilon \abs{h'(x)}},
\]
we get
\[
	\bigg|\frac{h'(x+i\epsilon)}{\sqrt{h(x+i\epsilon)}}\bigg| \leq
	\frac{1}{\sqrt{\epsilon}}
	\frac{|h'(x) + i 2a \epsilon|}{\sqrt{|h'(x)|}}.
\]
Thus,
\begin{equation} \label{eq:184}
	\lim_{\epsilon \downarrow 0} \epsilon |g(x+i\epsilon)| = 0.
\end{equation}

\noindent
{\bf Case $h'(x) = 0$}:
We shall prove
\begin{equation} \label{eq:185}
	\lim_{\epsilon \downarrow 0} 
	\epsilon \abs{g(x + \epsilon + i \epsilon)} = 0.
\end{equation}
Since
\[
	h(x + \epsilon + i \epsilon) = h(x) + i 2 a \epsilon^2,
\]
and
\[
	h'(x+\epsilon + i\epsilon) = 2 a \epsilon(1+i),
\]
we have
\[
	\bigg|\frac{h'(x+\epsilon + i \epsilon)}{\sqrt{h(x + \epsilon + i \epsilon)}} \bigg|
	=
	\frac{-2 a \epsilon \sqrt{2}}{\sqrt[4]{(h(x))^2 + 4 a^2 \epsilon^4}}
	\leq
	2 \sqrt{-a}.
\]
Therefore, \eqref{eq:185} follows.

In view of \eqref{eq:29} and \eqref{eq:179} together with \eqref{eq:184} and \eqref{eq:185} the measure $\tilde{\omega}_{z_0}$ has no atoms for $x$ satisfying $h(x)=0$ or $h'(x) = 0$. Next, by \eqref{eq:174} and \eqref{eq:178} together with \eqref{eq:181}, \eqref{eq:182} and \eqref{eq:183} the measure $\tilde{\omega}_{z_0}$ is purely absolutely continuous with the density given by
\[
	\frac{\ud \tilde{\omega}_{z_0}}{\ud x} =
	\frac{1}{\pi} 
	\frac{1}{(x-x_0)^2 + y_0^2}
	\frac{|h'(x)|}{\sqrt{-h(x)}} \mathds{1}_{\Lambda_-}(x),
\]
which gives the requested formula for $\omega$. Now, let us write $h$ in the form $h(x) = -a (x-x_c)^2 + a q$, where $x_c \in \RR$, $a > 0$ and $q \geq 0$. Then
\[
	\frac{\ud \tilde{\omega}_{x_c+i}}{\ud x} =
	\frac{2 \sqrt{a}}{\pi} 
	\frac{1}{(x-x_c)^2 + 1}
	\frac{|x-x_c|}{\sqrt{(x-x_c)^2 -q}} \mathds{1}_{\Lambda_-}(x).
\]
Then by the symmetry of the density around $x=x_c$ and using the variable $s=x-x_c$ and then $t=s^2-q$
\begin{align*}
	\tilde{\omega}_{x_c+i}(\RR) 
	&=
	\frac{4 \sqrt{a}}{\pi}
	\int_{\sqrt{q}}^\infty \frac{1}{s^2+1} \frac{y}{\sqrt{s^2-q}} \ud s \\
	&=
	\frac{2 \sqrt{a}}{\pi}
	\int_{0}^\infty \frac{1}{t+q+1} \frac{1}{\sqrt{t}} \ud t \\
	&=
	\frac{2\sqrt{a}}{\sqrt{q+1}}.
\end{align*}
On the other hand, since $g$ is a Herglotz function, it is continuous at $z=x_c+i \in \CC_+ \cap \scrD$, so we have
\[
	g(x_c+i) = 
	\lim_{\epsilon \to 0} 
	\frac{-h'(x_c+i+\epsilon)}{\sqrt{h(x_c+i+\epsilon)}} =
	\lim_{\epsilon \to 0} 
	\frac{2 a(i+\epsilon)}{\sqrt{a} \sqrt{q + 1 - \epsilon^2 - 2 \epsilon i}} 
	=
	\frac{2 i \sqrt{a}}{\sqrt{q+1}},
\]
where the last equality follows from the continuity of the principal square root in the right-half plane.
Thus, we have
\begin{equation} \label{eq:186}
	\tilde{\omega}_{x_c+i}(\RR) = \Im \big( g(x_c + i) \big). \qedhere
\end{equation}
\end{proof}

\begin{remark} \label{rem:6}
Notice that in the proof of Theorem~\ref{thm:12} we have shown that for all of the cases \eqref{thm:12:a}, \eqref{thm:12:b} 
and \eqref{thm:12:c} there is explicit $z_0 = x_0 + i y_0$ for some $x_0 \in \RR$ and $y_0 > 0$ such that 
$\tilde{\omega}_{z_0}(\RR) = \frac{1}{y_0} \Im \big( g(z_0) \big)$ holds true, see \eqref{eq:180}, \eqref{eq:187} and 
\eqref{eq:186}, respectively.
\end{remark}

\begin{bibliography}{jacobi}
	\bibliographystyle{amsplain}
\end{bibliography}

\end{document}